\documentclass[a4paper,11pt]{article}
\usepackage[margin=1in]{geometry}

\usepackage{amsthm,amsmath,amssymb,amsfonts}
\usepackage{thmtools}
\usepackage{thm-restate}

\usepackage[bookmarks=true, colorlinks=true, linkcolor=black, citecolor=black, urlcolor=black, plainpages=false, pdfpagelabels,final]{hyperref}
\usepackage{cleveref}
\usepackage[dvipsnames]{xcolor}
\usepackage{colonequals}
\usepackage{microtype} 
\usepackage[backend=biber,style=alphabetic,giveninits=true,doi=false,isbn=false,url=false,sorting=nyt,maxbibnames=99,date=year,maxalphanames=999]{biblatex}
\renewbibmacro{in:}{}
\def \RR {\mathbb{R}}

\def \eps {\varepsilon}
\def \madm {m_{\mathrm{ADM}}}
\def \Wloc {W_{\mathrm{loc}}}

\def \loc {\mathrm{loc}}
\def \Cinfc {C^\infty_{\mathrm{c}}}
\def \D {\mathcal{D}}
\def \Ric {\mathrm{Ric}}

\def \dg {d\mu_{g}}

\def \la {\langle}
\def \ra {\rangle}
\def \p {\partial}
\def \nablab {\overline{\nabla}}
\def \RCD {{\sf RCD}}
\def \Vol {\mathrm{Vol}}

\def\madm{m_{\mathrm{ADM}}}
\def\miso{m_{\mathrm{iso}}}
\def\misoc{m_{\mathrm{iso}}^{\mathrm{c}}}

\DeclareMathOperator \tr {tr}

\newcommand\restr[2]{{
  \left.\kern-\nulldelimiterspace 
  #1 
  \vphantom{\big|} 
  \right|_{#2} 
  }}

\makeatletter
\DeclareFontFamily{OMX}{MnSymbolE}{}
\DeclareSymbolFont{MnLargeSymbols}{OMX}{MnSymbolE}{m}{n}
\SetSymbolFont{MnLargeSymbols}{bold}{OMX}{MnSymbolE}{b}{n}
\DeclareFontShape{OMX}{MnSymbolE}{m}{n}{
    <-6>  MnSymbolE5
   <6-7>  MnSymbolE6
   <7-8>  MnSymbolE7
   <8-9>  MnSymbolE8
   <9-10> MnSymbolE9
  <10-12> MnSymbolE10
  <12->   MnSymbolE12
}{}
\DeclareFontShape{OMX}{MnSymbolE}{b}{n}{
    <-6>  MnSymbolE-Bold5
   <6-7>  MnSymbolE-Bold6
   <7-8>  MnSymbolE-Bold7
   <8-9>  MnSymbolE-Bold8
   <9-10> MnSymbolE-Bold9
  <10-12> MnSymbolE-Bold10
  <12->   MnSymbolE-Bold12
}{}

\let\llangle\@undefined
\let\rrangle\@undefined
\DeclareMathDelimiter{\llangle}{\mathopen}%
                     {MnLargeSymbols}{'164}{MnLargeSymbols}{'164}
\DeclareMathDelimiter{\rrangle}{\mathclose}%
                     {MnLargeSymbols}{'171}{MnLargeSymbols}{'171}
\makeatother

\theoremstyle{plain}
\newtheorem{theorem}{Theorem}[section]
\numberwithin{equation}{section}

\newtheorem{definition}[theorem]{Definition}
\newtheorem{proposition}[theorem]{Proposition}
\newtheorem{corollary}[theorem]{Corollary} 
\newtheorem{lemma}[theorem]{Lemma}
\newtheorem{remark}[theorem]{Remark}

\crefname{proposition}{Proposition}{Propositions}
\crefname{lemma}{Lemma}{Lemmas}
\crefname{remark}{Remark}{Remarks}
\crefname{equation}{Equation}{Equations}

\begin{document}

\title{Coordinate isoperimetric mass and a positive mass theorem for distributional scalar curvature}

\author{
Melanie Graf$^1$\thanks{\href{mailto:melanie.graf@uni-hamburg.de}{\tt melanie.graf@uni-hamburg.de}},\\
Eduardo Hafemann$^1$\thanks{\href{mailto:eduardo.hafemann@uni-hamburg.de}{\tt eduardo.hafemann@uni-hamburg.de}} \\ \\
$^1$Department of Mathematics, Universit\"{a}t Hamburg, \\
20146 Hamburg, Germany.
}

\maketitle
\vspace{-0.7em}
\begin{center}
Abstract
\end{center}
Extending the well-known Riemannian positive mass theorem to low-regularity metrics is an ongoing effort, starting from the early 2000s, that has recently seen renewed interest. Contrary to many other approaches we do not wish to assume smoothness at infinity nor to introduce a new definition of mass tailored to our approach, but rather go back to the coordinate isoperimetric mass originally proposed by Huisken. We establish  new asymptotic expressions for the coordinate isoperimetric mass of continuous asymptotically flat metrics, which constitute the main ingredient in proving a finiteness criterion, coordinate invariance and equality to the generalized ADM mass of Lee and LeFloch under reasonable assumptions. In particular, for $W^{1,p}_{-\tau}$ asymptotically flat metrics, with $p>n$, $\tau>(n-2)/2$, and nonnegative distributional scalar curvature, the coordinate isoperimetric mass is indeed coordinate invariant, and we establish a positive mass theorem in this regularity without any topological assumptions aside from asymptotic flatness. Our proof relies on a density theorem similar to that of Lee, Lesourd, and Unger, but under a weaker regularity and decay assumption, which may be of independent interest.


\vspace{-0.5em}
\addtocontents{toc}{\protect\enlargethispage{2\baselineskip}}
\setcounter{tocdepth}{2}
\tableofcontents
\clearpage

\section{Introduction}

In mathematical general relativity, time-symmetric initial data for an isolated system is modeled by an asymptotically flat Riemannian manifold, that is roughly speaking a Riemannian manifold that outside of some compact set containing all matter sources and any nontrivial topology looks like $\mathbb{R}^n$ outside a ball, with the metric approaching the Euclidean metric at infinity. Asymptotic flatness allows for a mathematically sound definition of physical quantities such as the ADM-mass \cite{ArnowittDeserMisner1961} 
\begin{equation}\label{eq:ADMmass_into_paper}
\madm:=\frac{1}{16\pi } \lim_{r\to \infty}  \int_{S_r} \sum_{i,j=1}^n (\partial_i g_{ij}-\partial_j g_{ii}) \frac{x^j}{|x|}\,dS^0_r,\vspace{-1ex}
\end{equation}
where $dS^0_r$ denotes the Euclidean area measure on $S_r$, the center of mass, and others. Notably the ADM-mass turned out to be finite and independent of the chosen asymptotically flat chart as shown by Bartnik \cite{Bartnik1986}, provided one has sufficient decay of the metric and its derivatives up to second order and the scalar curvature is integrable. It is a powerful tool for using mild global curvature control, together with geometric control available on the asymptotically flat end, to draw strong conclusion about the entire Riemannian manifold $(M, g)$: The Riemannian {\em Positive Mass Theorem} states that if $(M,g)$ is asymptotically flat and complete and has nonnegative scalar curvature, then also $\madm\geq 0$.  Further if, under the same assumptions of completeness and nonnegative scalar curvature $R\geq 0$, $\madm=0$ then $(M,g)$ must already be isometric to Euclidean space.  This theorem was first shown by Schoen and Yau in \cite{PMTSchoenYau} using minimal surface techniques for asymptotically Schwarzschildean Riemannian manifolds in dimension $n=3$, then extended  by them to the general asymptotically flat case and $3\leq n\leq 7$, cf.~\cite{SchoenYau1981}. In case $M$ is spin, the result was extended to all dimensions by Witten \cite{Witten1981} using harmonic spinors The higher dimensional non-spin case has been treated in \cite{Schoen2019,Lohkamp2016,Lohkamp2017}. More recently, new approaches have been developed using generic regularity for minimal surfaces in \cite{ChodoshMantoulidisSchulze2023, ChodoshMantoulidisSchulzeWang2025,BiHaoHeShiZhu26} covering dimensions up to $19$, and a new dimension descent scheme in \cite{BrendleWang2026} for arbitrary dimensions. Physically this geometric result further justifies $\madm$ being the total "mass" of the given isolated system as it guarantees non-negativity of the mass given a non-negative energy density.

The positive mass theorem has been generalized to allow for a limited drop in regularity of the metric. These efforts in lowering the necessary metric regularity have been crucial for the physical applicability of the positive mass theorem and for ensuring that arguments in the proof of, e.g., the Bunting--Masood-ul-Alam  black hole uniqueness theorem \cite{BuntingAlam1987} can be made fully mathematically rigorous. Another motivation comes from Riemannian geometry itself: Gromov \cite{Gromov2014} showed that scalar curvature lower bounds are preserved under $C^0$ limits of $C^2$ metrics, provided the limit is $C^2$ (also see \cite{MazurowskiYao2026QC0} for quantitative estimates). This indicates that scalar curvature lower bounds may be a $C^0$ property of the metric and it is therefore reasonable to ask whether a positive mass theorem can be proved at much lower regularity, perhaps even for $C^0$ metrics which could reveal interesting properties about scalar curvature. Additionally low-regularity versions of positive mass rigidity can be understood as removable singularity results. This is also evident in its compact analog, the classical torus rigidity result, stating that metrics with nonnegative scalar curvature on the torus are isometric to the flat torus metric. Investigating this property for low-regularity metrics also received increased attention recently. Concretely, \cite[Theorem 1.7]{LeeLitzingerSimon2026} established torus rigidity for $L^\infty \cap W^{1,n}_\loc$ metrics using Ricci flow, generalizing \cite{Jiang2022}.

The first low-regularity positive mass theorem for metrics admitting corners along a hypersurface (with rigidity in case $n=3$) goes back to Miao \cite{Miao2002} and was improved by McFeron and Szekelyhidi \cite{McFeron2012-gs} by combining some of Miao's approximation techniques with Ricci flow methods to include the rigidity case in any dimension. More recent results allow for lower regularity but impose restrictions on the dimension of the singular set \cite{Lee2013, Shi2018, Jiang2022}. Without restricting the dimension of the singular set, Grant and Tassotti \cite{grant2014positive} proved a positive mass theorem for $C^0 \cap \Wloc^{2, n/2}$ metrics that are smooth outside a compact set. More recently, in \cite{Hafemann2026}, this result was extended via convolution techniques and Friedrichs' lemma to $C^0 \cap W^{1,n}_\loc$ metrics that are smooth outside a compact set. Using Ricci-flow, Lee-Litzinger-Simon \cite{LeeLitzingerSimon2026} extended it to $L^\infty \cap W^{1,n}_\loc$ metrics that are smooth outside a compact set. All the results mentioned so far are based on reducing the problem to the smooth positive mass theorem by constructing suitable approximations and assume smoothness of the metric outside a compact set with the standard decay assumptions and hence use the classical ADM-mass in their formulation.

For metrics that are not smooth (or at least $C^2$) in the asymptotic region, the literature is considerably more limited. One of the new problems arising when allowing for a drop in regularity also in the asymptotic region, is that one first needs to replace the definition of the ADM-mass by a suitable candidate for the mass. By now there are several such candidates in the literature, two of which will be of particular relevance for us. 

In \cite{LeeLeFloch2015}, Lee and LeFloch proposed the following notion of mass for Sobolev metrics:
\begin{equation}\label{eq:LLmass_intro}
\mathbf{m} \colonequals \frac{1}{2(n-1) \omega_{n-1}} \inf_{\eps>0} \liminf_{r \rightarrow+\infty}\left(\frac{1}{\eps} \int_{r<\rho<r+\eps} \mathbf{V} \cdot \nablab \rho \,d\mu_h\right),
\end{equation}
where $\mathbf{V}$ is a vector field given by \eqref{eq:V_LeeLeFloch}, $\overline{\nabla}$ is the Levi-Civita connection of a fixed smooth background metric $h$, $\omega_{n-1}$ is the volume of the standard unit $(n-1)$-sphere, and $\rho$ is the radial function $\rho(x) = \sqrt{\sum (x^i)^2}$ defined by the asymptotically flat end. One should note that this is very immediately modeled after the ADM-mass, $\mathbf{V}$ is related to the linearization of the scalar curvature at the Euclidean background metric and this mass can be defined for $C^0 \cap W^{1,p}_{-\tau}$ metrics, where $p\geq1$ and $\tau>0$.

A different definition of the mass for asymptotically flat 3-manifolds, based on the isoperimetric inequality and aiming to capture more of a geometric meaning, was put forward earlier by Huisken. In \cite{Hui1}, he suggested a definition based on coordinate spheres in chosen asymptotically flat coordinates 
\begin{equation}\label{eq:intro_misoc}
    \misoc := \lim_{r\to\infty}\frac{2}{\mathcal{A}(r)} \left(V(r) - \frac{1}{6\sqrt{\pi}} \mathcal{A}(r)^{3/2} \right), 
\end{equation}
where $V(r)$ and $\mathcal{A}(r)$ denote, respectively, the volume of the coordinate ball and the area of the coordinate sphere of radius $r$, computed with respect to the metric $g$. This was later replaced by a manifestly coordinate invariant definition in \cite{Hui2,Hui3} (or \cite[Definition 7.1]{EichmairMetzger2013}), namely
\begin{equation}\label{eq:miso_exaust}
  m_{\mathrm{iso}}:= \sup_{\{\Omega_i\}_{i=1}^\infty} \limsup_{i\to \infty } \frac{2}{\mathrm{area}(\partial\Omega_i)} \Big(\mathrm{vol}(\Omega_i)-\frac{\mathrm{area}(\partial\Omega_i)^\frac{3}{2}}{6\sqrt{\pi}} \Big),
\end{equation}
where the supremum is taken over all sequences of sufficiently regular $\{\Omega_i\}$ compactly exhausting $M$. Clearly, $\misoc\leq m_{\mathrm{iso}}$ and both can be naturally generalized to dimension $n \geq 3$ as proposed in \cite[Section 8.3]{Jauregui2023-mh} (see \eqref{eq:misoc_def}). A remarkable feature of both definitions, which was already pointed out by Huisken \cite{Hui1,Hui2,Hui3}, is that they do not rely on any derivatives of the metric, so they naturally extend to $C^0$ metrics (replacing the area with the perimeter measure and adapting the assumptions on the $\Omega_i$, see e.g., \cite{Jauregui2019}) and even to the metric measure space setting (cf.\ also \cite[Sec.\ 5]{JaureguiPeralesPortegies2025} for a suggested definition  in local integral current spaces). Therefore, the isoperimetric mass in particular, given its evident coordinate invariance, seemed to be an interesting candidate for generalizing the positive mass theorem.

Progress towards establishing properties of the isoperimetric mass has been made by several authors: Following Huisken's suggestions, Fan, Shi and Tam, \cite{Fan2009-bl} showed that  $\misoc$ coincides with the ADM mass for $C^2$-asymptotically flat manifolds, based on an argument credited to P.~Miao, and is therefore coordinate invariant for $C^2$ metrics and hence $\madm\leq m_{\mathrm{iso}}$. Jauregui and Lee \cite{Jauregui2019} prove that $m_{\mathrm{iso}}=\madm$ under the assumption of non-negative scalar curvature. In particular the positive mass theorem holds for the isoperimetric mass. However, although several interesting results involving $\miso$ have been established, a recent note showed that, for any $C^0_{-\tau}$ asymptotically flat manifold with decay rate $\tau>0$, the isoperimetric mass $\miso$ is \emph{always} nonnegative, independently of the scalar curvature, rendering a positive mass theorem for $\miso$ vacuously true. We note that this observation does not apply to $\misoc$ nor to a suitable quasilocal version of \eqref{eq:miso_exaust}, the latter of which has recently been shown by Antonelli, Fogagnolo, Nardulli and Pozzeta, \cite{Antonelli2026}, to be nonnegative for $C^0$ Riemannian 3-manifolds that are $C^0_{\loc}$ asymptotic to $\RR^3$ with nonnegative scalar curvature in an approximation sense, using techniques from weak inverse mean curvature flow.

One remaining open question raised in \cite{Jauregui2024-uv} is whether a low-regularity generalization of the positive mass theorem might hold for $\misoc$. Even if such a result holds, one may wonder whether it carries any geometric meaning, given that the definition of $\misoc$  appears to be coordinate dependent (see \eqref{eq:intro_misoc}). Of course, the same could be said about the definition of the ADM-mass, which turned out to be coordinate independent under suitable assumptions, a result that directly carries over to $\misoc$  for $C^2_{-\tau}$ metrics since it coincides with the ADM mass at this regularity. Therefore one might ask whether a similar coordinate independence result for $\misoc$ carries over to lower-regularity. 
The main goal of this work is to examine these last two questions and assess their validity for low-regularity metrics.

Regarding the first question, so far no positive mass theorem has been established directly for $\misoc$ in low-regularity, the current existing results without assuming smoothness at infinity concern only other notions of mass. For instance, with the notion of mass defined in \eqref{eq:LLmass_intro}, Lee and LeFloch \cite{LeeLeFloch2015} generalized Witten's proof of the positive mass theorem to $n$-dimensional asymptotically flat \emph{spin} manifolds whose metrics only have regularity $C^0\cap W^{1,n}_{\loc}$ and decay in the weighted Sobolev space $W^{1,n}_{-\tau}$, where $\tau>(n-2)/2$ (see \Cref{sec:preliminaries} for the definition). Notably, their proof works directly at this regularity and does not rely on approximations or a density theorem to reduce the problem to the smooth positive mass theorem. 

For continuous metrics on $\RR^3$ respectively $\RR^n$ there are two results which have been announced very recently. Both use approximation based approaches and, contrary to our result, the existence of suitable approximations is essentially baked into the assumptions, namely in the definition of non-negative scalar curvature. Of course the upside of this is that it allows one to work with metrics that are merely continuous for which distributional scalar curvature is not available. The first, by Mazurowski and Yao \cite{MazurowskiYao2026PMT}, uses harmonic functions to prove a positive mass theorem for continuous asymptotically flat metrics $g \in C^0_{-\tau}(\RR^3)$ that admit local uniform approximations by smooth metrics with almost nonnegative scalar curvature. Their approach is based on a new proof of the smooth positive mass theorem in three dimensions using Green's functions (cf.\ \cite{Agostiniani2024}) and is restricted to $\RR^3$.  The second by Burkhardt-Guim \cite{burkhardtguim2026} is a positive mass theorem for $C^0_{-\tau}$ asymptotically flat metrics on $\RR^n$ which are sufficiently close to the Euclidean metric {\em globally} with nonnegative scalar curvature in the sense of Ricci flow. Both \cite{MazurowskiYao2026PMT} and \cite{burkhardtguim2026} use their own definition of mass, which is adapted to their setup instead of $\misoc$.\\   

The main result of this paper is a positive mass theorem for $W^{1,p}_{-\tau}$ asymptotically flat metrics on an arbitrary manifold (i.e.\ without assuming any additional dimension or topology restrictions apart from asymptotic flatness), with $\tau>(n-2)/2$ and $p>n$, using the distributional scalar curvature and $\misoc$. In particular, we answer the question concerning $\misoc$ posed in \cite{Jauregui2024-uv} in this setting and drop the hypothesis of smoothness outside a compact set from the main theorem of \cite{Hafemann2026,LeeLitzingerSimon2026} under this regularity assumption.

\begin{theorem}\label{theo:main_intro}
    Let $(M^n, g)$ be a $W^{1,p}_{-\tau}$, $p>n$, asymptotically flat manifold with $\tau>(n-2)/2$. Assume $R[g]\geq0$ in $\D^\prime$ and $R[g]$ is a finite measure outside some compact set. Then $\misoc(g)\geq0$. Moreover, $\misoc(g) = 0 $ if and only if $(M^n, g)$ is isometric to Euclidean space $(\RR^n, \delta)$.
\end{theorem}

Note that by the weighted Sobolev embedding, $W^{1,p}_{-\tau} \subset C^0_{-\tau}$, and $g$ is $C^0_{-\tau}$ asymptotically flat. Further, nonnegative distributions can be identified with nonnegative Radon measures \cite[Theorem 3.18]{Duistermaat2010}, i.e., if $R[g]\geq 0$ in $\D^\prime$ then there exists a Radon measure $dR[g]$ that agrees with $R[g]$. As consequence, the assumption that $R[g]$ is a finite measure outside some compact set just replaces the standard hypothesis of integrability of the scalar curvature to guarantee the finiteness of the mass. The proof of the nonnegativity follows directly from the following density theorem, which generalizes the density theorem of \cite{LLR2022} to lower regularity and a weaker decay rate.

\begin{restatable}{theorem}{densitythm}\label{theo:density}
 Let $(M^n, g)$ be a $W^{1,p}_{-\tau}$ asymptotically flat manifold, $p>n\geq3$, and $\tau \in (\frac{n-2}{2}, n-2)$. Assume $R[g]\geq0$ in $\D^\prime$ and $R[g]$ is a finite measure outside some compact set. Then, for any $\tau^\prime$ satisfying $0<\tau^\prime<\tau$ and $\tau^\prime< 2\tau-1$, there exists a sequence of $C^2$ asymptotically flat metrics $\widetilde{g}_i$ with the following properties:
\begin{enumerate}
        \item $\widetilde{g}_i$ is harmonically flat outside a bounded set in $\mathcal{E}$, that is, $\widetilde{g}_i = u_i^{\frac{4}{n-2}} \delta$ there, and $u_i$ is a Euclidean harmonic function outside this bounded set with expansion
        \begin{equation}\label{eq:harmonic_expansion}
        u_i(x) = 1 + \frac{A_i}{|x|^{n-2}} + O_\infty(|x|^{1-n}).
        \end{equation}
        \item $R[\widetilde{g}_i]\geq0$ and $R[\widetilde{g}_i]\,d\mu_{h} \rightharpoonup dR[g]$ weakly as finite Radon measures.
         \item The ADM mass of $\widetilde{g}_i$ is $2A_i$ and
        \begin{equation}
          \|g -\widetilde{g}_i\|_{W^{1,p}_{-\tau^\prime}}+  \|\partial g -\partial \widetilde{g}_i\|_{L^{2}}  + |\misoc(g)-\madm(\widetilde{g}_i)| \longrightarrow 0 \text{ as } i \to \infty.
        \end{equation}
        \item There exists a sequence $\delta_i \downarrow 0$ such that $g$ and $\widetilde{g}_i$ are $\delta_i$-close as bilinear forms on $M$,
        \begin{equation}
        (1-\delta_i)\widetilde{g}_i \leq g \leq (1+\delta_i) \widetilde{g}_i.  
        \end{equation}
        \end{enumerate}   
In particular, if $n\geq4$, then $\tau^\prime$ can be chosen so that $(n-2)/2<\tau^\prime < \tau$. 
\end{restatable}

The proof of rigidity in \Cref{theo:main_intro} is analogous to the rigidity argument provided in \cite{Hafemann2026} employing volume comparison theory for RCD-spaces (as in \cite{Jiang2022}) to conclude that the manifold is isometric to Euclidean space, but we use our density theorem (or the nonnegativity part of \Cref{theo:main_intro}) to prove Ricci-flatness.

Returning to more general questions about the coordinate isoperimetric mass, we can show that if $g \in C^0_{-\tau} \cap W^{1,2}_{-\tau}$, $p\geq2$ and $\tau>(n-2)/2$, then the coordinate isoperimetric mass $\misoc$ and the generalized ADM-mass of Lee and LeFloch $\mathbf{m}$ agree by \Cref{prop:EquivMasses}. Thus another way to interpret our results is that \Cref{theo:main_intro} drops the spin hypothesis from the positive mass theorem of Lee and LeFloch when $g \in W^{1,p}_{-\tau}$, $p>n$. On the other hand, if $g \in C^0_{-\tau} \cap W^{1,n}_{-\tau}$ and $M$ is spin, then the positive mass theorem for spin manifolds of Lee and LeFloch applies to $\misoc$, since $\misoc = \mathbf{m}$ (see \Cref{cor:miso_PMTspin}). This equality, or rather the new asymptotic expressions for $\misoc$ derived in the process (see \Cref{theo:miso_div}), have also been a crucial ingredient in the proof of the convergence of the masses in the density theorem.
Concerning the coordinate invariance of $\misoc$ we show that,  for $W^{1,p}_{-\tau}$ asymptotically flat metrics, with $\tau > (n-2)/2$ and $p>n$, $\misoc$ is a geometric invariant of the metric if $R[g]$ is a measure outside a compact set. In particular, this holds if $R[g]\geq 0$ in $\D^\prime$ (see \Cref{cor:miso_coordinv}). Geometric invariance can also be proved for metrics that are just $C^0_{-\tau}$, but only under additional reasonable assumptions on the change of coordinates, decay, and the well-definedness of $\misoc$ in different charts (see \Cref{theo:CoordIndep}). 

From all this, we conclude that $\misoc$ admits a positive mass theorem for $W^{1,p}_{-\tau}$ asymptotically flat metrics, $p>n$ and $\tau>(n-2)/2$, and can be considered a meaningful geometric invariant in this regularity, and possibly warrants further attention by the community.\\ 

\noindent \textbf{Use of AI Disclosure.} 
All the proofs and mathematical ideas were developed and written up by the authors, inspired by their own previous works and the references cited throughout the text. Artificial intelligence was used on parts of the manuscript as an auxiliary tool for correcting grammar and spelling, and to improve wording. In addition, it was used to check selected proofs of the finalized manuscript where it pointed out the dimension or decay restriction in \Cref{theo:CoordIndep} due to the translations. The authors take full responsibility for the content of the manuscript.\\

\noindent \textbf{Acknowledgments.}\label{acknowledgments} 
MG and EH acknowledge support from the Deutsche Forschungsgemeinschaft (DFG, German Research Foundation) under Germany’s Excellence Strategy – EXC 2121 "Quantum Universe" – 390833306. MG would like to thank Dan Lee for helpful discussions during the Thematic Program on Nonsmooth Riemannian and Lorentzian Geometry at the Fields Institute in fall 2022.

\subsection{Preliminaries} \label{sec:preliminaries}

Throughout this paper manifolds $M^n$ are assumed to be smooth, Hausdorff, second countable, connected, and $n$-dimensional with $n\geq 3$. Riemannian manifolds are denoted by $(M^n, g)$, and the Lebesgue measure induced by $g$ is denoted by $d\mu_g$. The covariant derivative and Laplacian with respect to $g$ are denoted by $\nabla^g$ and $\Delta^g$, respectively. The Euclidean metric is denoted by $\delta$, and the Euclidean Laplacian is denoted by $\Delta$. In Euclidean space, we denote the coordinate ball by $B_r=\{x\in \RR^n: |x|<r \}$ and the coordinate sphere by $S_r=\{x\in \RR^n: |x|=r \}$.

In what follows, we introduce the usual definitions of asymptotically flat manifolds, weighted Sobolev spaces, ADM mass, and isoperimetric mass.

\begin{definition}\label{def:TopStruc}
A smooth noncompact manifold $M^n$ is said to possess a \textit{topological structure of infinity} $(\mathcal{E}, \Phi)$ if there exists a compact set $K$ and a $C^\infty$ diffeomorphism 
\begin{equation}
\Phi \colon M \setminus K \to \RR^n\setminus \overline{B_{R}},
\end{equation}
where $B_R$ denotes a ball of radius $R>0$ in $\RR^n$, and $\mathcal{E} = M \setminus K$. The coordinates $x^1, \ldots, x^n$ induced by $\Phi$ define a background metric $h=\delta$ on $\mathcal{E}$, which can be extended arbitrarily to a complete metric on $M$. 
\end{definition}
Less regular topological structures of infinity can also be considered. For instance, if $\Phi$ is a $C^k$ diffeomorphism with $k \geq 1$, we can still introduce asymptotic coordinates and discuss $C^0$ metrics in these coordinates. In this case, we refer to $(\mathcal{E},\Phi)$ as a $C^k$ topological structure of infinity.

\begin{definition}
Let $(\mathcal{E}, \Phi)$ be a topological structure of infinity for $M^n$. Given $k\in \mathbb{N}$, $p\ge 1$, and $s\in \RR$, we define the \emph{weighted Sobolev space} $W^{k,p}_s(M)$ to be the space of functions $u\in W^{k,p}_\loc(M)$ with finite norm 
\[\|u\|_{W^{k,p}_s(M)}=\|u\|_{W^{k,p}(M\setminus \mathcal E)}+\sum_{|I|\leq k} \|\partial^I u\|_{L^p_{s-|I|}(\mathcal E)},\]
where $I$ is a multi-index and the weighted $L^p_s$ norm is defined by 
\[\|u\|_{L^p_s(\mathcal E)}=\left(\int_\mathcal E |u|^p|x|^{-sp} \,\frac{dx}{|x|^n}\right)^\frac 1p.\]
Moreover, for $\alpha \in [0,1]$, we define the \emph{weighted H\"older space} $C^{k,\alpha}_{s}(M)$ to be the space of functions $u\in C^{k,\alpha}_\loc(M)$ with finite norm 
$$
\|u\|_{C^{k,\alpha}_s(M)}=\|u\|_{C^{k,\alpha}(M\setminus \mathcal E)}+ \sum_{|I|=k} \sup _{\substack{x, y \in \mathcal{E} \\ 0<|x-y| \leq|x| / 2}}|x|^{\alpha+|I|-s} \frac{\left|\partial^I u(x)-\partial^I u(y)\right|}{|x-y|^\alpha}.
$$
\end{definition}

Note that the norms are constructed relative to the fixed smooth background metric $h$ and the spaces depend on the structure $\Phi$ and $h$. Moreover, denoting by $r(x) = \sqrt{\sum (x^i)^2}$ the radial function induced by $\Phi$ on $\mathcal{E}$, we have $r^s\notin L^p_s$ but $r^{s-\delta}\in L^p_s$ for any $\delta>0$ and $L^p_{s'}\subset L^p_{s}$ if $s'\ge s$, and $L^p = L^p_{-n/p}$.

\begin{definition}\label{def:AFmfd}
    Let $M^n$ be a smooth manifold with a complete Riemannian metric $g \in \Wloc^{k,p}$, $p>n$, $k\geq1$, and a topological structure of infinity $(\mathcal{E}, \Phi)$. Let $\tau>0$, the pair $(M^n, g)$ is said to be a $W^{k,p}_{-\tau}$ asymptotically flat manifold if
    \begin{equation}{\label{eq:AsympCond}}
    (\Phi_*g)_{ij} - \delta_{ij} \in W^{k,p}_{-\tau}(\mathcal{E}),    
    \end{equation}
    and $(\Phi_*g)$ is continuous and uniformly equivalent to $\delta$ on $\mathcal{E}$, i.e., there exists $\Lambda\geq 1$ such that
    \begin{equation}{\label{eq:UniformEquiv}}
        \Lambda^{-1}|X|^2 \leq (\Phi_*g)_{ij} (x) X^iX^j \leq \Lambda |X|^2 \quad \text{for all}  \quad x \in \mathcal{E},\,X^i \in \RR^n.
    \end{equation}
    We also define $C^{k,\alpha}_{-\tau}$ asymptotically flat manifolds, $\alpha \in [0,1]$, in a similar fashion.
\end{definition}

Uniform equivalence of $g$ to $\delta$ \eqref{eq:UniformEquiv} guarantees that the spaces $L^p_s$ are independent of the structure $\Phi$ and background metric $h$ (see proof of \cite[Lemma 1]{Chrusciel1986}, where the corresponding radial functions are compared via uniform equivalence). In contrast, the spaces $W_{s}^{k,p}$, $k\geq1$, defined using partial derivatives with respect to $(x^i)$ will depend on $\Phi$. As discussed in \cite[Section 2]{Bartnik1986}, this is unavoidable since the metric $g$ is not regular enough to define higher covariant derivatives and also because \eqref{eq:AsympCond} is a weak assumption on the derivatives of $\Phi$.

On the other hand, \eqref{eq:AsympCond} and \eqref{eq:UniformEquiv} are sufficiently strong to guarantee that, for a fixed $\Phi$, the weighted Sobolev spaces used in this work (see \Cref{sec:AnalysisConfLap}), namely $W^{1,p}_{-\tau}$, $W^{2,p/2}_{-\tau}$, and $W^{2,p}_{-\tau}$, with $p>n$, can be equivalently defined using the covariant derivatives with respect to $g$ or $h$. The only nontrivial case is the second-order spaces, for which $(\nabla_g^2 u-\nabla_h^2 u)\approx (\Gamma(h)-\Gamma(g)) \partial u$, and the decay assumptions together with the Sobolev embedding provide sufficient control over lower-order terms involving the Christoffel symbols $\Gamma$. Hence, for a fixed $\Phi$, the metrics $g$ and $h$ induce equivalent weighted Sobolev norms for $W^{2,p/2}_{-\tau}$ and $W^{2,p}_{-\tau}$.

If $g_{ij} - \delta_{ij} \in W^{1,p}_{-\tau}$, with $p>n$, then the weighted Morrey inequality implies that $g_{ij}-\delta_{ij} \in C^{0,\alpha}_{-\tau}$, for some $\alpha \in (0,1)$, and that there exists a \textit{pointwise decay constant} $C_g<\infty$ such that
\begin{equation}
|g_{ij}-\delta_{ij}| \leq C_g |x|^{-\tau}  \quad \text{ on }\mathcal{E}.
\end{equation}
Thus $g_{ij} - \delta_{ij} \in W^{1,p}_{-\tau}$ already implies that $g$ is asymptotically flat. 

If $(M^n, g)$ is $W^{2,p}_{-\tau}$ asymptotically flat manifold, with $p>n$ and $\tau>(n-2)/2$, and $R[g] \in L^1(M)$, the \textit{ADM mass of $(M, g)$} (see \cite{ArnowittDeserMisner1961,Bartnik1986}) is defined by
\begin{equation}
\madm(g) = \lim_{r \to \infty} \frac{1}{2(n-1)\omega_{n-1}} \int_{S_r} \sum_{i,j=1}^n (g_{ij,i}-g_{ii,j})\nu_j\,dS,
\end{equation}
where $\omega_{n-1}$ is the area of the standard unit $(n-1)$-sphere in $\RR^n$, $S_r$ is the coordinate sphere in $\mathcal{E}$ of radius $r$, $\nu$ is its outward unit normal, and $dS$ is the Euclidean area element on $S_r$. Note that the limit can also be taken for $C^1$ metrics, but not necessarily well-defined.

In contrast, the isoperimetric mass can be defined for $C^0_{-\tau}$ metrics and a natural generalization of the coordinate isoperimetric mass to dimension $n \geq 3$, proposed in \cite{Jauregui2019}, is seen to be:
\begin{equation}
    \misoc(g)= \lim_{r\to\infty}  \frac{2}{\omega_{n-1}}\left(\frac{\omega_{n-1}}{\mathcal{A}(r)}\right)^{\frac{2}{n-1}}\left(V(r)-\frac{1}{n\left(\omega_{n-1}\right)^{\frac{1}{n-1}}}\mathcal{A}(r)^{\frac{n}{n-1}}\right),
\end{equation}
where $V(r)$ and $\mathcal{A}(r)$ are the volume and area of the coordinate ball, respectively. However, it is not clear which are the minimal conditions one needs to impose on $g$ to guarantee that $\misoc(g)$ is well-defined and finite for merely continuous metrics. In particular, for smooth asymptotically flat metrics, the ADM mass $\madm$ is well-defined and finite if the scalar curvature $R$ is integrable \cite[Proposition 4.1]{Bartnik1986}, and hence $\misoc$ is also well-defined and finite \cite[Corollary 2.3]{Fan2009-bl}. 
Conditions guarantee that $\misoc(g)$ is finite and well-defined are discussed in \Cref{sec:finitenessmass}. 

\section{Coordinate isoperimetric mass} \label{sec:IsoFormulas}

The main goal of this section is to develop a framework for the coordinate isoperimetric mass and to clarify its relationship with the ADM mass in low-regularity. We first derive identities relating the coordinate isoperimetric mass to the ADM mass. We then address its well-definedness and uniqueness, and prove finiteness under conditions on the distributional scalar curvature.

\subsection{Asymptotic expansion}
In this section, we establish several identities for the isoperimetric mass that allow it to be expressed in a form resembling a volumetric version of the ADM mass. Most of the computations follow \cite{Fan2009-bl}, which assumes $n=3$. We restate the relevant proofs and lemmas to emphasize that many of the arguments remain valid in higher dimensions and $C^0_{-\tau}$ metrics. This may prove useful in future attempts to establish a $C^0$ positive mass theorem for the coordinate isoperimetric mass $\misoc$, however, such a result would require new techniques to address the lack of regularity in this setting.

Consider a $C^0_{-\tau}$ asymptotically flat manifold $(M^n,g)$ with coordinates $(x^1,\ldots,x^n)$ on the asymptotically flat end and $\tau>0$. Let $n=n^i\frac{\p}{\p x^i}$ be the unit outward normal (with respect to $g$) of $S_r$  and $n_i=g_{ij}n^j$, where we use Einstein summation convention. Then, we have
\begin{equation}\label{outwardnormal1}
   n^i=\sum_{j=1}^n\frac{g^{ij}x^j}{r|\nabla r|}\ \text{and}\
   n_i=\frac{x^i}{r|\nabla r|},
\end{equation}
where $r=(\sum_{i=1}^n(x^i)^2)^\frac12$. The metric
induced on $S_r$ is $\gamma_{ij}=g_{ij}-n_in_j$ and we define $\sigma_{ij} \colonequals g_{ij}-\delta_{ij}$. We also denote by $\misoc(r)$ as the isoperimetric mass function that is given by
\begin{equation}\label{eq:misoc_def}
        \misoc(r)= \frac{2}{\omega_{n-1}}\left(\frac{\omega_{n-1}}{\mathcal{A}(r)}\right)^{\frac{2}{n-1}}\left(V(r)-\frac{1}{n\left(\omega_{n-1}\right)^{\frac{1}{n-1}}}\mathcal{A}(r)^{\frac{n}{n-1}}\right),
\end{equation}
where $\omega_{n-1}$ is the volume of the $(n-1)$-sphere and $V(r)$ and $\mathcal{A}(r)$ are the volume and area of the coordinate ball, respectively. In particular $\misoc = \lim_{r \to \infty}\misoc(r)$.

We start by computing some relevant geometric quantities for metrics in $C^0_{-\tau}$. The following lemma can be found in \cite[Lemma 2.1.]{Fan2009-bl}.

\begin{lemma}\label{2nd}
Let $(M^n, g)$ be a $C^0_{-\tau}$ asymptotically flat manifold, $\tau>0$. With the above notations, on $S_r$ we have that
\begin{equation}\label{2nd-eq0}
|\nabla r|^2=1-\frac{\sigma_{ij}x^ix^j}{r^2}+O(r^{-2\tau}),
\end{equation}
and
\begin{equation}\label{eq:dSr_dSr0}
dS_r=(1+ \sum_{i,j=1}^n \gamma_{ij}\sigma_{ij}+O(r^{-2\tau}))^\frac12dS_r^0,   
\end{equation}
where $dS_r^0$ is the Euclidean induced volume form. Hence
\begin{equation}\label{eq:Ar_Ar0}
\mathcal{A}(r)=\omega_{n-1}r^{n-1}+\frac12 \sum_{i,j=1}^n \int_{S_r} \gamma_{ij}\sigma_{ij}
dS_r+O(r^{(n-1)-2\tau}).
\end{equation}
\end{lemma}
\begin{proof}
Recall that the components of $\nabla r$ are given by $(\nabla r)^i = \sum_{k=1}^n g^{ik} x^k/r$ and the inverse components of the metric can be approximated by $g^{ij} = \delta_{ij} - \sigma_{ij} + O(r^{-2\tau})$. Combining both identities, we have
\begin{equation}
(\nabla  r)^i = g^{ik} \partial_k r = (\delta_{ik}-\sigma_{ik}) \frac{x^k}{r}  +  O(r^{-2\tau}).    
\end{equation}
A simple computation then yields
\begin{equation}
|\nabla r|^2 = g^{ij} \partial_ir \partial_jr = 1  - \sigma_{ij} \frac{x^ix^j}{r^2} + O(r^{-2\tau}),
\end{equation}
which proves \eqref{2nd-eq0}. In particular, we have
\begin{equation}\label{eq:normal_approx}
n^i=\sum_{j=1}^n\frac{g^{ij}x^j}{r|\nabla r|} = \frac{x^i}{r} + O(r^{-\tau}),    
\end{equation}
and
\begin{equation}\label{eq:normal_approx2}
n_i=\frac{x^i}{r|\nabla r|} = \frac{x^i}{r} + O(r^{-\tau}).    
\end{equation}
Next, let $e_1,\ldots,e_{n-1}$ be an orthonormal frame on $S_r$ with respect to the Euclidean metric. Then
\begin{align*}
    \,dS_r & = \sqrt{\det (g(e_k,e_\ell))} \,dS^0_r.
\end{align*}
Since $g(e_k,e_\ell)= \delta_{k\ell} + \sigma(e_k,e_\ell)$ and $\sigma = O(r^{-\tau})$, we may approximate the determinant by
\begin{equation}
    \det (g(e_k,e_\ell)) = 1 + \sum_{k=1}^{n-1} \sigma(e_k, e_k) + O(r^{-2\tau}).
\end{equation}
Observing that $\sigma(e_k, e_k) = e_k(x^i)e_k(x^j) \sigma_{ij}$, using the identity $ \delta^{ij} - x^ix^j/r^2 = \sum_k^{n-1} e_k(x^i)e_k(x^j)$ and \eqref{eq:normal_approx}, we obtain
\begin{align}
    \,dS_r 
    & = \left(1+ \sum_{i,j=1}^n \gamma_{ij}  \sigma_{ij} + O(r^{-2\tau})\right)^{1/2} \,dS^0_r,
\end{align}
which proves \eqref{eq:dSr_dSr0}.

Finally, integrating over $S_r$ and applying the expansion $(1+x)^{1/2} = 1 + \frac{1}{2} x + O(x^2)$, we deduce
\begin{align}
\mathcal{A}(r)=\omega_{n-1}r^{n-1}+\frac12\sum_{i,j=1}^n\int_{S_r}\gamma_{ij}\sigma_{ij}
dS_r+O(r^{(n-1)-2\tau}).   
\end{align}
\end{proof}

The following theorem is based on the computations in Theorem 2.2 and Corollary 2.3 of \cite{Fan2009-bl}.

\begin{theorem}\label{theo:m_iso_tr}
Let $(M^n, g)$ be a $C^0_{-\tau}$ asymptotically flat manifold, $\tau>(n-2)/2$. Then, it holds that
    \begin{equation}\label{eq:tr_miso}
        \misoc(r)  = -r^{n-2} \mathcal{I}_A + \frac{1}{\omega_{n-1} r^2} \int_{B_r\setminus B_R} \tr_\delta(\sigma) \,d\mu_g +o(1),
    \end{equation}
where
\begin{equation}\label{eq:Ia_def}
\mathcal{I}_A = \frac{1}{(n-1)\omega_{n-1} r^{n-1}}\int_{S_r} \left(\delta^{ij} - \frac{x^ix^j}{|x|^2}\right)\sigma_{ij} \,dS_r.    
\end{equation}
\end{theorem}
\begin{proof}
By the coarea formula and \eqref{2nd-eq0}, we have
\begin{align}
    V^\prime(r)  = \int_{S_r} \frac{1}{|\nabla r|} \,dS_r
     = \mathcal{A}(r) + \frac12 \int_{S_r} \frac{\sigma_{ij}x^ix^j}{|x|^2}dS_r + O(r^{(n-1)-2\tau}).
\end{align}
Using \eqref{eq:Ar_Ar0} and expressing $\gamma_{ij}\sigma_{ij}$ in terms of $x^i$ via \eqref{eq:normal_approx2}, we obtain
\begin{equation}\label{eq:Ar_type2}
\mathcal{A}(r) = \omega_{n-1} r^{n-1} + \frac12\int_{S_r} (\delta^{ij} - \frac{x^ix^j}{|x|^2})\sigma_{ij} \,dS_r + O(r^{(n-1)-2\tau}),   
\end{equation}
Integrating $V^\prime(r)$ and combining with \eqref{eq:Ar_type2}, we have
\begin{align*}
    V(r)&  = \int^r_{R} \mathcal{A}(s)ds + \frac12 \int_{B_r\setminus B_R} \frac{\sigma_{ij}x^ix^j}{|x|^2} \,d\mu_g +o(r^2)\\
    &  = \left( \frac{\omega_{n-1}}{n} r^n + \frac12\int_{B_r\setminus B_R} (\delta^{ij} - \frac{x^ix^j}{|x|^2})\sigma_{ij} \,d\mu_g \right)+ \frac12 \int_{B_r\setminus B_R} \frac{\sigma_{ij}x^ix^j}{|x|^2} \,d\mu_g + o(r^2)\\
    &  = \frac{\omega_{n-1}}{n} r^n + \frac12\int_{B_r\setminus B_R} \tr_\delta(\sigma) \,d\mu_g + o(r^2),
\end{align*}
where we have absorbed the lower order term $R^n$ into $o(r^2)$.

We now compute $\misoc(r)$. We start by considering the following simplification
\begin{align}
\misoc(r) & = \frac{2}{\omega_{n-1}}\left(\frac{\omega_{n-1}}{\mathcal{A}(r)}\right)^{\frac{2}{n-1}}\left(V(r)-\frac{1}{n\left(\omega_{n-1}\right)^{\frac{1}{n-1}}}\mathcal{A}(r)^{\frac{n}{n-1}}\right)\\
& =  \frac{2 V(r)}{\omega_{n-1}}\left(\frac{\omega_{n-1}}{\mathcal{A}(r)}\right)^{\frac{2}{n-1}} - \frac{2}{n}\left(\frac{\mathcal{A}(r)}{\omega_{n-1}}\right)^{\frac{n-2}{n-1}}.
\end{align}
From the definition of $\mathcal{I}_A$ \eqref{eq:Ia_def} and \eqref{eq:Ar_type2}, we expand
\begin{equation}
\mathcal{A}(r) = \omega_{n-1} r^{n-1} +\frac{1}{2} (n-1) \omega_{n-1} r^{n-1} \mathcal{I}_A + O(r^{(n-1)-2\tau}).
\end{equation}
Hence,
\begin{align}
\frac{2}{n}\left(\frac{\mathcal{A}(r)}{\omega_{n-1}}\right)^{\frac{n-2}{n-1}}
  &  = \frac{2}{n} r^{n-2} \left( 1 + \frac{(n-1)}{2} \mathcal{I}_A + O(r^{-2\tau})\right)^\frac{n-2}{n-1} \\
    &  =  \frac{2}{n} r^{n-2} \left( 1 + \frac{(n-2)}{2} \mathcal{I}_A \right) + O(r^{(n-2)-2\tau}).
\end{align}
Similarly, we have
\begin{align}
\left(\frac{\omega_{n-1}}{\mathcal{A}(r)}\right)^{\frac{2}{n-1}} & = \frac{1}{r^2} \left(\frac{1}{1 +\frac{1}{2} (n-1) \mathcal{I}_A + O(r^{-2\tau})}\right)^{\frac{2}{n-1}} \\
    &  = \frac{1}{r^2} \left(1 - \mathcal{I}_A\right) + O(r^{-2-2\tau}).
\end{align}
Using the above approximation, the identity for $V(r)$, and that $\mathcal{I}_A = O(r^{-\tau})$, we obtain
\begin{align}
\frac{2 V(r)}{\omega_{n-1}}\left(\frac{\omega_{n-1}}{\mathcal{A}(r)}\right)^{\frac{2}{n-1}}
& = \frac{2 V(r)}{\omega_{n-1}} \frac{1}{r^2}  \left( 1 - \mathcal{I}_A\right) + o(1)\\
& = \frac{2}{n}r^{n-2}\left( 1 - \mathcal{I}_A\right) + \frac{1}{\omega_{n-1} r^2}\int_{B_r\setminus B_R} \tr_\delta(\sigma) \,d\mu_g  + o(1).
\end{align}
Subtracting the two expressions yields
\begin{align}
\misoc(r)
    & =  -r^{n-2}\mathcal{I}_A  + \frac{1}{\omega_{n-1} r^2}\int_{B_r\setminus B_R} \tr_\delta(\sigma) \,d\mu_g + o(1).
\end{align}
\end{proof}

Assuming extra regularity, we can obtain a convenient expression that resembles the ADM mass formula. Note that we only need additional regularity but not any additional decay for derivatives of the metric.

\begin{theorem}\label{theo:miso_div}
Let $(M^n, g)$ be a $C^0_{-\tau}\cap \Wloc^{1,p}$ asymptotically flat manifold, $\tau>(n-2)/2$ and $p\geq1$. Then, it holds that 
\begin{equation}
\misoc(r)  = \frac{1}{(n-1) \omega_{n-1} r^2} \left(\int_{B_r\setminus B_R} (\operatorname{div}_\delta(\sigma) - d\operatorname{tr}_\delta(\sigma))(\rho \partial_\rho) \,d\mu_\delta \right) + o(1),   
\end{equation} 
where $\partial_\rho$ is the radial vector and $d\mu_\delta$ denotes the Euclidean measure.
\end{theorem}
\begin{proof}
Rearranging \eqref{eq:tr_miso} and applying \eqref{eq:dSr_dSr0}, we obtain the following
\begin{align*}
\misoc (r) &  = -r^{n-2} \mathcal{I}_A + \frac{1}{\omega_{n-1} r^2} \int_{B_r\setminus B_R} \tr_\delta(\sigma) \,d\mu_g + o(1)\\
&  = -\left( \frac{1}{(n-1)\omega_{n-1}r}\int_{S_r} \left(\delta^{ij} - \frac{x^ix^j}{|x|^2}\right)\sigma_{ij} \,dS_r  \right)+ \frac{1}{\omega_{n-1} r^2} \int_{B_r\setminus B_R} \tr_\delta(\sigma)  \,d\mu_g + o(1)\\
&  = \frac{1}{(n-1) \omega_{n-1} r^2} \left( -r \left(\int_{S_r} \tr_\delta(\sigma) - \frac{x^ix^j}{r^2}\sigma_{ij}dS_r^0\right) +(n-1) \int_{B_r\setminus B_R} \tr_\delta(\sigma) \,d\mu_\delta\right) + o(1).
\end{align*}

Define $F(r) = \int_{S_r} \tr_\delta(\sigma) \,dS_1^0$, and recall that $\tr_\delta(\sigma) \in C^0_{-\tau}\cap\Wloc^{1,1}(M)$ and $dS_1^0$ is the surface measure of the unit ball. Since $r \mapsto r^n F(r)$ is absolutely continuous, $r\geq R$, by the fundamental theorem of calculus and product rule, we have 
$$
r^n F(r) - R^nF(R)= \int_R^r \partial_\rho \left( \rho^n F(\rho)\right) \,d\rho = \int_R^r  \left( n \rho^{n-1} F(\rho) + \rho^n \partial_\rho F(\rho) \right) \,d\rho ,
$$
in particular, substituting the definition of $F(r)$, we obtain
\begin{align*}
r^n F(r) - R^nF(R) &= \int_R^r  \left( n \rho^{n-1}  \int_{S_\rho} \tr_\delta(\sigma) \,dS_1^0 + \rho^n \partial_\rho\left(  \int_{S_\rho} \tr_\delta(\sigma) \,dS_1^0\right) \,d\rho\right)\\
&= \int_R^r  \left( n \rho^{n-1}  \int_{S_\rho} \tr_\delta(\sigma) \,dS_1^0 + \rho^n \int_{S_\rho} \partial_\rho\tr_\delta(\sigma) \,dS_1^0 \,d\rho\right)\\
& = \int_R^r\int_{S_\rho}  n \tr_\delta(\sigma)  + \rho \partial_\rho\tr_\delta(\sigma) \,dS^0_\rho \,d\rho.
\end{align*}
The last equality together yields
\begin{align}
  r \int_{S_r} \tr_\delta(\sigma) \,dS_r^0 = r^n F(r) =  \int_{B_r \setminus B_R} n \tr_\delta(\sigma) + \rho \partial_\rho(\tr_\delta(\sigma))d\mu_\delta + R^nF(R).
\end{align}
Consequently,
\begin{align*}
\misoc(r)
&  = \frac{1}{(n-1) \omega_{n-1} r^2} \left( \int_{S_r} \frac{x^ix^j}{r}\sigma_{ij}dS_r^0 -\int_{B_r\setminus B_R}  \tr_\delta\sigma + \rho \partial_\rho\tr_\delta\sigma \,d\mu_\delta  \right) + o(1)\\
&  = \frac{1}{(n-1) \omega_{n-1} r^2} \left( \int_{S_r} \frac{x^ix^j}{r}\sigma_{ij}dS_r^0 - \int_{B_r\setminus B_R}   \partial_\rho(\rho \tr_\delta\sigma) \,d\mu_\delta \right) + o(1).
\end{align*}
Consider the vector field $X$ with components $X^i = x^j\sigma_{ij} \partial_i$, since $\sigma \in C^0_{-\tau}\cap W^{1,1}_\loc$, the divergence theorem shows that
\begin{equation}
    \int_{B_r\setminus B_R}  \operatorname{div}_\delta  (X) \,d\mu_\delta= \int_{S_r} \frac{x^ix^j}{r}\sigma_{ij}dS_r^0 - \int_{S_R} \frac{x^ix^j}{R}\sigma_{ij}dS_R^0,
\end{equation}
which yields
\begin{align}\label{eq:miso_radial}
\misoc(r) &  = \frac{1}{(n-1) \omega_{n-1} r^2} \left(\int_{B_r\setminus B_R}  \operatorname{div}_\delta  (X) - \partial_\rho(\rho \tr_\delta(\sigma)) \,d\mu_\delta \right) + o(1).
\end{align}
The integrand can be further simplified:
\begin{align*}
 \operatorname{div}_\delta  (X) - \partial_\rho(\rho \tr_\delta(\sigma)) &  = \partial_i(x^j \sigma_{ij})- ( \rho \partial_\rho \tr_\delta(\sigma) + \tr_\delta(\sigma))\\
 & = \delta^{ij} \sigma_{ij} + x^j \partial_i \sigma_{ij} - ( \rho \partial_\rho \tr_\delta(\sigma) + \tr_\delta(\sigma))\\
 & =  x^j \partial_i \sigma_{ij} - \rho \partial_\rho \tr_\delta(\sigma) \\
 & =  x^j \partial_i \sigma_{ij}- x^j \partial_j\sigma_{ii}.
\end{align*}
Finally, recalling that $\operatorname{div}_\delta(\sigma) = (\partial_i \sigma_{ij}) dx^j$ and $d\operatorname{tr}_\delta(\sigma) = (\partial_j\sigma_{ii}) dx^j$, we conclude that
\begin{align*}
\misoc(r) & = \frac{1}{(n-1) \omega_{n-1} r^2} \left(\int_{B_r\setminus B_R} (\operatorname{div}_\delta(\sigma) - d\operatorname{tr}_\delta(\sigma))(\rho \partial_\rho) \,d\mu_\delta \right) + o(1).
\end{align*}
\end{proof}

\begin{corollary}\label{cor:proof_ADMmass}
    Let $(M^n, g)$ be a $C^0_{-\tau}$ asymptotically flat manifold such that $g \in C^1$, $\tau>(n-2)/2$. Assume that $\madm(g)$ exists, then $\misoc$ exists and 
    $$\misoc =  \lim_{r\rightarrow \infty} \frac{1}{(n-1)\omega_{n-1} r^{2}} \left(\int_{B_r\setminus B_R} (\operatorname{div}_\delta(\sigma) - d\operatorname{tr}_\delta(\sigma))(\rho \partial \rho) \,d\mu_\delta \right) =  \madm.$$
\end{corollary}
\begin{proof}
Since the metric is $C^1$, we define the ADM mass function $m(r)$ as
\begin{equation}
    m(r)\colonequals \frac{1}{2(n-1) \omega_{n-1}} \int_{S_r}\left( \operatorname{div}_\delta(\sigma) - d\tr_\delta \sigma\right) (\nu) \,dS_r^0,
\end{equation}
where $\nu = x^j/|x| \partial_j$. Note that
\begin{align}
r m(r) &=     \frac{1}{2(n-1) \omega_{n-1}}\int_{S_r}(\operatorname{div}_\delta(\sigma) - d\operatorname{tr}_\delta(\sigma))(r \partial_r) \,dS_r^0,
\end{align}
where $\p_r = x^j \partial_j$. From the definition of the ADM mass, we have $\madm = \lim_{r\rightarrow\infty} m(r)$ is well-defined. \Cref{theo:miso_div} and these facts imply that
\begin{align*}
\frac{1}{(n-1) \omega_{n-1} r^2} \left(\int_{B_r\setminus B_R} (\operatorname{div}_\delta(\sigma) - d\operatorname{tr}_\delta(\sigma))(\rho \partial_\rho) \,d\mu_\delta \right)  & = \frac{2}{r^2} \int_R^r \rho m(\rho) \,d\rho +o(1).
\end{align*}
If $\madm$ is infinite, then $\misoc$ is also infinite. We therefore assume that $\madm$ is finite, and we have
\begin{align*}
\frac{2}{r^2} \int_R^r \rho m(\rho) & = \frac{2}{r^2} \int_R^r \rho (m(\rho)-\madm) + \rho \madm \,d\rho + o(1) \\
 & = \left(1-\frac{R^2}{r^2}\right)\madm + \frac{2}{r^2}\int_R^r \rho (m(\rho)-\madm)\,d\rho + o(1).    
\end{align*}
Taking the limit as $r \rightarrow \infty$, we conclude that $\misoc = \madm$.
\end{proof}

\subsection{The uniqueness of mass}

The coordinate isoperimetric mass $\misoc$ is defined in terms of the coordinates balls, and hence it is not a priori clear that it is a coordinate invariant quantity. One might therefore expect this definition to depend on the choice of coordinates, even after restricting to an appropriate class of coordinate system, especially for continuous metrics. In this section, we show that within a suitable class of coordinate systems, it is in fact invariant under change of coordinates.

\begin{definition}\label{def:rigid_motion}
Let $(\Phi, x)$ be a structure of infinity on an end $\mathcal{E}_1$ for $(M^n, g)$, where $g$ is a $C^0_{-\tau_1}$ asymptotically flat metric with decay rate $\tau_1>0$. We define $\mathcal{A}_{\Phi}$ as the set of structures of infinity $(\Psi, y)$ on ends $(\mathcal{E}_2)$ such that $g$ is a $C^0_{-\tau_2}$ asymptotically flat metric with decay rate $\tau_2>0$, and the transition map 
$$(\Psi \circ \Phi^{-1}) \colon \RR^n \setminus \overline{B_R} \to \RR^n \setminus \overline{B_R}$$
satisfies, for $\tau=\min{\{\tau_1,\tau_2\}}$,
\begin{equation}
y^i(x)=O^i_j x^j+a^i+\zeta^i(x), \quad (O^i_j,a^i) \in O(n, \RR) \times \RR^n,    
\end{equation}
where the remainder term $\zeta^i$ satisfies
\begin{equation}
\zeta^i \in C^1_{1-\tau}(\Phi, \mathcal{E}_1), \quad \text{ for } i=1,\ldots ,n.    
\end{equation}
We further assume that $(\Psi \circ \Phi^{-1})$ is a $C^1$ diffeomorphism near infinity. We say that $\Psi$ differs from $\Phi$ by a \textit{rigid motion plus lower-order terms} if $\Psi \in \mathcal{A}_\Phi$.
\end{definition}

This definition is inspired by a result of \textcite[Corollary 3.2]{Bartnik1986}, which shows that any two structures at infinity of $(M, g)$ that $g$ is $W^{1,p}_{-\tau_1}$ and $W^{1,p}_{-\tau_2}$ asymptotically flat with decay $\tau_1$, $\tau_2>(n-2)/2$, respectively, and $p>n$, differ by a rigid motion plus lower order terms and $\zeta^i \in W^{2,p}_{1-\tau}(\Phi, \mathcal{E}_1)$. The proof relies on the construction of harmonic coordinates for the metric $g$ and therefore requires higher regularity of the metric.

In what follows, we prove that the coordinate isoperimetric mass $\misoc$ is invariant under changes of coordinates that are rigid motion plus lower-order term, provided that $\misoc$ is well defined in both coordinate systems and the decay rate satisfies $\tau>(n-3)$. For $n=3$ and $n=4$, this result can be understood as an analogue of the coordinate invariance of $\madm$ \cite[Theorem 4.2]{Bartnik1986} at lower regularity and without assuming integrability of the scalar curvature, provided $\misoc$ is well-defined in both coordinate systems. The condition $\tau>n-3$ arises in controlling the error caused by changing the family of coordinate spheres. A similar issue arises for the ADM mass in \cite{Bartnik1986}, where it is addressed by assuming integrability of the scalar curvature. At the regularity considered here, however, the scalar curvature need not be defined even as a distribution.

Interestingly, the isoperimetric mass and \Cref{theo:m_iso_tr} allow for an argument requiring only $C^1_{1-\tau}$ regularity of $\eta$, rather than the $W^{2,p}_{1-\tau}$ regularity used in \cite[Theorem 4.2]{Bartnik1986}, and no additional assumptions on the scalar curvature when $n=3$ or $n=4$. In particular, this invariance result also applies to $\madm$, provided that the metric satisfies the conditions stated in \Cref{cor:proof_ADMmass}. A criterion for the finiteness of $\misoc$ using the distributional scalar curvature is given in the next section (see \Cref{cor:miso_coordinv}), which also yields coordinate invariance without the additional assumption $\tau>n-3$ if $n\geq 5$.

\begin{theorem}[Coordinate-independence of the mass]\label{theo:CoordIndep}
Let $(\Phi, x)$ be a structure of infinity for $(M^n, g)$ such that $g$ is a $C^0_{-\tau_1}$ asymptotically flat metric with decay $\tau_1>(n-2)/2$. Let $(\Psi, y) \in \mathcal{A}_\Phi$ be another structure of infinity that differs from $\Phi$ by a rigid motion plus lower order terms with decay $\tau_2>(n-2)/2$ (see \Cref{def:rigid_motion}). For $\Phi$ (and $\Psi$), we define
\begin{equation}\label{eq:admiso_phi}
\misoc(g,\Phi) = \lim_{r \rightarrow \infty}\frac{2}{\omega_{n-1}}\left(\frac{\omega_{n-1}}{\mathcal{A}^\Phi(r)}\right)^{\frac{2}{n-1}}\left(V^\Phi(r)-\frac{1}{n\left(\omega_{n-1}\right)^{\frac{1}{n-1}}}\mathcal{A}^\Phi(r)^{\frac{n}{n-1}}\right),    
\end{equation}
where $V^{\Phi}(r)$ and $\mathcal{A}^{\Phi}(r)$ denote the volume and area of the coordinate ball of radius $r$ with respect to $g$ defined in the coordinate system $\Phi$. Assuming that $\misoc(g,\Phi)$ and $\misoc(g,\Psi)$ are well-defined and $\tau_1,\tau_2>n-3$, we have
$$\misoc(g,\Phi) = \misoc(g,\Psi).$$ 
In particular, there is no extra decay assumption if $n=3$ or $n=4$.
\end{theorem}
\begin{proof}   
By \Cref{def:rigid_motion}, we can write $y^i$ as
$$
y^i(x)=O^i_j x^j+a^i+\zeta^i(x), \quad (O^i_j,a^i) \in O(n, \RR) \times \RR^n,
$$
where $\tau=\min{\{\tau_1,\tau_2\}}>(n-2)/2$ and 
$\zeta^i \in C^{1}_{1-\tau}(\Phi, \mathcal{E}_1)$. In particular, the inverse has the expansion
$$
x^i(y)=(O^{-1})^i_j (y^j-a^j)-\eta^i(y),
$$
where $\eta^i(y) \in C^{1}_{1-\tau}(\Psi, \mathcal{E}_2)$.

The volume and area of coordinate spheres do not change under rigid coordinate rotation. Consequently, replacing the coordinate $y^i$ by $(O^{-1})^i_jy^j$ preserves the mass. In order to proceed with the proof, we may assume without loss of generality that the matrix $O$ is the identity. The metric transformations are given by
\begin{equation}
g_{ij}^\Psi(y)  =g_{k \ell}^\Phi(x(y)) \frac{\partial x^k}{\partial y^i} \frac{\partial x^{\ell}}{\partial y^j},    
\end{equation}
where
\begin{equation}
\frac{\partial x^{\ell}}{\partial y^j}     = \delta^\ell_j - \partial_j \eta^\ell(y),    
\end{equation}
and $\partial_j \eta^\ell \in  C^{0}_{-\tau}(\Psi, \mathcal{E}_2)$. Combining the decay rate of $\partial_j \eta^\ell$ with 
\begin{equation}\label{eq:r_estimate}
r(y) = r(x(y)) + O(|a|) +  O(r(x(y))^{1-\tau}),
\end{equation}
and the decay rate of $g^\Phi(x(y))$, we have
$$
\sigma^\Psi_{ij} = g^\Psi_{ij} - \delta_{ij} = \sigma_{ij}^\Phi(x(y)) - \partial_i \eta^j(y) - \partial_j\eta^i(y) + O(r(y)^{-2\tau}), 
$$
where $\sigma_{ij}^\Phi(x(y)) = g_{ij}^\Phi - \delta_{ij}$.

We now show that the linear terms in $\eta$ contribute only a vanishing error to the mass. This is due to an interesting cancellation, which also appears under stronger regularity assumptions in Bartnik \cite[Theorem 4.2]{Bartnik1986}, where it is described as a "curious" cancellation. Because coordinate balls depend on the chosen chart, we denote by $B^\Psi_r$ and $S^\Psi_r$ the coordinate ball and sphere of radius $r$ defined with respect to the chart $\Psi$. Analogously, we define $S^\Phi_r$ and $B^\Phi_r$ using the chart $\Phi$. Computing the trace of $\sigma^\Psi_{ij}-\sigma^\Phi_{ij}$ and applying the divergence theorem, we have
\begin{align}
     \frac{(n-1)}{r^2}\int_{B_r^\Psi\setminus B_R} \tr_\delta\left(\sigma^\Psi-\sigma^\Phi\right) \,d\mu_\delta^\Psi & =-(n-1)\frac{2}{r^2}\int_{B_r^\Psi\setminus B_R^\Psi} \operatorname{div}_\delta(\eta)d\mu_\delta  + o(1)\\
     &=-(n-1)\frac{2}{r^2}\int_{S_r^\Psi} \eta \cdot \nu_\Psi \,dS_r^{0}  + o(1),
\end{align}
where $\nu_\Psi = y^i/r\partial_i$ is the outward normal vector of $S_r^\Psi$ and $o(1)$ absorbed the boundary term. 

Similarly, for the $\mathcal{I}_A$ term in \Cref{theo:m_iso_tr}, we can split $\eta = \eta^\top + \eta^\perp$, where $\eta^\top$ and $\eta^\perp$ are the tangent and normal parts to the sphere $S_r^\Psi$ and $\eta^\perp = (\eta \cdot \nu_\Psi) \nu_\Psi$. Let $e_1,\ldots, e_{n-1}$ be an orthonormal frame on $S_r^\Psi$ with respect to the Euclidean metric, then
\begin{align*}
     -\frac{1}{r}\int_{S_r^\Psi} & \left(\delta^{ij}-\frac{y^iy^j}{r(y)^2}\right) \left(\sigma^\Psi_{ij}-\sigma^\Phi_{ij}\right)  \,dS_r^0  \\
     & = \frac{2}{r}\int_{S_r^\Psi} \left(\delta^{ij}-\frac{y^iy^j}{r(y)^2}\right) \partial_i \eta^j \,dS_r^0 + O(r^{(n-2)-2\tau})\\
    & = \frac{2}{r}\int_{S_r^\Psi} \sum_{k=1}^{n-1} e_k(y^i)e_k(y^j) \partial_i \eta^j \,dS_r^0 + O(r^{(n-2)-2\tau})\\
      & = \frac{2}{ r}\int_{S_r^\Psi} \left(\operatorname{div}_{S_r}(\eta^\top)  + \frac{(n-1)}{r} (\eta \cdot \nu_\Psi)\right) \,dS_r^0 +O(r^{(n-2)-2\tau})\\
      & =  (n-1)\frac{2}{r^2} \int_{S_r^\Psi} (\eta \cdot \nu_\Psi) \,dS_r^0 +O(r^{(n-2)-2\tau}).
\end{align*}
Combining these equations, we find that 
\begin{align*}
    \frac{1}{r^2} \left( -r \int_{S_r^\Psi} \left(\delta^{ij}-\frac{y^iy^j}{r(y)^2}\right)\partial_i \eta^j  \,dS_r^0 + (n-1) \int_{B_r^\Psi\setminus B_R^\Psi} \operatorname{div}(\eta) \,d\mu_\delta \right) & = O(r^{(n-2)-2\tau}) + o(1).
\end{align*}
Then, if $\tau>(n-2)/2$ and $n\geq3$, we have that $(n-1) \omega_{n-1} \misoc(g,\Psi)(r)$ can be written as
\begin{equation}\label{eq:approx_aftercancelation}
\left( -\frac{1}{r} \int_{S_r^\Psi} \left(\delta^{ij} - \frac{y^i y^j}{r^2}\right)\sigma_{ij}^\Phi(x(y))\,dS_r^{0} + \frac{(n-1)}{r^2}\int_{B_r^\Psi\setminus B_R^\Psi} \tr_\delta \sigma^\Phi(x(y)) \,d\mu_\delta\right) + o(1).    
\end{equation}
Assuming that $O$ is the identity matrix, we have $y^i = x^i(y) + a^i + \eta^i(y)$ with $\eta^i = O(r^{1-\tau})$. From this and \eqref{eq:r_estimate}, we have
\begin{align}
\frac{y^iy^j}{r(y)^2}
& = \frac{x(y)^ix(y)^j}{r(x(y))^2}+ O(|a| r(x(y))^{-1}) + O(r(x(y))^{-\tau}),    
\end{align}
where $O(|a| r(x(y))^{-1})$ is a term that cannot be absorbed if $\tau>1$. Note that
\begin{equation}
    \frac{1}{r}\int_{S_r^\Psi} O(|a| r(x(y))^{-1}) \sigma_{ij}^\Phi(x(y))\,dS_r^{0} =  O(|a|r^{n-3-\tau}),
\end{equation}
and $O(|a|r^{n-3-\tau}) = o(1)$ if $\tau>n-3$. This condition is trivially satisfied for $n=3$ and $n=4$ since $\tau>(n-2)/2$. However, it is an extra restriction for higher-dimensions. From now on, assume that $\tau>n-3$, then $(n-1) \omega_{n-1} \misoc(g,\Psi)(r)$ becomes
\begin{equation}\label{eq:diff_spheres}
\left( -\frac{1}{r} \int_{S_r^\Psi} \left(\delta^{ij} - \frac{x^i(y)x^j(y)}{r(x(y))^2}\right)\sigma_{ij}^\Phi(x(y))\,dS_r^{0} + \frac{(n-1)}{r^2}\int_{B_r^\Psi\setminus B_R^\Psi} \tr_\delta \sigma^\Phi(x(y)) \,d\mu_\delta\right) + o(1),
\end{equation}
We claim that the limit defining the coordinate isoperimetric mass \eqref{eq:admiso_phi} (or \eqref{eq:diff_spheres}) is independent of the family of spheres chosen. In particular, we show that
\begin{equation}\label{eq:CVClaim0}
\frac{1}{r}\int^r_R \misoc(g,\Psi)(\rho) \,d\rho= \frac{1}{r} \int^r_R \misoc(g,\Phi)(\rho) \, d\rho+ o(1).
\end{equation}
Taking the limit as $r \to \infty$, we obtain equality of the masses, since both masses are well-defined.

We start by showing that 
\begin{equation}\label{eq:CVClaim2}
\frac{1}{r^2}\int_{B_r^\Psi\setminus B_R^\Psi} \sigma_{ii}^\Phi(x(y))\,d\mu_\delta - \frac{1}{r^2} \int_{B_r^\Phi \setminus B_R^\Phi} \sigma_{ii}^\Phi(x)\,d\mu_\delta = o(1).
\end{equation}
Since $y^i(x) = x^i + a^i + \zeta(x)^i$, where $\zeta^i = O(r^{1-\tau})$, after changing coordinates we have
\begin{equation}
\frac{1}{r^2}\int_{B_r^\Psi\setminus B_R^\Psi} \sigma_{ii}^\Phi(x(y))\,d\mu_\delta = \frac{1}{r^2} \int_{\Sigma_r\setminus \Sigma_R} \sigma_{ii}^\Phi(x)d\mu_\delta + o(1),
\end{equation}
where $\Sigma_r = (\Phi \circ \Psi^{-1})(B_r^\Psi)
$, and similarly for $\Sigma_R$. In particular $\Sigma_r$ is close to a coordinate ball in the asymptotically flat end of $\Phi$ but slightly deformed and possibly translated. Then, we estimate
\begin{align*}
    \frac{1}{r^2}\left| \int_{\Sigma_r\setminus \Sigma_R} \sigma_{ii}^\Phi(x)d\mu_\delta - \int_{B_r^\Phi \setminus B_R^\Phi} \sigma_{ii}^\Phi(x)\,d\mu_\delta \right|& \leq   \frac{1}{r^2} \int_{\Omega_r} |\sigma_{ii}^\Phi(x)|\, d\Omega_r + o(1)\\
    & \leq  \frac{1}{r^2}\sup_{x \in \Omega_r} |\sigma_{ii}(x)| |\Omega_r|+ o(1),
\end{align*}
where $\Omega_r$ is the symmetric difference $(\Sigma_r \setminus B_r^\Phi) \cup (B_r^\Phi \setminus \Sigma_r)$. Note that $\Omega_r$ is contained in an annular region of width $O(|a| + r^{1-\tau})$ around $S^\Phi_r$ due to $\zeta^i$. From this, we have that $|\Omega_r| = O(|a| r^{n-1} +r^{n-\tau})$ and the claim \eqref{eq:CVClaim2} follows since $\tau>n-3$.

Note that integration with respect to the radius converts the surface integrals of $\misoc(g,\Psi)(r)$ into volumetric integrals. Therefore, it is sufficient to show that
\begin{equation}\label{eq:CVClaim1}
\frac{1}{r}\left(\int_R^r \frac{1}{\rho(y)} \int_{S_\rho^\Psi} \sigma_{ij}^\Phi(x(y))\,dS_\rho^{0}\,d\rho - \int_R^r \frac{1}{\rho(x)} \int_{S_\rho^\Phi} \sigma_{ij}^\Phi(x)\,dS_\rho^{0}\,d\rho \right) = o(1).
\end{equation}
Note that the change of coordinates also affects $\rho$. Since $y^i(x) = x^i + a^i + \zeta^i(x)$ and $|\det Dy(x)| = 1 + O(r(x)^{-\tau})$, by \eqref{eq:r_estimate}, we have 
\begin{equation}
    \frac{|\det Dy(x)|}{r(y(x))} - \frac{1}{r(x)} = O(|a|r^{-2}) + O(r^{-1-\tau}).
\end{equation}

Applying the coarea theorem and changing coordinates, we have
\begin{align}
    \frac{1}{r} \int_{\Sigma_r\setminus \Sigma_R} |\sigma_{ij}^\Phi(x)| \left| \frac{|\det Dy(x)|}{r(y(x))} - \frac{1}{r(x)} \right|\,dx & = \frac{1}{r} \int_{\Sigma_r\setminus \Sigma_R} O(|a|r^{-2-\tau}) + O(r^{-1-2\tau})dx \\
    & = o(1),
\end{align}
since $\tau>n-3$ and $\tau>(n-2)/2$. The term $o(1)$ also absorbs the inner boundary terms. Therefore, estimating the difference over the symmetric difference $\Omega_r$, we obtain
\begin{align*}
    \frac{1}{r}\left|\int_R^r \frac{1}{\rho(y)} \int_{S_\rho^\Psi} \sigma_{ij}^\Phi(x(y))\,dS_\rho^{0}\,d\rho - \int_R^r \frac{1}{\rho(x)} \int_{S_\rho^\Phi} \sigma_{ij}^\Phi(x)\,dS_\rho^{0}\,d\rho \right| &  \leq  \frac{1}{r}\int_{\Omega_r} \frac{1}{\rho(x)}|\sigma_{ij}^\Phi(x)|\,d\Omega_r + o(1)\\
    &  \leq \frac{1}{r} \sup_{x \in \Omega_r} \left|\frac{\sigma_{ij}(x)}{\rho(x)}\right| |\Omega_r| + o(1)\\
    &  =  o(1),
\end{align*}
where we have used that $|\Omega_r| = O(|a| r^{n-1} +r^{n-\tau})$.The remaining term containing $\frac{x^i(y)x^j(y)}{r(x(y))^2}\sigma_{ij}^\Phi(x(y))$ can be treated similarly. This shows that \eqref{eq:CVClaim0} holds and then both masses agree as $r \to \infty$. 
\end{proof}

\subsection{The finiteness of mass}\label{sec:finitenessmass}

At this stage, it is not clear under which conditions $\misoc$ is well-defined, finite, and nonnegative for asymptotically flat manifolds endowed with low-regularity metrics. In the setting of $C^0\cap W_{-\tau}^{1, n}$ asymptotically flat manifolds with $\tau>(n-2) / 2$, \cite{LeeLeFloch2015} introduced a notion of generalized ADM mass $\mathbf{m}$, which agrees with $\madm$ and therefore with $\misoc$ whenever the metric possesses sufficient regularity. They further established a criterion ensuring the finiteness of $\mathbf{m}$, as well as a positive mass theorem for spin manifolds. In this section, we show that both notions of mass coincide under the additional assumption of $C^0_{-\tau}$ decay, and the requirement that the distributional scalar curvature be a measure outside some compact, and that the corresponding results extend to $\misoc$ in this setting. In the sequel, we recall the definition of distributional scalar curvature used in \cite{LeeLeFloch2015}, and together with the definition of $\mathbf{m}$, and establish the equivalence between $\misoc$ and $\mathbf{m}$. 

\subsubsection{Distributional scalar curvature}

The space of \textit{distributions on $M$} is defined as the topological dual of the space of compactly supported densities, that is,
\begin{equation}
\D^\prime(M) := \big(\Gamma_c(M, \text{Vol}(M)) \big)^\prime.
\end{equation}
This definition can be extended to distributional (r,s)-tensor fields on $M$ (see \cite{Grosser2001}). From this, we can define a distributional connection, a distributional Riemann curvature tensor, and so on. For a sufficiently regular metric, i.e., $g \in C^0 \cap W^{1,2}_\loc$ (see \cite{LeFloch2007} for minimal conditions), it is also possible to define a notion of distributional Levi-Civita connection. Building upon this theory, we can define the \textit{distributional scalar curvature} $R[g] \in \D^\prime(M)$ as the trace of the distributional Ricci tensor (for a detailed construction, see \cite[Section 2.3]{Hafemann2026}).

To connect this with the classical setting, recall that for a $C^2$ metric the scalar curvature can be expressed in coordinates as
\begin{equation}\label{eq:CoordScalarC}
R = \partial_k V^k + Q,    
\end{equation}
where the vector field $V$ and the scalar field $Q$ are defined by
\begin{equation}\label{eq:V_defSc}
V^k \colonequals g^{i j} \Gamma_{i j}^k - g^{i k} \Gamma_{ij}^j,    
\end{equation}
\begin{equation}\label{eq:F_defSc}
Q \colonequals  - \left(\partial_m g^{i j}\right) \Gamma_{i j}^m+\left(\partial_j g^{i j}\right) \Gamma_{i m}^m+g^{i j} \Gamma_{i j}^m \Gamma_{k m}^k-g^{i j} \Gamma_{i k}^m \Gamma_{j m}^k.   
\end{equation}
The same coordinate decomposition extends naturally to the distributional setting. Let $(U_i, \psi_i)$ be a coordinate cover with a partition of unity $(\eta_i)$. In each chart, let $V_i$ and $Q_i$ denote the above expressions computed in the corresponding coordinates. Then for any $u \in \Cinfc(M)$, we define 
\begin{align}\label{eq:dist_scalar}
\la R, u \,d\mu_h \ra & = \sum_i \int_{U_i} -V_i^k \partial_k \left(u \eta_i \sqrt{|h|}\right)+ Q_i \left(u \eta_i \sqrt{|h|}\right)dx.
\end{align}
This expression is independent of the choice of partition of unity. 

The following proposition shows that the scalar curvature can also be extended to less regular test densities as long as $g$ is sufficiently regular. 

\begin{proposition}\label{prop:ExtendingRDist_H26}
Let $(M^n, g)$ be a Riemannian manifold with $ g \in C^0\cap W^{1,p}_{\loc}$, $p\geq 2$. Then, the distributional scalar curvature $R[g]$ can be extended to all compactly supported test densities in $L^{(p/2)^\prime}$ whose derivatives lie in $L^{p^\prime}$, where $p^\prime = p/(p-1)$ and $(p/2)^\prime = p/(p-2)$ for $p>2$, while $(p/2)^\prime=\infty$ if $p=2$.
\end{proposition}
\begin{proof}
    The result follows by examining \eqref{eq:dist_scalar} and applying H\"{o}lder's inequality to verify that all terms are integrable. See \cite[Proposition 2.8]{Hafemann2026} for details.  In particular, note that for any $p\geq 2$, any $L^\infty$ test function with derivatives in $L^2$ is an admissible test function.
    \end{proof}
 
In a different vein, the distributional scalar curvature can also be defined in a coordinate-invariant way via a smooth background metric, an approach that is also useful for our purposes. For instance, the definition of distributional scalar curvature in \cite{LeeLeFloch2015} uses this approach and is slightly different from the definition in \eqref{eq:dist_scalar}. The former is well suited for computations, while the latter is better suited for the regularisation procedure. Both definitions are equivalent, and in the proof of finiteness of $\misoc$, the definition provided in \cite{LeeLeFloch2015} is more convenient.

Let $\nablab$ and $\nabla$ be the Levi-Civita connections induced by the smooth background metric $h$ and by $g \in C^2$, respectively. Define the tensor $W:= \nabla - \nablab$, whose components are given by
\begin{equation}\label{eq:W_LeeLeFloch}
    W^k_{ij} = \frac{1}{2} g^{kl} \left( \nablab_i g_{jl}+\nablab_j g_{il}-\nablab_l g_{ij}\right).
\end{equation}
Then, as in \cite{LeeLeFloch2015}, the scalar curvature admits the following decomposition
\begin{equation}
    R[g] = \nablab_k \mathbf{V}^k + F,
\end{equation}
where the vector field $\mathbf{V}^k$ and scalar field $F$ are determined by
\begin{equation}\label{eq:V_LeeLeFloch}
    \mathbf{V}^k := g^{ij} W^k_{ij}- g^{ik} W^j_{ji} = g^{ij}g^{k\ell} \left(\nablab_j g_{i\ell} - \nablab_\ell g_{ij}\right),
\end{equation}
\begin{align}\label{eq:F_LeeLeFloch}
    F  := g^{ij}\overline{R}_{ij}  -  \nablab_k g^{ij} W^k_{ij} + \nablab_k g^{ik} W^j_{ji}
+ g^{ij}\big(W^k_{k\ell}W^\ell_{ij}-W^k_{j\ell}W^\ell_{ik}\big), 
\end{align}
where $\overline{R}_{ij}$ is the Ricci curvature of the smooth background metric $h$ (see \Cref{def:TopStruc}).

Then, if $g \in C^0 \cap W^{1,2}_\loc$, the distributional scalar curvature $R[g]$ is defined, for every compactly supported smooth test function $u:M \to \RR$, by
\begin{equation}\label{eq:ScalarCurvLeeLeFloch}
    \llangle R[g],  u \rrangle
:= \int_M \left( - \mathbf{V} \cdot\nablab \Big( u \, \frac{d\mu_g}{d\mu_h}\Big)  + F u \, \frac{d\mu_g}{d\mu_h} \right)  \,d\mu_h, 
\end{equation}
where $\frac{d\mu_g}{d\mu_h} \in C^0 \cap W^{1,2}_\loc$ is the density of $d\mu_g$ with respect to $d\mu_h$ and $W$, $\mathbf{V}$, and $F$ are defined in \Cref{eq:W_LeeLeFloch,eq:V_LeeLeFloch,eq:F_LeeLeFloch}. Note that all the terms are well-defined since $W \in L^2_\loc$, $V \in L^2_\loc$, and $F \in L^1_\loc$. The definition of $\llangle R[g],  u \rrangle$ does not depend on $h$ (see \cite[Definition 2.1]{LeeLeFloch2015}). 

The coordinate definition \eqref{eq:dist_scalar} and the above construction \eqref{eq:ScalarCurvLeeLeFloch} define the same distribution in the sense that they agree on the natural identification of test densities. More precisely, since $d\mu_h = \frac{d\mu_h}{d\mu_g} d\mu_g$, the two distributions coincide under the following identification,
\begin{equation}
\langle R[g], u  \,d\mu_h \rangle = \langle R[g], u \frac{d\mu_h}{d\mu_g}  \,d\mu_g \rangle  = \llangle R[g],  u\frac{d\mu_h}{d\mu_g} \rrangle.
\end{equation}

If $g \in C^0 \cap W^{1,p}_\loc$, $p\geq 2$, from \Cref{prop:ExtendingRDist_H26} (or a generalized version of \cite[Proposition 2.2]{LeeLeFloch2015}), the distributional scalar curvature \eqref{eq:ScalarCurvLeeLeFloch} can be extended for all compactly supported functions $u \in L^{(p/2)^\prime}$ whose derivatives lie in $L^{p^\prime}$, where $p^\prime = p/(p-1)$ and $(p/2)^\prime = p/(p-2)$ for $p>n$, while $(p/2)^\prime=\infty$ if $p=2$. 

\subsubsection{Finiteness of $\misoc$ and equivalence to the mass of \cite{LeeLeFloch2015}}

Assuming that $g$ is $C^0 \cap W^{1,p}_{-\tau}$ asymptotically flat, with $p\geq 2$ and $\tau>(n-2)/2$, and that the distributional scalar curvature $R[g]$ is a measure outside some compact set, $\mathbf{m}$ can be defined, for any $\eps>0$, by
\begin{equation}\label{eq:mLeeLeFloch}
\mathbf{m} \colonequals \frac{1}{2(n-1) \omega_{n-1}} \lim _{r \rightarrow+\infty}\left(\frac{1}{\eps} \int_{r<\rho<r+\eps} \mathbf{V} \cdot \nablab \rho \,d\mu_h\right),
\end{equation}
where $\mathbf{V}$ is given by \eqref{eq:V_LeeLeFloch} and $\rho$ is the radial function $\rho(x) = \sqrt{\sum (x^i)^2}$ defined by the asymptotically flat end. This limit exists and is independent of $\eps$, though possibly infinite, and finite when $R[g]$ is a finite measure. Although, this statement was proven only for $p=n$ in \cite[Proposition 2.4]{LeeLeFloch2015}, the same holds for $p\geq2$ due to \Cref{prop:ExtendingRDist_H26}. In particular, using this definition of mass \eqref{eq:mLeeLeFloch}, \cite{LeeLeFloch2015} also established a positive mass theorem for spin  $C^0 \cap W^{1,n}_{-\tau}$ asymptotically flat manifolds.

In what follows, we show the finiteness of $\misoc$ by an argument analogous to Lee-LeFloch~\cite[Proposition 2.4]{LeeLeFloch2015} assuming that $g \in C^0_{-\tau} \cap W^{1,p}_{-\tau}$, $p\geq2$ and the distributional scalar curvature is a finite signed measure outside some compact set. This assumption serves as a natural replacement for $L^1$ integrability, since $R[g] \geq 0$ in $\D^\prime$ implies that $R[g]$ is a distribution of order zero and hence defines a nonnegative Radon measure \cite[Theorem 3.18]{Duistermaat2010}. At the end of the proof, it is clear by inspection that $\misoc$ and $\mathbf{m}$ are equivalent. This also implies that $\mathbf{m}$ is coordinate invariant by applying \Cref{theo:CoordIndep}, a fact not addressed in \cite{LeFloch2007} (see \Cref{cor:miso_coordinv}). The additional decay required for this equivalence arises from the fact that the formulas for $\misoc$ in \Cref{sec:IsoFormulas} were derived assuming $C^0_{-\tau}$ rather than $C^0$.

\begin{proposition}\label{prop:EquivMasses}
Let $(M^n, g)$ be a $C^0_{-\tau}\cap W^{1,p}_{-\tau}$, asymptotically flat manifold, with $\tau>(n-2) / 2$ and $p\geq 2$. It follows that
\begin{enumerate}
    \item  If $R[g]$ is a finite, signed measure outside a compact set, then $\misoc$ well-defined and finite.
    \item If $R[g]$ is a measure outside some compact set, then $\misoc$ is well-defined. Moreover, $\misoc$ is finite if and only if $R[g]$ is a finite measure outside a compact set.
    \item Assume either (1) or (2), we have $\misoc = \mathbf{m}$.
\end{enumerate}

\end{proposition}
\begin{proof}
Since $g-\delta \in W^{1,p}_{-\tau}$, and $\tau>(n-2)/2$, \Cref{theo:miso_div} applies, giving
\begin{align}\label{eq:CoordMisoApprox}
    \misoc(r) &= \frac{1}{(n-1)\omega_{n-1} r^2} \left(\int_{B_r \setminus B_R}  (\partial_i g_{ij} - \partial_j g_{ii}) x^j \,d\mu_\delta\right) + o(1).
\end{align}
Using $g^{ij}= \delta^{ij}-\sigma_{ij} + O(r^{-2\tau})$, we can rewrite
\begin{align}
    \misoc(r) &= \frac{1}{(n-1)\omega_{n-1} r^2} \left(\int_{B_r \setminus B_R}  g^{i j} g^{k \ell}\left(\partial_k g_{\ell i}-\partial_i g_{\ell k}\right) x^j \,d\mu_\delta \right) + \frac{1}{r^2}O(\sigma \partial \sigma x)+ o(1).
\end{align}
where $O(\sigma \partial \sigma x)$ schematically denotes all terms involving products of $\sigma$, $\partial \sigma$ and $x^i$. 

We claim that 
\begin{equation}\label{eq:claim_decay_int}
    \frac{1}{r^2}O(\sigma \partial \sigma x) = o(1).
\end{equation}
Since $|\sigma| \leq C_g |x|^{-\tau}$, we see
\begin{equation}
\left|\int_{B_r\setminus B_R} \sigma \partial \sigma x\,d\mu_\delta\right| \leq C_g \int_{B_r\setminus B_R} |x|^{1-\tau} \partial \sigma \,d\mu_\delta.    
\end{equation}
Moreover, because $\partial \sigma \in L^p_{-\tau-1}$, we write
\begin{equation}
|x|^{1-\tau}\partial \sigma = \left(|x|^{1+\tau-\frac{n}{p}}\partial \sigma\right) |x|^{-2\tau+\frac{n}{p}},    
\end{equation}
and applying H\"older, with $1/p^\prime = 1-1/p$, we obtain
\begin{align*}
\left|\int_{B_r\setminus B_R} \sigma \partial \sigma x \,d\mu_\delta\right|  & \leq C_g\left(\int_{B_r\setminus B_R} ||x|^{(1+\tau)-\frac{n}{p}} \partial \sigma|^p \,d\mu_\delta\right)^{1/p}  \left(\int_{B_r\setminus B_R} |x|^{(-2\tau+\frac{n}{p})p^\prime}  \,d\mu_\delta \right)^{1/p^\prime}\\
& \leq C_g \| \partial \sigma\|_{L^{p}_{-\tau-1}} \left(\int_{B_r\setminus B_R} |x|^{(-2\tau+\frac{n}{p})p^\prime}  \,d\mu_\delta \right)^{1/p^\prime}\\
& \leq C_g \| \partial \sigma\|_{L^{p}_{-\tau-1}} r^{\max\left(0,-2\tau+n\right)+\eps},
\end{align*}
for some $\eps>0$ small and the claim follows after dividing by $r^2$.

Therefore, we have 
\begin{align}
     \misoc(r) &= \frac{1}{(n-1)\omega_{n-1} r^2} \int_{B_r \setminus B_R} \sum_{i=1}^n V^i x^i  \,d\mu_\delta + o(1),
\end{align}
where $V^i$ is given in \eqref{eq:V_defSc}.

Next, we relate this to the distributional scalar curvature via \eqref{eq:dist_scalar}. For any $r>R$, define $\chi_r(x)$ to be a compactly supported Lipschitz test function as
\begin{equation}\label{eq:testFunc}
\chi_r(x) =
\begin{cases}
1 &|x| \le R, \\
1-\frac{|x|^2-R^2}{r^2-R^2}, & R \leq|x| \le r, \\
0, &|x| \ge r, \\
\end{cases}
\end{equation}
such that $0\leq \chi_r(x) \leq 1$. From \Cref{prop:ExtendingRDist_H26}, it is clear that $\chi_r$ is sufficiently regular to be used as a test function. A direct computation gives,
\begin{equation}
\partial_i \chi_r= - 2 \frac{x^i}{r^2}+O\left(r^{-2}\right)\,\quad \text{ on } \operatorname{int}(B_{r}\setminus B_R),
\end{equation}  
and $\partial_i \chi_r \equiv 0$ outside $B_r$. Hence,
\begin{align}
    \frac{1}{(n-1)\omega_{n-1}} \int_{B_r \setminus B_R} \sum_{i=1}^n V^i \frac{x^i}{r^2}  \,d\mu_\delta 
    & = \frac{1}{2(n-1)\omega_{n-1}}\int_{B_{r} \setminus B_{R}} - V^i \partial_i\chi_r\,d\mu_\delta + o(1),
\end{align}
where the lower order terms were absorbed by $O(r^{1-2\tau})$.  

Note that, $V$ defined by \eqref{eq:V_defSc} and $\mathbf{V}$ by \eqref{eq:V_LeeLeFloch} agree outside $B_{R}$ as well as $d\mu_h$ and $d\mu_\delta$, then
\begin{equation}
    \frac{1}{2(n-1)\omega_{n-1}}\int_{B_{r} \setminus B_{R}} -V^i \partial_i\chi_r\,d\mu_\delta =  \frac{1}{2(n-1)\omega_{n-1}}\int_{B_{r} \setminus B_{R}} - \mathbf{V} \cdot \nablab \chi_r\,d\mu_h.
\end{equation}
By the definition of distributional scalar curvature in \eqref{eq:ScalarCurvLeeLeFloch}, we have
\begin{equation}
    \int_{B_{r} \setminus B_{R}} -\mathbf{V} \cdot \nablab \chi_r\,d\mu_h =   \llangle R[g], \chi_r \frac{d\mu_h}{d\mu_g} \rrangle - \int_M F \chi_r d\mu_h,
\end{equation}
where $F$ is given by \eqref{eq:F_LeeLeFloch}. Note again that $\chi_r\frac{d\mu_h}{d\mu_g}$ is also an admissible test function by \Cref{prop:ExtendingRDist_H26}. Since $F \in L^{p/2}_{-2\tau-2} \subset L^1$ and $\chi_r \to 1$ pointwise as $r \to \infty$, by Lebesgue's dominated converge theorem 
\begin{equation}
    \int_M F \chi_r d\mu_h \to  \int_M F \,d\mu_h \quad \text{ as } r \to \infty.
\end{equation}

Arguing similarly as in \cite[Proposition 2.4]{LeeLeFloch2015}, if $R[g]$ is a finite signed measure outside a compact, the dominated convergence theorem and boundedness of $\frac{d\mu_h}{d\mu_g}$ implies that $\llangle R[g], \chi_r \frac{d\mu_h}{d\mu_g} \rrangle$ converges as $r \to \infty$. Passing to the limit as $r \to \infty$, we obtain that
\begin{equation}\label{eq:miso_dist_form}
\misoc = \lim_{r \to \infty} \misoc(r)  = \frac{1}{2(n-1)\omega_{n-1}}\left(\llangle R[g], \frac{d\mu_h}{d\mu_g} \rrangle  -  \int_M F \,d\mu_h \right),   
\end{equation}
exists, and is finite.

If $R[g]$ is a (nonnegative) measure outside a compact set, the quantity  $\llangle R[g], \chi_r \frac{d\mu_h}{d\mu_g} \rrangle$ is monotone in $r$, and so it must converge as $r \to \infty$, though possibly infinite. In particular, $\misoc(g)$ exists and it is finite if and only if $R[g]$ is a finite measure outside a compact set. 

Finally, the equality between the masses $\misoc$ and $\mathbf{m}$ follows by comparing  \eqref{eq:miso_dist_form} with the corresponding expression for $\mathbf{m}$ in \cite[Equation 2.9]{LeeLeFloch2015} that coincides with \eqref{eq:mLeeLeFloch}.
\end{proof}

The finiteness condition for $\misoc$ (resp. $\mathbf{m})$ directly implies coordinate invariance for more regular metrics.

\begin{corollary}\label{cor:miso_coordinv}
  Let $(\Phi, x)$ and $(\Psi, y)$ be two structures of infinity for $(M^n, g)$ such that $g$ is a $W^{1,p}_{-\tau_1}(\Phi)$ and $W^{1,p}_{-\tau_2}(\Psi)$ asymptotically flat metric with $p>n$ and decay rates $\tau_1, \tau_2$ respectively, and $\tau = \min\{\tau_1, \tau_2\} >(n-2)/2$. Assume that $R[g]$ satisfies (1) or (2) in \Cref{prop:EquivMasses}. Then $\misoc(g,\Phi)$ and $\misoc(g, \Psi)$ are well-defined and coincide. 
\end{corollary}
\begin{proof}
    The masses $\misoc(g,\Phi)$ and $\misoc(g, \Psi)$ are well-defined by \Cref{prop:EquivMasses}. By harmonic coordinates \cite[Corollary 3.2]{Bartnik1986}, the coordinate systems $(\Phi, x)$ and $(\Psi, y)$ differ by a rigid motion plus lower-order terms. If $n<5$ or $\tau>n-3$, the coordinate invariance of the coordinate isoperimetric mass under such transformations is given by \Cref{theo:CoordIndep}, which implies $\misoc(g,\Phi) = \misoc(g, \Psi)$. 

    The above can be proved more generally in any dimension $n\geq3$ by showing that the defining limit of $\misoc$ is independent of the family of coordinate spheres. The proof starts by using the cancellation from the proof of \Cref{theo:CoordIndep}, without invoking the additional decay assumption $\tau>n-3$. The remaining steps essentially reproduce the proofs of \Cref{theo:miso_div} and \Cref{prop:EquivMasses}. In what follows, we sketch the main steps. Using the notation of the proof of \Cref{theo:CoordIndep}, recall that from \eqref{eq:approx_aftercancelation} $(n-1)\omega_{n-1} \misoc(g,\Psi)(r)$ can be written as
$$\left( -\frac{1}{r} \int_{S_r^\Psi} \left(\delta^{ij} - \frac{y^i y^j}{r^2}\right)\sigma_{ij}^\Phi(x(y))\,dS_r^{0} + \frac{(n-1)}{r^2}\int_{B_r^\Psi\setminus B_R^\Psi} \tr_\delta \sigma^\Phi(x(y)) \,d\mu_\delta\right) + o(1).$$
    Following the same argument as in \Cref{theo:miso_div}, but applied to $\sigma_{ij}^\Phi(x(y))$, we obtain
    \begin{equation}\label{eq:change_coord1}
        \misoc(g,\Psi)(r) = \frac{1}{(n-1)\omega_{n-1} r^2} \left(\int_{B^\Psi_r \setminus B^\Psi_R}  \left[\partial_i^y \left(\sigma^\Phi_{ij}(x(y))\right) - \partial_j^y \left(\sigma^\Phi_{ii}(x(y))\right)\right] y^j \,d\mu_\delta\right) + o(1),
    \end{equation}
    where $\partial_i^y $ are the coordinate derivatives with respect to $y$.

    Since both charts differ by a rigid motion plus lower-order terms, w.l.o.g, assuming no rotation, we have that  $\frac{\partial x^{i}}{\partial y^j} = \delta^i_j + O(r^{-\tau})$. By the chain rule and the claim \eqref{eq:claim_decay_int}, we see that \eqref{eq:change_coord1} can be rewritten as 
    \begin{equation}
        \misoc(g,\Psi)(r) = \frac{1}{(n-1)\omega_{n-1} r^2} \left(\int_{B^\Psi_r \setminus B^\Psi_R}  \left[\partial_i^x \sigma^\Phi_{ij}(x(y)) - \partial_j^x \sigma^\Phi_{ii}(x(y))\right] y^j \,d\mu_\delta\right) + o(1),
    \end{equation}
    Denote by $V_\Phi$ the vector field \eqref{eq:V_defSc} induced by $\Phi$ and define a test function $\chi_r^\Psi(y)$ as in \eqref{eq:testFunc}. Following the proof of \Cref{prop:EquivMasses}, we can rewrite the above expression as
    \begin{align*}
                \misoc(g,\Psi)(r) & = \frac{1}{2(n-1)\omega_{n-1}} \left(\int_{B^\Psi_r \setminus B^\Psi_R} - {V}_\Phi^i(x(y)) \partial_i^y \chi_r^\Psi(y) \,d\mu_{ \delta} \right)+o(1).
    \end{align*}
    Changing coordinates and using the claim \eqref{eq:claim_decay_int} again, we have
        \begin{align}\label{eq:Vi_Phi_eq_1}
                \misoc(g,\Psi)(r) & = \frac{1}{2(n-1)\omega_{n-1}} \left(\int_{\Sigma_r \setminus \Sigma_R} - {V}_\Phi^i(x) \partial_i^x \chi_r^\Psi(y(x)) \,d\mu_{ \delta} \right)+o(1),
    \end{align}
    where $\Sigma_r = (\Phi \circ \Psi^{-1})(B_r^\Psi)$.

    Let $h^\Phi$ be the Riemannian background metric induced by $\Phi$ that is Euclidean outside a large ball. Without loss of generality, assume that $h^\Phi=\delta$ outside $\Sigma_R$. Let $\overline{\nabla}^\Phi$ be the Levi-Civita connection of $h^\Phi$ and let $\mathbf{V}_\Phi$ be the vector field \eqref{eq:V_LeeLeFloch} induced by $\overline{\nabla}^\Phi$ and $\nabla^g$. Then, \eqref{eq:Vi_Phi_eq_1} can be written in terms of the geometric quantities defined by $\Phi$, which are related to the distributional scalar curvature \eqref{eq:ScalarCurvLeeLeFloch} acting on $\chi^\Psi_r(y(x))$, i.e.,
        \begin{align*}
        \misoc(g,\Psi)(r) & = \frac{1}{2(n-1)\omega_{n-1}} \left(\int_{\Sigma_r \setminus \Sigma_R}  - \mathbf{V}_\Phi(x) \cdot \overline{\nabla}^\Phi \chi_r^\Psi(y(x)) \,d\mu_{ \delta} \right)+o(1)\\
         & =   \frac{1}{2(n-1)\omega_{n-1}} \left(\llangle R[g], \chi_r^\Psi \frac{d\mu_{h^\Phi}}{d\mu_g} \rrangle - \int_M F(x) \chi_r^\Psi(y(x)) \,d\mu_{h^\Phi}\right) + o(1).
    \end{align*}
    Taking the limit as $r \to \infty$, arguing exactly as in the proof of \Cref{prop:EquivMasses} using its conditions (1) or (2), and using the fact that $\chi_r^\Psi \to 1$ pointwise as $r \to \infty$, we have
    \begin{align*}
        \lim_{r \to \infty} \misoc(g,\Psi)(r) =  \frac{1}{2(n-1)\omega_{n-1}}\left(\llangle R[g], \frac{d\mu_{h^\Phi}}{d\mu_g} \rrangle - \int_M F(x) \,d\mu_{h^\Phi} \right)= \lim_{r \to \infty}\misoc(g,\Phi)(r),
    \end{align*}
    then $\misoc(g,\Psi)  = \misoc(g,\Phi)$.
\end{proof}

\begin{remark}\label{remark:weakADM}
    For $C^0_{-\tau}\cap W^{1,p}_{-\tau}$ asymptotically flat metrics with $p\geq 2$ and $\tau>(n-2)/2$, the proof of \Cref{prop:EquivMasses} also shows that $\misoc$ agrees with the weak ADM mass $m_W(g, \Phi)$ in \cite[Definition 3.6]{Lundgren2026}, provided that $g$ has distributional scalar curvature in $L^1$ sense of \cite[Definition 3.3]{Lundgren2026}. This equivalence can be seen by noting that the family of test functions $\chi_{r}$ in \eqref{eq:testFunc} form a sequence of compactly supported locally Lipschitz functions $\{\chi_{r}:M \to \RR\}_{r\geq R}$ satisfying 
    $$\sup_{r}\sup_M \left(|\chi_{r}| + |rD\chi_{r}|\right) \leq C,$$
    for some constant $C>0$ and $\chi_r^{-1}(1)$ contains $K$ and exhausts $M$ as $r \to \infty$. Then, by \eqref{eq:CoordMisoApprox}, we obtain
    \begin{equation}
        \misoc(r) = \frac{1}{2(n-1)\omega_{n-1}} \left(\int_{M}  (\operatorname{div}_\delta(e) -d\tr_\delta(e)) (-\overline{\nabla}\chi_r) \,d\mu_h\right) + O(r^{(n-2)-2\tau}),
    \end{equation}
    where $e = g - h$. This expression is an approximation for $m_W(g, \Phi)$, which is independent of $\{\chi_{r}:M \to \RR\}_{r\geq R}$ (see \cite[Theorem 3.9]{Lundgren2026}). Taking the limit as $r \to \infty$ yields
    \begin{equation}
        \misoc = \lim_{r \to \infty} \frac{1}{2(n-1)\omega_{n-1}} \left(\int_M  (\operatorname{div}_\delta (e) -d\tr_\delta (e)) (-\overline{\nabla}\chi_r) \,d\mu_h\right) = m_W (g, \Phi).
    \end{equation}
\end{remark}

\section{Positive mass theorem and density theorem}

The equivalence between $\misoc$ and $\mathbf{m}$ established in the previous section, combined with the distributional positive mass theorem for spin manifolds in \cite[Theorem 1]{LeeLeFloch2015}, yields the following spin positive mass theorem for $\misoc(g)$.

\begin{corollary}\label{cor:miso_PMTspin}
   Let $(M^n, g)$ be a $C^0_{-\tau}\cap W^{1,n}_{-\tau}$, complete asymptotically flat manifold, with $\tau>(n-2) / 2$ endowed with a spin structure. If $R[g]\geq 0$ in $\D^\prime$ and is a finite measure outside a compact set, then $\misoc\geq 0$. Moreover, $\misoc(g) \equiv 0$ if $(M,g)$ is isometric to Euclidean space.
\end{corollary}

In this section, our goal is to show that the spin condition in \Cref{cor:miso_PMTspin} can be dropped. The argument combines weighted Friedrichs' lemma, established in \cite{Hafemann2026} for compact sets, with a generalization of the density theorem of \cite{LLR2022} to metrics without decay assumptions on their second derivatives. The proof of the weighted Friedrichs lemma (\Cref{cor:ConvFriedPlusApproximations_weighted}) largely reduces to the compact case, so it is deferred to \Cref{append:A}. The density theorem (\Cref{theo:density}) depends crucially on this result and on the observation that decay of the first derivatives of the metric yields a certain control of the second derivatives of the mollified metrics  (see \eqref{eq:Weighted_reg_dev}). 

The remainder of the paper is organized as follows. In \Cref{sec:reg_metric}, we derive improved convergence results for the scalar curvature and for $\misoc$ of smooth approximations $g_\varepsilon$ using weighted Friedrichs' lemma developed in \Cref{append:A}. \Cref{sec:AnalysisConfLap} develops elliptic estimates for $W^{1,p}_{-\tau}$ metrics, where $p>n$ and $\tau>(n-2)/2$, while \Cref{sec:ProofDensityTheo} verifies that they hold for the mollified metrics $g_\eps$, even after interpolating with Euclidean space, and apply it to prove the density theorem \Cref{theo:density}. These latter two sections follow the exposition of \cite{LLR2022}. In \Cref{sec:Rigidity_W1p}, we show the rigidity part of the main theorem (\Cref{theo:main_intro}) following the approach in \cite{Hafemann2026} but without assuming smoothness outside a compact set.

\subsection{Regularised metric and isoperimetric mass} \label{sec:reg_metric}

\subsubsection{Basic regularisation procedure and properties}

In order to deal with distributions, we introduce the regularisation of non-smooth tensor fields by performing convolution with a mollifier in local charts. Start by choosing a countable, locally finite family of relatively compact coordinate charts $\psi_i \colon U_i \to B_1(0)\subset \RR^n$, $i \in \mathbb{N}$, which together cover $M$. Let $(\eta_i)$ be a partition of unity subordinate to this covering, with $\operatorname{supp}(\eta_i) \subseteq U_i$. Additionally, choose smooth cut-off functions $\chi_i \in \Cinfc(U_i)$ satisfying $\chi_i \equiv 1$ on an open set containing $\operatorname{supp}(\eta_i)$. Next, let $\rho \in \Cinfc(\RR^n)$ be a nonnegative function with unit integral supported in $B_1(0)$ and define $\rho_\eps(x) \colonequals \eps^{-n} \rho(x/\eps)$ for $\eps \in (0,1]$. Using the standard notation $f_*$ and $f^*$ for push-forward  and pull-back of distribution under diffeomorphism $f$, we define the regularised version of a distributional tensor field $\mathcal{T} \in \D^\prime \mathcal{T}^r_s(M)$ as
\begin{equation}\label{eq:reg_tensor}
    \mathcal{T} \star_M \rho_\eps (x) \colonequals \sum_{i \in \mathbb{N}} \chi_i (x) (\psi_i)^* \left[ \big((\psi_i)_*(\eta_i \mathcal{T})\big) * \rho_\eps\right] (x),
\end{equation}
where $(\psi_i)_*(\eta_i\mathcal{T})$ is a distribution with compact support on $\RR^n$, and convolution with $\rho_\eps$ is performed component-wise. This procedure produces a smooth tensor field on $\RR^n$. The smoothness of the map $(\eps, x) \mapsto \mathcal{T} \star_M \rho_\eps$ on $(0, 1] \times M$ is guaranteed by the cut-off function $\chi_i$. Since $\rho$ is nonnegative, we have that for any $u \in \D^\prime(M)$ we also have $u \star_M \rho_\eps\geq 0$ for any $\eps \in (0,1]$.

This smoothing procedure provides convergence properties similar to those of standard convolution regularisation in $\RR^n$. Let $g$ be a $C^0\cap W^{1,p}_\loc$-Riemannian metric, $p\geq2$ and define
\begin{equation} \label{eq:def_reg_metric}
    g_\eps \colonequals g \star_M \rho_\eps,
\end{equation}
then, for sufficiently small $\eps>0$, $g_\eps$ is a Riemannian metric and $g_\eps \to g $ locally uniformly and in $W^{1,p}_\loc$ as $\eps \to 0$. From the general formula for the inverse of $g$, we have $(g \star_M \rho_\eps)^{-1} \to g^{-1}$ locally uniformly and in $W^{1,p}_\loc$ as well. In particular, $R[g] \in \D^\prime(M)$ and $R[g_\eps]$ and $R[g] \star_M \rho_\eps$ converges to $R[g]$ as distributions (see \cite[Theorem 4.6]{LeFloch2007}).

For manifolds $M$ possessing a topological structure of infinity, the procedure described in \Cref{eq:reg_tensor} is always carried out using a family of coordinate charts which cover $K$ together with $\Phi$. This provides a natural choice of charts to regularize rough metrics on $M$, as in \eqref{eq:def_reg_metric}. Given this identification, the convergence properties can also be extended to weighted Sobolev spaces as follows.
\begin{lemma}\label{lemma:FamilyOfMetrics_weighted}
    Let $g$ be a $C^0_s \cap W^{k,p}_{s}$ asymptotically flat metric, with $k$, $p \in [1,\infty)$, and $s \in \RR$. Then, for all $\eps>0$, there exists a smooth Riemannian metric $g_\eps$ and with the following properties:
    \begin{enumerate}
        \item For every $s^\prime>s$, $g_\eps $ converge to $g$ in $C^0_{s^\prime}(M) \cap W^{k,p}_s(M)$ as $\eps \to 0$. 
        \item If $k-n/p>0$, then $g_\eps $ converge to $g$ in $C^0_{s}(M) \cap W^{k,p}_s(M)$ as $\eps \to 0$. Moreover, if  $s\leq0$, then there exists $\eps_0>0$, such that for all $\eps < \eps_0$, the metrics $g_\eps$ are uniformly equivalent to $g$ with the same constant, i.e., there exists $\Lambda(\eps)>0$ such that
        $$
        \Lambda(\eps)^{-1} g_\eps \leq g \leq \Lambda(\eps) g_\eps,
        $$
        and $\Lambda(\eps) \to 1$ as $\eps \to 0$.
    \end{enumerate}
\end{lemma}
\begin{proof}
    The convergence in $C^0_\loc$ and $W^{k,p}_\loc(M)$ follows almost exactly as in \cite[Lemma 4.1]{grant2014positive}, except that restricting the mollification of the metric $g$ to a compact set is unnecessary. To obtain convergence in $W^{k,p}_{s}$, we decompose $M = K \cup \mathcal{E}$ where $K$ is a sufficiently large compact set and set $\mathcal{E} = M \setminus K$. The convergence on $K$ follows from the local result while the convergence on $\mathcal{E}$ follows from the estimate $\|f \star_M\rho_\eps\|_{L^p_s(\mathcal{E})} \leq C \|f\|_{L^p_s(\mathcal{E}_\eps)}$ where $\mathcal{E}_\eps$ is the $\eps$-neighbourhood of $\mathcal{E}$, which follows by Jensen's inequality and the continuity of translation operators on weighted spaces (see \Cref{lemma:translation_weighted} and \Cref{lemma:AuxFriedrichsWp_weighted}). The $C^{0}_{s^\prime}$ convergence follows, for $s^\prime > s$, from a similar argument. Item (2) follows from the weighted Sobolev embedding \cite[Theorem 1.2]{Bartnik1986}, if $k-n/p>0$ then $W^{k,p}_s(M) \subset C^{0,\alpha}_s(M)$ for some $\alpha>0$, and by applying the $C^0_s$ convergence with $s\leq 0$ and the $g$ uniform equivalence to $\delta$ in $\mathcal{E}$ \eqref{eq:UniformEquiv} to conclude uniform equivalence of $g_\eps$ and $g$. 
\end{proof}

\subsubsection{Improved estimates for $R[g_\varepsilon]$ and convergence of $\misoc$}

The approximating metrics in the density theorem are built upon the approximations $g_\eps$, and for this, appropriate control over the negative part of their scalar curvature $R[g_\eps]_-$ and coordinate isoperimetric mass $\misoc(g_\eps)$ is necessary. The two main steps consist in obtaining improved convergence estimates for $R[g_\eps]_-$ (\Cref{prop:ConvofRs_weighted}) and showing that $\misoc(g_\eps)$ remains close to $\misoc(g)$ (\Cref{cor:miso_close_eps}). In what follows, we derived these results.

Because $R[g]\star_M \rho_\eps$ inherits good properties of $R[g]$, while $R[g_\eps]$ does not immediately do so, we prove the following result, which can be understood as an improved convergence statement for $R[g_\eps]$ when compared to the regularised distribution $R[g] \star_M \rho_\eps$, obtained via the weighted Friedrichs' lemma (\Cref{cor:ConvFriedPlusApproximations_weighted}). From this, we also obtain improved estimates for $R[g_\eps]_-$ if $R[g]\geq 0$ in $\D^\prime$, which form the cornerstone of the main proof. Note that the following result does not require any decay of the $C^0$ part of the metric.

\begin{proposition}\label{prop:ConvofRs_weighted}
    Let $(M^n, g)$ be a $C^0\cap W^{1,p}_{-\tau}$, $p\geq 2$, asymptotically flat manifold, with $\tau>0$. Then 

   \begin{enumerate}
       \item  $R[g] \star_M \rho_\eps - R[g_\eps] \longrightarrow 0 \text{ in } L^{p/2}_{-2\tau-2}(M) \text { as } \eps \to 0.$
       \item If $R[g]\geq0$ in $\D^\prime$, then $R[g_\eps]_- \longrightarrow 0$ in $L^{p/2}_{-2\tau-2}(M)$ as $\eps \to 0$.
   \end{enumerate}
\end{proposition}
\begin{proof}
This proof largely follows that of \cite[Proposition 3.5]{Hafemann2026}. Recall that, as presented in \eqref{eq:CoordScalarC}, the scalar curvature can be written schematically as
$$
R[g] = \partial V[g] + Q[g],
$$
where $V[g] \approx (g^{-1})^2 \partial g$ and $Q[g] \approx (\partial g)^2$.

Since $\partial g \in L^p_{-\tau-1}$, and $g^{-1}$ is continuous and uniformly bounded (see \Cref{sec:preliminaries}), we have that 
$$
Q[g_\eps]-Q[g]\star_M \rho_\eps \longrightarrow 0 \text{ in } L^{p/2}_{-2\tau-2}(M) \text{ as } \eps \to 0,
$$
by the weighted H\"{o}lder inequality together with the standard mollifier convergence properties.

For the remaining term, it is sufficient to show that
$$
V[g_\eps] - V[g] \star_M\rho_\eps \longrightarrow 0 \text{ in } W^{1, p/2}_{-2\tau-1}(M) \text { as } \eps \to 0.
$$
We now sketch the proof of this convergence. Since $g$ is $C^0 \cap W^{1,p}_{-\tau}$ asymptotically flat metric, the quantity $V[g_\eps]$ can be written schematically as $
V[g_\eps] \approx a_\eps f_\eps$ where $a_\eps - c = \left((g_\eps^{-1})^2 -c \right) \in W^{1,p}_{-\tau}$,  $f_\eps = \partial g_\eps \in L^p_{-\tau-1}$, and where $c$ is equals to $1$ if $a_\eps$ consists of diagonal components the metric and zero otherwise. By the standard mollifier convergence properties $a_\eps \to a$ in $W^{1,p}_{-\tau}$ and in $C^0_\loc$ and $f_\eps \to f$ in $L^{p}_{-\tau-1}$ where $a =(g^{-1})^2$ and $f =\partial g$. Moreover, the convergence rate estimate $\|a_\eps - a \|_{L^p_{-\tau-1}} \leq C \eps$, for some constant $C>0$ independent of $\eps$, can be obtained by expressing the inverse of the metric using the cofactor matrix formula together with the estimate \eqref{eq:speed_decay2}. These are precisely the hypotheses of \Cref{cor:ConvFriedPlusApproximations_weighted}. Therefore, together with \Cref{remark:appendixFried}, we conclude that
\begin{equation}
R[g] \star_M \rho_\eps - R[g_\eps] \longrightarrow 0 \text{ in } L^{p/2}_{-2\tau-2}(M) \text { as } \eps \to 0,
\end{equation}
and the first result follows. A detailed proof can be found in \cite[Proposition 3.5]{Hafemann2026}.

The convergence of $R[g_\eps]_-$ follows from the fact that if $R[g]\geq0$ in $\D^\prime$, then $R[g] \star_M \rho_\eps \geq 0$ pointwise. Hence $|R[g_\eps]_-| \leq |R[g_\eps] - R[g] \star_M \rho_\eps|$ and
\begin{equation}
\|R[g_\eps]_-\|_{L^{p/2}_{-2\tau-2}(M)} \leq \|R[g_\eps]-R[g] \star_M \rho_\eps\|_{L^{p/2}_{-2\tau-2}(M)} \longrightarrow 0 \text { as } \eps \to 0,
\end{equation}
by the first part of the proof.
\end{proof}  

The above result is essential to show that the finiteness of the $L^1$ norm of $R[g_\eps]$ provided that $R[g]$ is a finite signed measure outside a compact set. In particular, this implies that $\misoc(g_\eps)$ is finite. Moreover, recall that if $R[g]\geq 0$ in $\D^\prime$, then $R[g]$ can be identified with a nonnegative Radon measure \cite[Theorem 3.18]{Duistermaat2010}.

\begin{lemma}\label{lemma:conv_eps_sc_curv_L1}
    Let $(M^n, g)$ be a $C^0\cap W^{1,p}_{-\tau}$, $p\geq 2$, asymptotically flat manifold, with $\tau>(n-2) / 2$. Suppose that $R[g]$ is a finite signed measure outside a compact set. Then the following holds.
    \begin{enumerate}
    \item $R[g] \star_M \rho_\eps, R[g_\eps] \in L^1(M)$ with $L^1$-norms uniformly bounded in $\eps$ outside a compact set.
    \item $R[g_\eps] \to R[g]$ in $\D^\prime$ globally and weakly as signed measures outside a compact set, i.e., there exists a compact set $K$ and a finite signed measure $dR[g]$ outside $K$ that agrees with the distribution $R[g]$ such that 
    \begin{equation}
        \int_{M\setminus K} f R[g_\eps] d\mu_h - \int_{M\setminus K} f dR[g] \longrightarrow 0 \text{ as } \eps \to 0,
    \end{equation}
    for all $f \in C^0_\mathrm{b}(M)$ bounded continuous functions on $M$ such that $f = 0$ on $K$.
    \item If $R[g]\geq 0$ in $\D^\prime$, then $R[g] \star_M \rho_\eps$ and $R[g_\eps]$ are uniformly bounded in $L^1(M)$ as $\eps \to 0$. Moreover, $R[g]$ can be identified with a nonnegative Radon measure and $R[g_\eps]\,d\mu_h \rightharpoonup R[g]$ weakly as signed Radon measures on $M$.
    \end{enumerate}  
\end{lemma}
\begin{proof}
   Assume that $R[g]$ is a finite signed measure outside a compact set, then there exists a compact set $K_0$ and some finite signed measure $dR[g]$ supported on $M \setminus K_0$ such that 
   \begin{equation}\label{eq:measure_dRg}
    \langle R[g], \varphi \,d\mu_h \rangle = \int_{M \setminus K_0} \varphi(x) \,dR[g](x), \quad \forall \varphi \in \D(M \setminus K_0),
   \end{equation}
   and $|dR[g]|(M \setminus K_0) < \infty$. Without loss of generality, assuming that $K_0$ is sufficiently large and its complement included in the asymptotic end $\mathcal{E}$, the convolution operator can be written only in the asymptotically flat chart. Choose a compact set $K_1$ such that $K_0 \subset \operatorname{int}(K_1)$. Then for any $x \in M \setminus K_1$ and $\eps>0$ sufficiently small, the support of $\rho_\eps(x - \cdot)$ is contained in $M \setminus K_0$ and we have
   \begin{align*}
    |R[g]\star_M \rho_\eps(x)| = \left| \langle R[g], \rho_\eps(x-\cdot) \,d\mu_h \rangle \right| 
    =  \left|  \int_{M\setminus K_0} \rho_\eps(x-y) \,dR[g] (y) \right|
    \leq     \int_{M\setminus K_0} \rho_\eps(x-y) \left|dR[g]\right|(y).
   \end{align*}
   Now, integrating in $x$ and applying Fubini's theorem yields the estimate
      \begin{align*}
    \int_{M\setminus K_1} |R[g]\star_M \rho_\eps(x)| \,dx & \leq   \int_{M\setminus K_1} \int_{M\setminus K_0} \rho_\eps(x-y) \left|dR[g]\right|(y) \,dx \\ 
    & \leq   \int_{M\setminus K_0} \left(\int_{M\setminus K_1} \rho_\eps(x-y) \,dx \right)\left|dR[g]\right|(y) \\
    & \leq   \int_{M\setminus K_0} \left|dR[g]\right|(y) = |dR[g]|(M\setminus K_0) < \infty,
   \end{align*}
   where we have used that $\rho_\eps$ has unit $L^1$ norm. Since $R[g]\star_M\rho_\eps \in C^\infty(M)$, this proves that $R[g]\star_M\rho_\eps \in L^1(M)$ and $\|R[g]\star_M\rho_\eps\|_{L^1(M \setminus K_1)}\leq C$ uniformly in $\eps$.

    The integrability of $R[g_\eps]$ follows from the estimate
    \begin{equation}
    \|R[g_\eps]\|_{L^1} \leq \|R[g_\eps]-R[g] \star_M \rho_\eps\|_{L^1} +  \|R[g] \star_M \rho_\eps\|_{L^1},    
    \end{equation}
    because \Cref{prop:ConvofRs_weighted} implies that $R[g_\eps]-R[g] \star_M \rho_\eps \to 0 \text{ in } L^{p/2}_{-2\tau-2}(M) \subset L^1(M)$ if $\tau>(n-2)/2$, and $R[g] \star_M \rho_\eps \in L^1(M)$. In particular, the uniform $L^1$ bound for $R[g] \star_M \rho_\eps$ outside a compact set, together with the fact $R[g_\eps]-R[g] \star_M \rho_\eps \to 0 \text{ in } L^1$ that implies that $\|R[g_\eps]\|_{L^1}$ remains uniformly bounded for all $\eps$ sufficiently small outside a compact set as well, that is,
    \begin{equation}\label{eq:Unif_bound_Rgeps}
    \|R[g_\eps]\|_{L^1(M\setminus K_1)} \leq \|R[g_\eps]-R[g] \star_M \rho_\eps\|_{L^1(M\setminus K_1)} +  |dR[g]|(M\setminus K_0),    
    \end{equation}
    and this proves item (1).

   For item (2), since $g \in C^0 \cap W^{1,p}_\loc$ and $g_\eps \to g$ in $C^0_\loc \cap W^{1,p}_\loc$, we have that $R[g_\eps] \to R[g]$ in $\D^\prime$ as $\eps \to 0$. Let $f \in C_\mathrm{b}^0(M)$ be a bounded continuous functions such that $f=0$ on $K_1$, then 
   \begin{align*}
       \left| \int_{M} fR[g_\eps] \,dx  - \int_{M}fdR[g]\right| & \leq \|f\|_\infty \|R[g_\eps] - R[g]\star_M \rho_\eps\|_{L^1(M)} +  \left|\int_{M} f (R[g]\star_M \rho_\eps \,dx - dR[g] )\right|,
   \end{align*}
   \Cref{prop:ConvofRs_weighted} implies that the first term converges to zero. To see that the second term converges to zero, we use Fubini's theorem and the fact that $\rho_\eps$ has unit integral to estimate
   \begin{align*}
     \left| \int f (R[g]\star_M \rho_\eps \,dx - dR[g] )\right| & = \left|\int \int f(x)  \rho_\eps(x-y) dR[g](y) \,dx - \int f(y)dR[g](y) \right|\\
     & = \left|\int \int (f(x)-f(y))  \rho_\eps(x-y) \,dx dR[g](y) \right|\\
     & \leq  \int \underbrace{\left(\int \left|(f(x)-f(y))\right| \rho_\eps(x-y) \,dx\right)}_{\colonequals F_\eps(y)} |dR[g]|(y).
   \end{align*}
   For every fixed $y$, the continuity of $f$ implies that $F_\eps(y) \to 0$ as $\eps \to 0$. Moreover, $|F_\eps| \leq 2\|f\|_\infty$ by the standard properties of mollifiers. Since $|dR[g]|(M \setminus K_0)$ is finite, the dominated convergence theorem shows that the second term converges to zero. Therefore, $R[g_\eps]\,d\mu_h \rightharpoonup R[g]$ weakly as signed measures outside a compact set as $\eps \to 0$.

   The distributional nonnegativity of $R[g]$ assumed in item (3) implies that $R[g]$ can be identified with a nonnegative Radon measure on $M$ \cite[Theorem 3.18]{Duistermaat2010}. This measure is finite on compact sets and finite outside $K_0$ by assumption, so it is finite globally. Hence, the above arguments can be globalized, yielding in particular that $\|R[g]\star_M\rho_\eps\|_{L^1(M)}$ and $\|R[g_\eps]\|_{L^1(M)}$ are uniformly bounded in $\eps$. Since $\rho_\eps \,dx \rightharpoonup \delta_0 $ weakly as measures, a chartwise argument using the definition of $\star_M$ gives $R[g] \star_M \rho_\eps \,d\mu_h \rightharpoonup dR[g]$ weakly as signed Radon measures on $M$ as $\eps \to 0$. Combining this with $\|R[g_\eps]-R[g] \star_M \rho_\eps\|_{L^1} \to 0$ as $\eps \to 0$ yields $R[g_\eps]\,d\mu_h \rightharpoonup dR[g]$ weakly as signed Radon measures on $M$ as $\eps \to 0$.

\end{proof}

In what follows, we note that suitable approximations do not significantly change the mass and preserve its finiteness.

\begin{lemma}\label{lemma:mass_cont}
Let $(M^n, g)$ be a $C^0_{-\tau} \cap W^{1,p}_{-\tau}$ asymptotically flat manifold, with $\tau>(n-2)/2$ and $p\geq 2$, such that $R[g]$ is a finite signed measure outside a compact set $K$. Let $g_i$ be a sequence of $C^0_{-\tau} \cap W^{1,p}_{-\tau}$ asymptotically flat metrics such that each $R[g_i]$ is a finite signed measure outside a fixed compact set $K$, represented by $dR[g_i]$, and such that $|dR[g_i]|(M \setminus K)$ is uniformly bounded. Moreover, assume that $g_i \to g$ in $L^\infty$, $\partial g_i \to \partial g$ in $L^2$, and $dR[g_i] \rightharpoonup dR[g]$ weakly as signed measures outside $K$. Then, 
$$\misoc(g_i) \longrightarrow \misoc(g) \quad \text{ as } i \to \infty.$$
\end{lemma}
\begin{proof}
    Since $R[g_i]$ is a finite, signed measure outside a compact set, $\misoc(g_i)$ is finite by \Cref{prop:EquivMasses}, and from \eqref{eq:miso_dist_form}, the difference $\misoc(g)-\misoc(g_i)$ is given by
\begin{equation}   
 \frac{1}{2 (n-1) \omega_{n-1} }\Bigg(\llangle R[g]-R[g_i], \frac{d\mu_h}{d\mu_g} \rrangle +\llangle R[g_i], \frac{d\mu_h}{d\mu_{g}} - \frac{d\mu_h}{d\mu_{g_i}}  \rrangle
 - \int_M F-F_i \,d\mu_h \Bigg),
\end{equation}
where $ \frac{d\mu_h}{d\mu_{g_i}} \in C^0 \cap W^{1,p}_\loc$, $p\geq 2$, is the density of $d\mu_h$ with respect to $d\mu_{g_i}$. Note that $\frac{d\mu_h}{d\mu_{g_i}} $ is an admissible test density for both distributional scalar curvatures $R[g]$ and $R[g_i]$ by \Cref{prop:ExtendingRDist_H26} and the assumed regularity of $g, g_i \in W^{1,p}_\loc$, $p\geq 2$.

Let $\chi \in C^\infty_c(M)$ be equal to $1$ in a neighbourhood of $K$. Splitting the first term using $\chi$, its compactly supported part converges to zero by the distributional convergence implied by $g_i \to g$ in $L^\infty_\loc$ and $\partial g_i \to \partial g$ in $L^2_\loc$. The exterior part converges to zero by the weak convergence, since $(1-\chi) d\mu_h/d\mu_g$ is an admissible test function converging to $1$ at infinity. 

The second term is handled similarly, using that $\frac{d\mu_h}{d\mu_{g_i}}  \to \frac{d\mu_h}{d\mu_{g}}$ in $L^\infty \cap W^{1,2}_\loc$ for the compact region. On the exterior region, we can estimate
\begin{equation}
    \left|\llangle R[g_i], (1-\chi) \frac{d\mu_h}{d\mu_{g}} - \frac{d\mu_h}{d\mu_{g_i}}  \rrangle\right| \leq \| \frac{d\mu_h}{d\mu_{g}} - \frac{d\mu_h}{d\mu_{g_i}} \|_{L^\infty} |dR[g_i]|(M\setminus K),
\end{equation}
which converges to zero by $\frac{d\mu_h}{d\mu_{g_i}} \to \frac{d\mu_h}{d\mu_g}$ in $L^\infty$ the uniform bound of $|dR[g_i]|(M \setminus K)$. 

Finally, since $F_i \approx g^{-1}(\partial g_i)^2$ schematically, we can write $F_i-F$ in terms of $(\partial g_i-\partial g)(|\partial g_i|+  |\partial g|)$  multiplied by uniformly bounded coefficients, and terms involving $(g_i^{-1}-g^{-1})$ multiplied by quadratic expressions in $\partial g_i$ and $\partial g$, hence $F_i-F \to 0$ in $L^1$ by combining H\"older's inequality, $g_i \to g \in L^\infty$, and $\partial g_i \to \partial g$ in $L^2$. 
\end{proof}

Recall that, for smooth metrics, if $g_i \to g$ in $W^{2,q}_{-\tau}$, $q>n$ and $R[g_i] \to R[g]$ in $L^1$ on an \textit{exterior coordinate chart}, then $\madm(g_i) \to \madm(g)$ (see \cite[Lemma 3.35]{L:19}). By \Cref{lemma:mass_cont}, if $g$ is smooth but only $C^0_{-\tau} \cap W^{1,p}_{-\tau}$ asymptotically flat, $p\geq2$, and $g_i$ are smooth as well, with $g_i \to g$ in $L^\infty$ and $\partial g_i \to \partial g$ in $L^2$ \textit{globally}, the same conclusion holds for $\madm(g)$, if well-defined, by \Cref{cor:proof_ADMmass}, as long as the distributional scalar curvature converges weakly as measures outside a compact set and that $R[g_i]$ has uniformly bounded $L^1$-norm. This provides an alternative convergence criterion for the ADM mass, requiring lower regularity, integrability, and decay rate assumptions and relying only on the distributional and weak measure convergence of the scalar curvature on the exterior region.

\begin{corollary} \label{cor:miso_close_eps}
Let $(M^n, g)$ be a $C^0_{-\tau} \cap W^{1,p}_{-\tau}$ asymptotically flat manifold, with $p\ge 2$ and $\tau>(n-2)/2$. Suppose that $R[g]$ is a finite signed measure outside a compact set. Then, there exists $\eps_0>0$ such that $\misoc(g_\eps)$ is finite for all $\eps<\eps_0$ and
$$
 \misoc(g_\eps) \longrightarrow \misoc(g) \text{ as } \eps \to 0.
$$
\end{corollary}
\begin{proof}
Notice that the sequence $g_\eps$ satisfies the hypothesis of \Cref{lemma:mass_cont}, since $g_\eps \to g$ in $L^\infty \cap W^{1,p}_{-\tau}$ by \Cref{lemma:FamilyOfMetrics_weighted}, $\|R[g_\eps]\|_{L^1}\leq C$ for some constant $C$ independent of $\eps$, and $R[g_\eps]\,d\mu_h \rightharpoonup R[g]$ weakly as signed measures outside a compact set by \Cref{lemma:conv_eps_sc_curv_L1}.
\end{proof}

\subsection{Analysis of the conformal Laplacian} {\label{sec:AnalysisConfLap}}

The main result of this section is the construction of a controlled solution to a certain Schr\"{o}dinger equation on $W^{1,p}_{-\tau}$ asymptotically flat manifolds. This is essential for handling the negative scalar curvature introduced after mollifying the metric $g$ and interpolating $g_\eps$ with the Euclidean metric in order to prove the density theorem, \Cref{theo:density}. A careful tracking of the constants is necessary to apply the same result to $g_\eps$ and to show that the resulting bounds can be chosen uniformly in $\eps$. This is carried out in the present section. The presentation follows \cite{LLR2022}. For convenience we begin by collecting standard properties of the Laplacian for asymptotically flat manifolds.

\subsubsection{Properties of the Laplacian for asymptotically flat manifolds}\label{sec:PropLap}
The theory of elliptic operators on $W^{1,p}_{-\tau}$ asymptotically flat manifolds, with $p>n$, is presented in a particularly clear form in \cite{Bartnik1986}, which will serve as our main reference throughout this work. We recall the definition of an elliptic operator $P$ asymptotic to the Euclidean Laplacian $\Delta$ at rate $\tau$ and collect two fundamental results that show that if $P$ is a selfadjoint operator then it is Fredholm with the same index as $\Delta$.
\begin{definition}\label{def:Deltaasymp_decayrate}
    The operator $u \mapsto Pu$ defined by
    \begin{equation}
    Pu = a^{ij}(x) \partial_{ij}^2 u + b^i(x)\partial_iu + c(x)u    
    \end{equation}
    on $\RR^n$ is said to be asymptotic to $\Delta$ (at rate $\tau$) if there exists $n<p<\infty$ and $\tau\geq0$ and constants $C_1$, $\lambda$ such that
    \begin{equation}\label{eq:g_uniform_elip}
    \lambda^{-1} |\xi|^2 \leq a^{ij}(x) \xi_i \xi_j\leq \lambda |\xi|^2 \quad \text{for all} \quad x\in \RR^n, \xi \in \RR^n     
    \end{equation}
    \begin{equation}\label{eq:C1_def}
    \|a^{ij}-\delta_{ij}\|_{W^{1,p}_{-\tau}} + \|b^i\|_{L^{p}_{-\tau-1}} + \|c\|_{L^{p/2}_{-\tau-2}}\leq C_1,    
    \end{equation}
    where $\delta_{ij}$ is the usual flat metric on $\RR^n$.
\end{definition}
    It is straightforward to generalize this definition to $W^{1,p}_{-\tau}$ asymptotically flat manifolds. The conditions above are satisfied by the divergence-form operator $\Delta_g$ when $g_{ij}$ is uniformly elliptic metric on $\RR^n$ and $(g_{ij}-\delta_{ij}) \in W^{1,p}_{-\tau}$ for some $\tau\geq0$. Standard weighted estimates then imply that if an operator $P$ is asymptotic to $\Delta$, the map $P \colon W^{2,q}_{s} \to L^q_{s-2}$ is bounded for all $1\leq q \leq p/2$ and $s \in \RR$. The restriction on $q$ is a consequence of the weak assumption on\footnote{Note that Bartnik assumes that $q\leq p$. However, if $c \in L^{p/2}$ and $q>p/2$, it is not necessarily true that $Pu \in L^q_{s-2}$ for every $u \in W^{2,q}_s$. This does not affect applications to the Laplacian, for which $c=0$.} $c$, but it still allows us to work in $W^{2,p/2}_s$, whose elements are continuous since $p>n$. We say that a decay rate $s \in \RR$ is exceptional if $ s \in \mathbb{Z} \setminus(2-n, 0)$, which correspond to the orders of growth of harmonic functions in $\RR^n \setminus \{0\}$.

We first consider the Laplacian on an asymptotically flat manifold (see \cite[Proposition 2.2]{Bartnik1986}). 
\begin{proposition}\label{prop:Lap_iso}
Let $(M,g)$ be a $W^{1,p}_{-\tau}$ asymptotically flat manifold with $p>n$, $\tau>0$. Then, for any $1<q\leq p/2$ and $s \in \RR$ not exceptional,
$$
\Delta_g \colon W_s^{2, q} \rightarrow L_{s-2}^{q}
$$
is Fredholm. Moreover, $\Delta_g$ is an isomorphism if $2-n<s<0$. 
\end{proposition}

Similar reasoning extends to more general elliptic operators (see \cite[Corollary 2.3]{Bartnik1986}).

\begin{proposition}\label{prop:P_elliptic_index}
Let $(M, g)$ be a $W^{1,p}_{-\tau}$ asymptotically flat manifold with $p>n$, $\tau>0$. Let $P$ be a formally selfadjoint elliptic operator and asymptotic to $\Delta$ at rate $\tau$. Then, for any $1<q\leq p/2$ and $s \in \RR$ not exceptional,
$$
P \colon W_s^{2, q} \rightarrow L_{s-2}^{q}
$$
is Fredholm, and its Fredholm index equals that of the Euclidean Laplacian on the corresponding weighted spaces.
\end{proposition}

As further preparation let us state the standard Sobolev inequality for $C^0$ metrics that are uniformly equivalent to Euclidean space.

\begin{lemma}\label{SobolevIneq}
    Let $M^n$ be a smooth manifold with a complete Riemannian metric $g \in C^0$, and a topological structure of infinity $(\mathcal{E},\Phi)$. Assume that $g$ is uniformly equivalent to $\delta$ on $\mathcal{E}$. Then there exists a constant $C(g)>0$, depending only on $g$, such that
\begin{equation}
    \left(\int_M |u|^{\frac{2n}{n-2}}d\mu_g \right)^{\frac{n-2}{n}}\leq C(g) \int_M |du|_g^2d\mu_g,
\end{equation}
    for every $u \in C^\infty_c(M)$. By a density argument, it then holds for every $u \in W^{1,2}(M)$. Moreover, if $\overline{g}$ is uniformly equivalent metric to $g$ with constant $\Lambda>0$, then
\begin{equation}
\frac{1}{\Lambda^n} C(\overline{g}) \leq C(g) \leq \Lambda^n C(\overline{g})    .
\end{equation}
\end{lemma}
\begin{proof}
  The proof for the first statement can be found in \cite[Lemma 3.1]{PMTSchoenYau}, and the proof also holds for $C^0$ metrics that are uniformly equivalent to the Euclidean metric by inspection. The second result is obtained simply by multiplying the Sobolev constant by a power of the uniform bound $\Lambda$. A similar argument can be found in \cite[Proposition 4.10]{grant2014positive}.
\end{proof}

\subsubsection{Analysis of the Schrödinger equation}

Let us now turn to the analysis of the operator $-\Delta_g+V$ for smooth potentials $V\in L^{p/2}_{-\tau-2}(M)$ with negative part in $L^{n/2}(M)$. The main result of this section is the following, which is similar to \cite[Proposition 2.12]{LLR2022} and the results in Section 2 therein. The proof relies on the results and constants presented in \Cref{sec:PropLap}.

\begin{proposition}\label{prop:Isom}
    Let $(M^n, g)$ be a $W^{1,p}_{-\tau}$ asymptotically flat manifold, with $p>n$ and $0<\tau<n-2$. Let $V \in L^{p/2}_{-\tau-2}(M)$ and denote by $V_-$ its negative part. Assume that $\|V_-\|_{L^{\frac{n}{2}}(M, g)} < C(g)^{-1}$, where $C(g)$ is the constant from \Cref{SobolevIneq}. Then the operator
    \begin{equation}
    -\Delta_g +V \colon W^{2,p/2}_{-\tau}(M) \to L^{p/2}_{-\tau-2}(M)    
    \end{equation}
    is an isomorphism.
    For every $f \in L^{p/2}_{-\tau-2}$, let $w \in W^{2,p/2}_{-\tau}(M)$ be the solution of
    \begin{equation}
    -\Delta_g w + V w = f.
    \end{equation}
    Let $\Lambda>0$ and let $h$ be a smooth background metric such that $\Lambda^{-1} h\leq g \leq \Lambda h$. There exist constants $\delta_0>0$ and $C>0$, depending on $n, p, \tau, h, \Lambda$, and the operator norm of $\Delta_g^{-1}$, such that if
\begin{equation}
    \|V_-\|_{L^{n/2}} + \|V\|_{L^{p/2}_{-\tau-2}} <\delta_0,
\end{equation}
then
\begin{equation}\label{eq:improved_gradient}
          \|w\|_{L^\infty_{-\tau}} + \|w\|_{W^{1,p}_{-\tau}} + \|w\|_{W^{2,p/2}_{-\tau}} \leq C \|f\|_{L^{p/2}_{-\tau-2}}.
\end{equation}
Moreover, $\delta_0$ and $C$ can be chosen uniformly for metrics that are sufficiently close to $g$ in $W^{1,p}_{-\tau}$.

If $f$ can be written as $f = \operatorname{div}_g X + \mathcal{Q}$, where $X \in C^1_c(TM)$ and $\mathcal{Q} \in C^0_c(M)$, then 
        \begin{equation}\label{eq:gradient_estimate_L2}
            \|\nabla^{g} w\|_{L^{2}(M, g)}  \leq  \frac{\left(\|X\|_{L^2(M,g)} + \sqrt{C(g)} \|\mathcal{Q}\|_{L^{\frac{2n}{n+2}}(M, g)}\right)}{1- C(g) \|V_-\|_{L^{n/2}(M, g)}}.
        \end{equation}
\end{proposition} 
\begin{proof}
    Since $g$ is  $W^{1,p}_{-\tau}$ asymptotically flat with $p>n$ and $V \in L^{p/2}_{-\tau-2}(M)$, the operator $\Delta_g+V$ is asymptotic to $\Delta$ at rate $\tau$ (see \Cref{def:Deltaasymp_decayrate}). By \Cref{prop:P_elliptic_index}, the operator is Fredholm of index zero and injectivity is equivalent to surjectivity, so it suffices to show the operator has trivial kernel. Recall that the weighted spaces $W^{2,p/2}_{-\tau}$ and $L^{p/2}_{-\tau}$ defined with respect to $g$ and $h$ are equivalent (see the discussion after \Cref{def:AFmfd}). Suppose $w \in W^{2,p/2}_{-\tau}(M)$ satisfies
    \begin{align}
    -\Delta_g w + V w= 0.    
    \end{align}
     Multiplying by $w$ and integrating by parts (justified below) yields
    \begin{align*}
    \int_M |\nabla^{g}w|_g^2 \,d\mu_g & = -\int_MVw^2 \,d\mu_g\\
    & \leq \int_M V_-w^2 \,d\mu_g\\
    & \leq \left(\int_M |V_-|^{\frac{n}{2}} \,d\mu_g\right)^\frac{2}{n} \left(\int_M |w|^{\frac{2n}{n-2}} \,d\mu_g\right)^\frac{n-2}{n}.
    \end{align*}
    Invoking \Cref{SobolevIneq}, we obtain
    \begin{align}
        \int_M |\nabla^g w|_g^2 \,d\mu_g \leq C(g) \left(\int_M |V_-|^{\frac{n}{2}} \,d\mu_g\right)^\frac{2}{n}  \int_M |\nabla^gw|_g^2 \,d\mu_g.
    \end{align}
    Since $\|V_-\|_{L^{\frac{n}{2}}(M,g)} < C(g)^{-1}$, it follows that $w\equiv 0$ on $M$.

    We now justify the integration by parts. To this end, we first improve the decay rate of $w$. The weighted Sobolev embedding gives $w \in L^\infty_{-\tau}$. Together with $V \in L^{p/2}_{-\tau-2}$, we have
    $$
    \Delta_g w = V w \in L^{p/2}_{- 2 \tau - 2 }.
    $$
    In particular, $\Delta_g:W^{2,p/2}_{-s}(M) \to L^{p/2}_{-s-2}(M)$ is an isomorphism for any $0<s<(n-2)$ (\Cref{prop:Lap_iso}). The improved decay of the right-hand side, together with uniqueness, implies that $w \in W^{2,p/2}_{-s}$ whenever $s \leq 2 \tau$ and $s<n-2$. If $0<\tau\leq(n-2)/2$, iterating the argument yields $w \in W^{2,p/2}_{-\gamma}$ for some $(n-2)/2<\gamma<n-2$. If $\tau>(n-2)/2$, we simply take $\gamma=\tau$.

    Since $p>n$, the weighted Morrey inequality implies that $w \in C^{0,\alpha}_{-\gamma}$ for some $\alpha \in (0,1)$ and the weighted Sobolev embedding gives $\nabla^g w \in W^{1,p}_{-\gamma}$. This fact implies that $w\Delta_gw$ and $|\nabla w|^2_g$ belong to $L^{p/2}_{-2\gamma-2}(M) \subset L^1(M)$, since $\gamma>(n-2)/2$. Let $\chi_R$ be a smooth cutoff function such that $\chi_R=1$ on $B_R$, $\chi=0$ on $M \setminus B_{2R}$, $0\leq \chi_R \leq 1$, and $|\nabla \chi_R | \leq C/R$ for some constant $C$. Testing the equation with $\chi_R w$ gives
    \begin{equation}
         \int_M \chi_R |\nabla^{g}w|_g^2 \,d\mu_g  = -\int_M \chi_RVw^2 \,d\mu_g - \int_M w \langle \nabla \chi_R, \nabla w \rangle_g\, d\mu_g,
    \end{equation}
    where the second integrand on the right-hand side is in $L^{p/2}_{-2 \gamma-2}\subset L^1$ and is supported in $B_{2R} \setminus B_R$. Taking the limit and applying the dominated convergence theorem yields the desired equality. Therefore, integrating by parts is justified and this concludes the first part of the proof.

Let $f \in L^{p/2}_{-\tau-2}$, and let $w \in W^{2,p/2}_{-\tau}$ satisfying $\Delta_g w = V w - f$. Since $\Delta_g:W^{2,p/2}_{-\tau}(M) \to L^{p/2}_{-\tau-2}(M)$ is an isomorphism,  there exists a constant $C_{\Delta_g^{-1}}>0$ such that 
\begin{equation}
    \|w\|_{W^{2,p/2}_{-\tau}} \leq C_{\Delta_g^{-1}} \| \Delta_g w\|_{L^{p/2}_{-\tau-2}} =  C_{\Delta_g^{-1}} \| Vw - f\|_{L^{p/2}_{-\tau-2}},
\end{equation}
where the constant $C_{\Delta_g^{-1}}=\|\Delta_g^{-1}\|_{\mathcal{L}(L^{p/2}_{-\tau-2}, W^{2,p/2}_{-\tau})}$ is the operator norm of the inverse Laplacian $\Delta_g^{-1}$ and depends on $g$, $p$, and $\tau$. The map $g \mapsto \Delta_g$ is continuous in the operator norm with respect to the $W^{1,p}_{-\tau}$ topology and invertibility is stable under small perturbations, so $C_{\Delta_g^{-1}}$ can be chosen uniformly for metrics that are sufficiently $W^{1,p}_{-\tau}$ close to $g$. Then, we can estimate

\begin{align}\label{eq:LaplacianEstimate}
 \|w\|_{W^{2,p/2}_{-\tau}} \leq C_{\Delta_g^{-1}} \|Vw - f\|_{L^{p/2}_{-\tau-2}} & \leq   C_{\Delta_g^{-1}} \left(\|w\|_{L^\infty}\|V\|_{L^{p/2}_{-\tau-2}} + \|f\|_{L^{p/2}_{-\tau-2}}\right).
\end{align}

Since $n-k(p/2)<0$ for $k=2$ and $p>n$, the weighted Sobolev embedding implies the weighted $L^\infty_{-\tau}$ bound for $w$, 
\begin{equation}\label{eq:weighted_Linfinity}
\|w\|_{L^\infty} \leq     \|w\|_{L^\infty_{-\tau}} \leq C  \|w\|_{W^{2,p/2}_{-\tau}},
\end{equation}
for some constant $C>0$ that additionally depends on $\|h\|_{L^\infty}$ and $\|h-\delta\|_{W^{1,p}_{-\tau}}$ (see the discussion after \Cref{def:AFmfd}). 

Let $C(h)$ be the Sobolev constant of $h$ (\Cref{SobolevIneq}). Since $\|V_-\|_{L^{n/2}(M,g)} \leq \Lambda \|V_-\|_{L^{n/2}}$ and $C(g) \leq \Lambda^n C(h)$, taking $\delta_0<\frac{1}{2}\min\{\frac{1}{\Lambda^{n+1} C(h)}, \frac{ 1}{C C_{\Delta_g^{-1}}}\}$, the term containing $\|V\|_{L^{p/2}_{-\tau-2}}$ in \eqref{eq:LaplacianEstimate} can be absorbed into the left-hand side. It follows that there exists a constant $C>0$ such that
\begin{align}
 \|w\|_{W^{2,p/2}_{-\tau}} \leq C  \|f\|_{L^{p/2}_{-\tau-2}},
\end{align}
where $C=C(n, p,  \tau, \Lambda, C(h), h, g)$, and $C$ can be chosen uniformly for metrics that are sufficiently $W^{1,p}_{-\tau}$ close to $g$.

Moreover, by the weighted Sobolev embedding with $p>n$, we estimate
\begin{equation}
    \|w\|_{W^{1,p}_{-\tau}} \leq C   \|w\|_{W^{2,p/2}_{-\tau}},
\end{equation}
where the constant $C>0$ does not depend on any additional parameters. Combining everything, we have that
\begin{equation}
     \|w\|_{L^\infty_{-\tau}} + \|w\|_{W^{1,p}_{-\tau}} + \|w\|_{W^{2,p/2}_{-\tau}}\leq  C \|f\|_{L^{p/2}_{-\tau-2}},
\end{equation}
where $C=C(n, p,  \tau, \Lambda, C(h), h, g)$, and $C$ can be chosen uniformly for all metrics sufficiently close to $g$ in $W^{1,p}_{-\tau}$.

  Assume that $f$ is given by $f = \operatorname{div}_g X + \mathcal{Q}$, where $X \in C^1_c(TM)$ and $\mathcal{Q} \in C^0_c(M)$, and $w$ solving
   \begin{equation}
       -\Delta_{g} w + V w = f.
   \end{equation}
By the decay improvement argument used in the proof of the triviality of the kernel, the compact support of $f$ implies that  $w \in W^{2,p/2}_{-\gamma}$ for some $(n-2)/2<\gamma<(n-2)$. Then, we multiply both sides by $w$, integrate over $M$, and integrate by parts to obtain

\begin{align}\label{eq:weak_form_ConfLap}
    \int_M |\nabla^{g} w|^2  \dg= -\int_M \left(V w^2 + \langle X, \nabla^g w\rangle_g - \mathcal{Q} w \right)  \dg.
\end{align}
Let $n^* = \frac{2n}{n-2}$. It follows from the above equality and H\"{o}lder's inequality that 
\begin{align*}
\int_M |\nabla^{g} w|^2  \dg & \leq  \int_{M} V_- w^2 - \langle X, \nabla^g w\rangle_g + \mathcal{Q} w  \dg\\
&  \leq  \|V_-\|_{L^{n/2}(M, g)} \|w\|_{L^{n^*}(M, g)}^2+  \|X\|_{L^{2}(M, g)} \|\nabla^g w\|_{L^{2}(M, g)}\\
&  + \|\mathcal{Q}\|_{L^{\frac{2n}{n+2}}(M, g)} \|w\|_{L^{n^*}(M, g)}.
\end{align*}
The Sobolev inequality (\Cref{SobolevIneq}) yields the following estimate 
\begin{align}
    \|w\|_{L^{n^*}(M, g)}^2& \leq C(g) \|\nabla^{g} w\|_{L^{2}(M, g)}^2.
\end{align}
Combining the last two inequalities, we derive
\begin{align}
     \|\nabla^{g} w\|_{L^{2}(M, g)}^2 & \leq C(g) \|V_-\|_{L^{n/2}(M, g)} \|\nabla^{g} w\|_{L^{2}(M, g)}^2 \\
     & + \left(\|X\|_{L^2(M,g)} + \sqrt{C(g)} \|\mathcal{Q}\|_{L^{\frac{2n}{n+2}}(M, g)}\right) \|\nabla^{g} w\|_{L^{2}(M, g)}.
\end{align}
Rearranging the inequality, we conclude
\begin{equation}\label{eq:n_star_boundf}
     \|\nabla^{g} w\|_{L^{2}(M, g)}  \leq  \frac{\left(\|X\|_{L^2(M,g)} + \sqrt{C(g)} \|\mathcal{Q}\|_{L^{\frac{2n}{n+2}}(M, g)}\right)}{1- C(g) \|V_-\|_{L^{n/2}(M, g)}}.
\end{equation}
\end{proof}

From the above result, we can construct a conformal metric of $g$ satisfying the standard properties.

\begin{proposition}\label{prop:ConfReescProp}
     Let $(M^n, g)$ be a $W^{1,p}_{-\tau}$ asymptotically flat manifold, with $p>n$ and $\tau \in (\frac{n-2}{2}, n-2)$. Let $V \in L^{p/2}_{-\tau-2}(M)$ and denote by $V_-$ its negative part. Assume that $\|V_-\|_{L^{\frac{n}{2}}(M, g)} < \frac{c_n}{C(g)}$, where $C(g)$ is the constant from \Cref{SobolevIneq} and $c_n \colonequals 4\frac{n-1}{n-2}$. Then there exists a globally defined positive function $u>0$ on $M$, with  $u-1 \in W^{2,p/2}_{-\tau}$, satisfying
    \begin{equation}\label{eq:conf_PDE}
    -c_n\Delta_gu +Vu = 0.
    \end{equation}
    Moreover, the conformally rescaled metric $\widetilde{g} = u^{\frac{4}{n-2}}g$ is $W^{1,p}_{-\tau}$ asymptotically flat. Assuming that $\misoc(g)$ is finite and $V \in L^1(M)$, the isoperimetric mass of $\widetilde{g}$ is given by
    \begin{equation}\label{eq:conf_misoc}
    \misoc(\widetilde{g})=\misoc(g) - \frac{1}{2(n-1)\omega_{n-1}}\int_M V u \,d\mu_g    
    \end{equation}
     and
\begin{equation}\label{eq:Sc_conformalmetric}
     R[\widetilde{g}] = (R[g] - V)u^{-\frac{4}{n-2}} \quad \text{ in } \D^\prime(M).
\end{equation}
Furthermore, there exists a constant $\delta_0>0$, depending on $n, p, \tau, h, \Lambda$, and the operator norm of $\Delta_g^{-1}$, such that if 
\begin{equation}
    \|V_-\|_{L^{\frac{n}{2}}}+ \|V\|_{L^{p/2}_{-\tau-2}}< \delta_0,
\end{equation}
then $\frac{1}{2}\leq u\leq \frac{3}{2}$ on $M$.
\end{proposition}
\begin{proof}
    Write $w = u-1$, then solving \eqref{eq:conf_PDE} is equivalent to finding a solution of
    \begin{align}
    -c_n\Delta_g w +Vw &= -V,    \\
    w &\in W^{2,p/2}_{-\tau},
    \end{align}
    where $c_n= 4 \frac{n-1}{n-2}$. By \Cref{prop:Isom} and the hypothesis $\|V_-\|_{L^{\frac{n}{2}}(M, g)} < \frac{c_n}{C(g)}$, the operator $-c_n\Delta_g +V \colon W^{2,p/2}_{-\tau} \to L^{p/2}_{-\tau-2}$ is an isomorphism, and hence there exists a unique solution $w \in W^{2,p/2}_{-\tau}$. In particular, $w \in C^{0,\alpha}_{-\tau}$, $\alpha \in (0,1)$ by the Sobolev embedding with $k=2$ and the weighted Morrey inequality.
    
    The positivity proof of $u$ follows as in \cite[Lemma 3.3]{PMTSchoenYau}. Let $E \colonequals \{ x \in M \colon u<0\}$, which is precompact since $u$ is asymptotic to one. Let $u_-$ be the negative part of $u$. If $E$ is not empty, a density argument allows us to use $u_- \in W^{1,p}_\loc$ as a test function in the weak formulation of \eqref{eq:conf_PDE}, then 
    \begin{align*}
        \|\nabla^{g} u\|_{L^2(E, g)}^2 = -\frac{1}{c_n} \int _E V u_-^2 \,d\mu_{g}  \leq  \frac{1}{c_n} \int _E V_- u^2 \,d\mu_{g} \leq  \frac{C(g)}{c_n} \|V_-\|_{L^{n/2}(E,g)}  \|\nabla^{g} u\|_{L^2(E, g)}^2,
    \end{align*}
    which is a contradiction since $\frac{C(g)}{c_n}\|V_-\|_{L^{n/2}(E,g)}<1$. This implies that $u\geq0$, and positivity of $u$ follows from the strong maximum principle with potential $V \in L^{p/2}_\loc$, $p>n$ (see \cite[Theorem 8.19]{GilbargTrudinger2001} and remark at the end of Section 8.10).
     
    Define the conformal metric $\widetilde{g} = u^{\frac{4}{n-2}}g$. Applying the chain rule, the interpolation inequality, and weighted H\"{o}lder's inequality to $u^\frac{4}{n-2}=(1+w)^\frac{4}{n-2}$ we see that $(1+w)^\frac{4}{n-2}-1 \in W^{2,p/2}_{-\tau}$. The $W^{1,p}_{-\tau}$ asymptotic flatness of $\widetilde{g}$ is obtained in a similar fashion by computing $\widetilde{g}_{ij}-\delta_{ij}$.

    Let $B_r$ denote the coordinate ball of radius $r$ and $\sigma_{ij} = g_{ij}-\delta_{ij}$. Using the weighted Morrey inequality for $w$ and the decay rate of $\sigma_{ij}=g_{ij}-\delta_{ij}$, we have
 \begin{align*}
    \frac{1}{r^2}\int_{B_r\setminus B_R}\partial_k\widetilde{g}_{ij} x^\ell \,dB_r
    & = \frac{1}{r^2}\int_{B_r\setminus B_R}  \left(\frac{4}{(n-2)}\partial_kw \delta_{ij} + \partial_k \sigma_{ij}\right) x^\ell \,dB_r + o(1),
 \end{align*}  
    where the terms involving powers of $w$ and $\partial\sigma$ were dealt with analogously to those in the proof of \Cref{prop:EquivMasses}. 
    
    Then, by \Cref{theo:miso_div}, we obtain
    \begin{equation}
        \misoc(\widetilde{g})(r) = \misoc(g)(r) - \frac{4}{(n-2) \omega_{n-1} r^2} \int_R^r \rho \int_{S_\rho} \frac{\p u}{\p \rho}  \,dS_\rho \,d\rho + o(1).
    \end{equation}
    Applying the divergence theorem and  $\eqref{eq:conf_PDE}$, then
    \begin{equation}
    \misoc(\widetilde{g})(r) = \misoc(g)(r) - \frac{1}{(n-1) \omega_{n-1} r^2} \int_R^r \rho \int_{B_\rho} V u  \,d\mu_g \,d\rho + o(1).
    \end{equation}
    Since $Vu$ is integrable and $\misoc(g)$ is finite, taking the limit as $r \to \infty$ gives
    \begin{equation}
    \misoc(\widetilde{g}) = \misoc(g) - \frac{1}{2(n-1) \omega_{n-1}} \int_M V u \,d\mu_g.
    \end{equation}
    \Cref{eq:Sc_conformalmetric} follows from the standard formula for the scalar curvature under a conformal change of metric and by approximating $g$ by smooth metric $g_\eps$.

    For the last claim, choosing $\delta_0$ small enough ensures that $\sup_M|w|\leq \frac{1}{2}$ by \Cref{prop:Isom}. It then follows that $u=1+w$ has positive upper and lower bounds, namely $u(x) \in \left[\frac{1}{2},\frac{3}{2}\right]$ for all $x \in M$. 
    
\end{proof}

\subsection{Proof of the density theorem, \Cref{theo:density}} \label{sec:ProofDensityTheo}

Let $\{g_\eps\}_{\eps>0}$ be a smooth family of metrics obtained by mollifying globally. Let $\chi$ be a smooth nonnegative cut-off function on $\RR^n$ such that $\chi \equiv 1$ on $B_1$ and vanishes outside $B_2$. Define $\chi_\lambda(x) \colonequals \chi(x/\lambda)$, for $\lambda\geq1$. Define the following interpolation between the metric $g_\eps$ and the Euclidean metric $\delta$ on $M$ by
    \begin{equation}\label{eq:def_g_lambdaeps}
    g_{\eps,\lambda} \colonequals \chi_\lambda g_\eps + (1-\chi_\lambda)\delta,    
    \end{equation}
    such that $g_{\eps,\lambda}=g_\eps$ for $|x|<\lambda$, $g_{\eps,\lambda} = \delta$ for $|x|>2\lambda$, and $g_{\eps,\lambda}$ interpolates between the two in the annular region in between.  This interpolation is necessary for our argument since the second derivatives of $g_\eps$ do not decay sufficiently fast to apply the smooth positive mass theorem at the end of the argument.
    
     The regularisation process introduces pointwise negative scalar curvature and changes the mass. Moreover, $\misoc(g_{\eps,\lambda})=0$ since the metric is Euclidean outside a compact set. The goal of this section is to conformal rescale the metric $g_{\eps,\lambda}$ in order to obtain a metric with nonnegative scalar curvature, isoperimetric mass close to $g$, and with better decay rate than $g$ and $g_\eps$.

    For every $\eps>0$ consider the conformal Laplacian equation
    \begin{align*}
    -c_n \Delta_{\eps,\lambda} u_{\eps,\lambda} + V_{\eps,\lambda} u_{\eps,\lambda} & = 0  \text{ in } M,\\
        u_{\eps,\lambda}  &\to 1  \text{ on } \mathcal{E},
    \end{align*}
where $\Delta_{\eps,\lambda}$ refers to the Laplacian of the metric $g_{\eps,\lambda}$ and the potential
\begin{equation}\label{eq:V_def}
   V_{\eps,\lambda}=-\chi_\lambda R[g_{\eps}]_-+\left(R[g_{\eps,\lambda}]-\chi_\lambda R[g_{\eps}]\right).
\end{equation}
We define $w_{\eps,\lambda}\colonequals u_{\eps,\lambda}-1$ and then solve
\begin{align*}
-c_n \Delta_{\eps,\lambda} w_{\eps,\lambda} + V_{\eps,\lambda} w_{\eps,\lambda} & = -V_{\eps,\lambda}  \text{ in } M,\\
    w_{\eps,\lambda}  & \in W^{2,p/2}_{-\tau^\prime},
\end{align*}
as in \Cref{prop:Isom,prop:ConfReescProp}, where $\tau>\tau^\prime>0$. Note that the second proposition can be applied since $V_{\eps,\lambda}$ is compactly supported. The hypothesis of these results are checked below using the following two lemmas. The only term in $V_{\eps,\lambda}$ \eqref{eq:V_def} that presents an additional difficulty is $R[g_{\eps,\lambda}]-\chi_\lambda R[g_{\eps}]$, which is handled in \Cref{lemma:AnnularDecay} and requires the smaller decay rate $\tau^\prime$. On the other hand, $R[g_{\eps}]_-$ can be controlled by the improved convergence obtained from weighted Friedrichs' Lemma in \Cref{prop:ConvofRs_weighted}. First, we show an auxiliary result from \cite{LLR2022}.

\begin{lemma} \label{lemma:dyadic}
Let $f\in L^p_{-s}$, where $p\geq1$ and $s>0$. Let $A_i$ denote the dyadic annulus $2^i\le |x|\le 2^{i+1}$. For any $q \in [1,p]$ and $i$ sufficiently large there exists a constant $C$ independent of $f$ and $i$ such that
\[\|f\|_{L^q(A_i)}\le C2^{-\eta i}\|f\|_{L^p_{-s}},\]
where $\eta = s-\frac nq$.
\end{lemma}
\begin{proof}
 The following proof follows \cite[Lemma 3.1]{LLR2022}. Since $p\geq q$, we can use H\"older's inequality to estimate
 \begin{align*}
     \left(\int_{A_i}|f|^q\,dx\right)^{\frac pq} &\le \Vol(A_i)^{\frac{p}{q}-1}\int_{A_i}|f|^p\,dx \le C (2^i)^{n(\frac pq-1)} \int_{A_i}|f|^p\,dx.
 \end{align*}
Let $s>0$. We decompose the prefactor as 
 \[(2^i)^{n(\frac pq-1)}=\left((2^{i})^{\frac nq-s}\right)^p \cdot \left((2^{i})^{s}\right)^p(2^i)^{-n},\]
then
  \begin{align*}
     (2^i)^{n(\frac pq-1)} \int_{A_i}|f|^p\,dx &=\left((2^{i})^{\frac nq-s}\right)^p\int_{A_i}|(2^{i})^{s}f|^p\,\frac{dx}{(2^i)^n}\\
     &\le C \left((2^{i})^{\frac nq-s}\right)^p \int_{A_i}||x|^{s}f|^p\,\frac{dx}{|x|^n}\\
     &\leq C (2^{-i\eta})^p \|f\|_{L^p_{-s}}^p,
 \end{align*}
 where $\eta = s-\frac nq$
\end{proof}

\begin{lemma}\label{lemma:AnnularDecay}
 Let $(M^n, g)$ be a $W^{1,p}_{-\tau}$ asymptotically flat manifold, $p>n$, $n\geq3$ and $\tau>\frac{n-2}{2}$. Let $\lambda_i=2^i$. For any fixed $\eps>0$ sufficiently small and any $\tau^\prime$ satisfying $0<\tau^\prime<\tau$ and $\tau^\prime< 2\tau-1$, there exists $I$ such that for all $i\geq I$, one has
\begin{equation}
\|R[g_{\eps, \lambda_i}] - \chi_{\lambda_i} R[g_\eps]\|_{L^{\frac{n}{2}} \cap  L^{p/2}_{-\tau^\prime-2}}\leq \frac{C}{\eps} 2^{-\eta i} \to 0 \text{ as } \lambda_i \to \infty,
\end{equation}
where $C$ is independent of $\eps$ and $\eta>0$ depending on $g$, $\tau$, $\tau^\prime$, $p$, and $n$. In particular, if $n\geq4$, then $\tau^\prime$ can be chosen so that $(n-2)/2<\tau^\prime < \tau$.
\end{lemma} 
\begin{proof}
Let $\lambda_i = 2^i$. To avoid excessive subscripts, we write $g_{\eps,i} = g_{\eps,\lambda_i}$ and $\chi_i \equiv \chi_{\lambda_i}$, and denote by $A_i$ the annular region $\{2^i\le |x|\le 2^{i+1}\}$. Since $R[g_{\eps,i}] - \chi_i R[g_\eps] \in C^\infty_c(M)$ is supported on $A_i$, we apply \Cref{lemma:dyadic} with $q=\frac{n}{2}$, $s = \tau^\prime +2$, with $\tau^\prime \in (0, \tau)$, yielding
\begin{align}
\|R[g_{\eps,i}] - \chi_i R[g_\eps]\|_{L^\frac{n}{2} }\leq C 2^{-\tau^\prime i}\|R[g_{\eps,i}] - \chi_i R[g_\eps]\|_{L^{p/2}_{-\tau^\prime-2}}.        
\end{align}
Therefore, it is sufficient to show that
\begin{equation}
\|R[g_{\eps,i}] - \chi_i R[g_\eps] \|_{L^{p/2}_{-\tau^\prime-2}}\leq \frac{C}{\eps} 2^{-\eta i} \longrightarrow 0 \quad \text{ as } i \to \infty,
\end{equation}
for $\tau^\prime<\tau$, some constant $C>0$ independent of $i$ and $\eps$, and some $\eta>0$.

We now claim that, in the difference $R[g_{\eps,i}] - \chi_i R[g_\eps]$, the leading second-order terms cancel, leaving
\begin{equation}\label{eq:scalar_curv_interp}
R[g_{\eps,i}] - \chi_i R[g_\eps] = O( |x|^{-\tau})  \partial^2 g_\eps +  O(1)(\chi_i^2-\chi_i)(\partial g_\eps)^2+ O (|x|^{-1})\partial g_\eps + O(|x|^{-\tau-2}),
\end{equation}
where each term on the right-hand side is supported in $A_i$, and the implicit constants in the $O(\cdot)$ notation are uniform in $\eps$ and $i$.

First, we establish uniform pointwise decay estimates for $g_{\eps,i}$. In asymptotically flat coordinates,
\begin{equation}\label{eq:gi_minus_delta}
g_{\eps,i} - \delta = \chi_i (g_{\eps}-\delta).   
\end{equation}
By the weighted Sobolev embedding and the uniform weighted estimate for $g_\eps$, we have a uniform constant $C_g$ depending on $g$, for all $\eps$ sufficiently small and $i$ sufficiently large, such that
\begin{equation}\label{eq:Decay_gi}
|g_{\eps,i}-\delta| \leq C_g |x|^{-\tau}.    
\end{equation}
From \eqref{eq:gi_minus_delta}, the derivatives of $g_{\eps,i}$ can be expressed as
\begin{equation}\label{eq:g_lambda_1dev_eq}
\partial g_{\eps,i} = \partial \chi_i (g_{\eps}-\delta) + \chi_i \partial g_{\eps},    
\end{equation}
and note that $|\partial \chi_i| \approx \lambda_i^{-1} \approx |x|^{-1}$ on $\operatorname{spt}(\partial \chi_i)=A_i$.

Recall that by \eqref{eq:CoordScalarC}, for any smooth metric, we can write $R[g]$ schematically as
\begin{equation}
R[g] = \partial V+ Q,    
\end{equation}
where $V = g^2*\partial g$ and $Q = (1+g)*g*(\partial g)^2$. Here, we have suppressed the distinction between the metric $g$ and its inverse $g^{-1}$, writing both schematically as $g$. This is justified because $g^{-1}$ and $\partial g^{-1}$ satisfy the same decay estimates and convergence properties as $g$ and $\partial g$, respectively.

Note that derivatives of $V$ can be written schematically as 
\begin{align}
    \partial V & = \partial (g^2*\partial g) 
     = g^2 * \partial^2 g + g *(\partial g)^2.
\end{align}
For the metric $g_{\eps,i}$, we denote $R[g_{\eps,i}]=\partial V_{\eps,i} + Q_{\eps,i}$. It is clear that $g_\eps-g_{\eps,i} = O(|x|^{-\tau})$, uniformly in $\eps$ and in $i$, which also holds for the inverses, and it is supported on $\operatorname{spt}(1-\chi_i)$. Using the decay and support of the terms $g_\eps-g_{\eps,i}$, we can estimate $\partial V_{\eps,i}$ by
\begin{align*}
        \partial V_{\eps,i} & =   g^2_i * \partial^2 g_{\eps,i} + g_{\eps,i} * (\partial g_{\eps,i})^2\\
        & = g^2_\eps * \partial^2 g_{\eps,i} +  g_\eps * (\partial g_{\eps,i})^2+ O(|x|^{-\tau})  (\partial^2 g_{\eps,i} + (\partial g_{\eps,i})^2),
\end{align*}
where the lower order terms are only supported on $A_i$, since $\partial g_{\eps,i}$ are supported on $B_{2^{i+1}}$ and the terms $O(|x|^{-\tau})$ are supported outside $B_{2^i}$.

For the second derivatives, from \eqref{eq:g_lambda_1dev_eq} and $|\partial \chi_i| \approx \lambda_i^{-1} \approx |x|^{-1}$, we have that
\begin{align*}
    \partial^2 g_{\eps,i} & = \partial (\partial\chi_i (g_\eps-\delta) + \chi_i \partial g_\eps))\\
    & = \partial^2 \chi_i (g_\eps-\delta) + 2 \partial \chi_i \partial g_\eps + \chi_i \partial^2 g_\eps\\
    &  = \chi_i \partial^2 g_\eps + O (|x|^{-1})\partial g_\eps + O(|x|^{-\tau-2}),
\end{align*}
where the lower order terms are only supported in $A_i$ since $\operatorname{spt}(\partial \chi_i)=A_i$. The quadratic term can be written as
\begin{align*}
    (\partial g_{\eps,i})^2 & =  (\partial\chi_i (g_\eps-\delta) + \chi_i \partial g_\eps)^2\\
    & =(\partial\chi_i (g_\eps-\delta))^2 + 2 \chi_i \partial\chi_i (g_\eps-\delta) \partial g_\eps + (\chi_i \partial g_\eps)^2\\
    & =(\chi_i \partial g_\eps)^2 +  O(|x|^{-1-\tau}) \partial g_\eps + O(|x|^{-2\tau-2}).
\end{align*}
Combining both, we obtain
\begin{align*}
        \partial V_{\eps,i} & = \chi_i \partial V_\eps +  (\chi_i^2-\chi_i)(\partial g_\eps)^2+ O(|x|^{-\tau})(\partial^2 g_{\eps,i} + (\partial g_{\eps,i})^2)+O(|x|^{-1})\partial g_\eps  +  O(|x|^{-\tau-2}),
\end{align*}
where $\partial V_\eps \colonequals g^2_\eps\partial^2 g_\eps +  g_\eps  (\partial g_\eps)^2$. Rewriting only in terms of $g_\eps$, we have
\begin{align}
        \partial V_{\eps,i} - \chi_i \partial V_\eps & = O(|x|^{-\tau})\partial^2 g_\eps + (\chi_i^2-\chi_i) (\partial g_\eps)^2+O(|x|^{-1})\partial g_\eps + O(|x|^{-\tau-2}),
\end{align}
where each term on the right-hand side is supported in $A_i$.

A similar computation for $Q_{\eps,i}$ reveals that
\begin{equation}
Q_{\eps,i} - \chi_i Q_\eps=  O(1)(\chi_i^2-\chi_i)(\partial g_\eps)^2+ O(|x|^{-1-\tau})\partial g_\eps + O(|x|^{-2\tau-2}),    
\end{equation}
where $Q_\eps = (1+g_\eps)*g_\eps*(\partial g_\eps)^2$, where $O(1)$ can be bounded by a constant uniformly in $\eps$ and $i$. Combining the two previous equations yields \eqref{eq:scalar_curv_interp} and concludes the claim.

What remains is to establish that all these terms converge to zero in $L^{p/2}_{-\tau^\prime-2}$ as $i \to \infty$. We start with the first term in \eqref{eq:scalar_curv_interp}. From \Cref{eq:Weighted_mollifier}, for any $f \in L^p_{-\tau}$, it holds that 
\begin{align}
\|f*\rho_\eps\|_{L^p_{-\tau}}& \leq C\| f\|_{L^p_{-\tau}},\\
\|\partial (f*\rho_\eps)\|_{L^p_{-\tau}}& \leq \frac{C}{\eps}\| f\|_{L^p_{-\tau}}, \label{eq:weighted_eps_estimate}
\end{align}
for some $C>0$ depending on $\rho$ and $p$. We can apply this to the second-order derivatives as follows:
\begin{align*}
\||x|^{-\tau} \partial^2 g_\eps \|_{L^{p/2}_{-\tau^\prime-2}(A_i)} = \||x|^{-\tau+1} \partial^2 g_\eps \|_{L^{p/2}_{-\tau^\prime-1}(A_i)} 
& \leq 2^{-(\tau-1)i} \||x|^{\tau^\prime-\tau} \partial^2 g_\eps \|_{L^{p/2}_{-\tau-1}(A_i)} \\
& \leq C 2^{-(2\tau-1-\tau^\prime)i} \|\partial^2 g_\eps \|_{L^{p}_{-\tau-1}(A_i)} \|1\|_{L^{p}_{0}(A_i)} \\
& \leq\frac{C}{\eps} 2^{-(2\tau-1-\tau^\prime)i} \|\partial g \|_{L^{p}_{-\tau-1}},
\end{align*}
where we applied H\"{o}lder's inequality and that $\|1\|_{L^{p}_{0}(A_i)}^p\leq C \log(2)$ is bounded independently of $i$, and the last line follows from \eqref{eq:weighted_eps_estimate}.
Thus, this quantity converges to zero as $i \to \infty$ since $\tau^\prime< 2\tau-1$.

The second term quadratic in $\partial g_\eps$, since $\tau^\prime<\tau$ and $\partial g_\eps \in L^p_{-\tau-1}$, we have
$$
\|(\partial g_\eps)^2 \|_{L^{p/2}_{-\tau^\prime-2}(A_i)} = \||x|^{\tau^\prime-2\tau} (\partial g_\eps)^2 \|_{L^{p/2}_{-2\tau-2}(A_i)} \leq C 2^{-(2\tau-\tau^\prime)i} \| \partial g \|_{L^{p}_{-\tau-1}}^2 \to 0.
$$

For the third term, similarly, we have
\begin{equation}
\||x|^{-1} \partial g_\eps \|_{L^{p/2}_{-\tau'-2}(A_i)} = \||x|^{\tau^\prime-\tau} \partial g_\eps \|_{L^{p/2}_{-\tau-1}(A_i)} \leq C 2^{-(\tau-\tau^\prime)i} \| \partial g \|_{L^{p}_{-\tau-1}} \to 0,
\end{equation}
where again we applied H\"{o}lder's inequality.

For the fourth term, since $\tau>\tau^\prime$, we have
\begin{equation}
\||x|^{-\tau-2}\|_{L^{p/2}_{-\tau^\prime-2}(A_i)}^{p/2}\le C\int_{2^{i}}^{2^{i+1}}r^{-\frac{p}{2}(\tau-\tau^\prime)-1}\,dr\le C2^{-\frac{p}{2}(\tau-\tau^\prime)i}\to 0,   
\end{equation}
and putting everything together, we conclude the claim.
\end{proof}

Now, we restate and prove our density theorem (\Cref{theo:density}).

\densitythm*
\begin{proof}
    Let $\lambda_i=2^i$, and $\eps_i = 2^{-i\eta/2}$, where $\eta$ is given by \Cref{lemma:AnnularDecay}, and define $g_i \colonequals g_{\eps_i,\lambda_i}$ \eqref{eq:def_g_lambdaeps}, $\Delta_i \colonequals \Delta_{\eps_i,\lambda_i}$,  $\chi_i\colonequals \chi_{\lambda_{i}}$, and $V_i \colonequals V_{\eps_i,\lambda_i}=-\chi_{i} R[g_{\eps_i}]_-+\left(R[g_i]-\chi_{i} R[g_{{\eps_i}}]\right)$ \eqref{eq:V_def}. Clearly, there exists $I>0$ such that the operators $\Delta_{i}$ are asymptotic to the Euclidean Laplacian $\Delta$ at rate $\tau^\prime$ (see \Cref{def:Deltaasymp_decayrate}) with uniform bounds for all $i \geq I$.
      
According to \Cref{prop:ConfReescProp}, the problem boils down to showing that $V_i$ has sufficiently small $L^{n/2}$ and $L^{p/2}_{-\tau^\prime-2}$ norms. Note that $V_i \in L^p_{-\tau-2}$ because it is compactly supported and locally Lipschitz. Since $R[g]\geq0$ in $\D^\prime$,  $\tau^\prime>0$, and $p\geq n$, we can apply \Cref{prop:ConvofRs_weighted} to show 
\begin{align}\label{eq:Rnegative_convzero}
\|R[g_{\eps_i}]_-\|_{L^{\frac{n}{2}}\cap L^{p/2}_{-\tau^\prime -2}}\leq C \|R[g_{\eps_i}]_-\|_{L^{p/2}_{-2\tau^\prime -2}}\longrightarrow 0 \text{ as } i \to \infty,
\end{align}
as $L^{p/2}_{-2\tau^\prime-2}\subset  L^{n/2} \cap L^{p/2}_{-\tau^\prime-2}$. Analogously, by \Cref{lemma:AnnularDecay}, we see
\begin{equation}
\|R[g_i] - \chi_{\lambda_i} R[g_{\eps_i}]\|_{L^{\frac{n}{2}} \cap L^{p/2}_{-\tau^\prime-2}}\leq C 2^{-i{\eta}/{2}} \longrightarrow 0 \text{ as } i \to \infty.
\end{equation}
Combining the two previous inequalities, we have
\begin{equation}\label{eq:V_bound}
\|V_i\|_{L^{\frac{n}{2}} \cap L^{p/2}_{-\tau^\prime-2}}\longrightarrow 0 \text{ as } i \to \infty.
\end{equation}

Therefore, after increasing $I$, if necessary, $V_i$ satisfies the hypotheses of \Cref{prop:ConfReescProp} for every $i \geq I$. Consequently, for any $\tau^\prime$ satisfying $0<\tau^\prime<\tau$ and $\tau^\prime< 2\tau-1$, there exists a global solution $u_i$, with $w_i\colonequals (u_i-1)\in W^{2,p/2}_{-\tau^\prime}$ such that $\widetilde{g}_i = u_i^{4/(n-2)} g_i$ is a $W^{1,p}_{-\tau^\prime}$ asymptotically flat metric. The constants in \Cref{prop:ConfReescProp,prop:Isom} can be chosen independently of $i$, since $g_i \to g$ in $W^{1,p}_{-\tau^\prime}$. This convergence follows from $g_{\eps_i} \to g$  in $W^{1,p}_{-\tau}$ and an estimate later established in the proof (see \eqref{eq:gi_gepsi_W1p_estimate}).

Since $V_i$ is Lipschitz and $g_i$ smooth, elliptic regularity implies that $u_i$ is $C^{2,\alpha}$, for some $\alpha \in (0,1)$, and hence $\widetilde{g}_i \in C^{2,\alpha}$. Moreover, since $g_i=\delta$ and $V_i = 0$ outside a compact set, $u_i$ is Euclidean harmonic there, and thus $\widetilde{g}_i = u_i^{4/(n-2)} \delta$ is harmonically flat at infinity. By \cite[Corollary 1.9]{Bartnik1986}, $u_i$ admits the asymptotic expansion \eqref{eq:harmonic_expansion}. This expansion, together with the regularity of $u_i$ and $g_i$, yields $\widetilde{g}_i \in W^{2,p}_{-\tau^\prime}$.

 The properties of the scalar curvature in item (2) of \Cref{theo:density} are verified as follows. By the usual formula for the conformal Laplacian \eqref{eq:Sc_conformalmetric} and the definition of $V_i$ \eqref{eq:V_def}, we have 
\begin{equation}\label{eq:ScalarCurvFormProof}
    R[\widetilde{g}_i] = u_i^{-\frac{4}{n-2}} (R[g_i] - V_i) = u_i^{-\frac{4}{n-2}}\chi_i R[g_{\eps_i}]_+\geq 0,
\end{equation}
and $R[\widetilde{g}_i] \in L^1(M)$ because it is compactly supported. 

Now, we check the closeness of $R[\widetilde{g}_i]$ and $R[g_{\eps_i}]$ in $L^1$ in order to obtain the weak convergence of measure. By \eqref{eq:ScalarCurvFormProof}, the difference between the scalar curvatures is given by
\begin{align}
R[\widetilde{g}_i]-R[g_{\eps_i}]& =\chi_i u_i^{-\frac{4}{n-2}} R[g_{\eps_i}]_+-R[g_{\eps_i}]\\
& =\chi_i (u_i^{-\frac{4}{n-2}}-1) R[g_{\eps_i}]_++(\chi_i-1)R[g_{\eps_i}]_+ + R[g_{\eps_i}]_-.
\end{align}
As $R[g]\geq 0$ in $\D^\prime$ and $R[g]$ is a finite measure outside a compact set by hypothesis, \Cref{lemma:conv_eps_sc_curv_L1}, item (3), gives $\|R[g_{\eps_i}]\|_{L^1}\leq C$ uniformly in $\eps_i$. To control the first term, we combine this estimate with the bound $u_i \in (\frac{1}{2}, \frac{3}{2})$, as established in the proof of \Cref{prop:ConfReescProp}. Hence, 
\begin{align}
\int_M \left| \chi_i (u_i^{-\frac{4}{n-2}}-1) R[g_{\eps_i}]_+\right| \,d\mu_h&
\leq \|u_i^{-\frac{4}{n-2}}-1\|_{\infty} \int_M\chi_i  \left| R[g_{\eps_i}]_+\right| \,d\mu_h \\&\leq C  \|u_i-1\|_\infty,
\end{align}
which converges to zero by \Cref{prop:Isom} (see \eqref{eq:improved_gradient}) combined with the estimate \eqref{eq:V_bound}.
For the second term,
\begin{align}
\int_M |(\chi_i-1)R[g_{\eps_i}]_+| \,d\mu_h \leq \int_{|x|\geq\lambda_i}  |R[g_{\eps_i}]_+| \,d\mu_h,
\end{align}
which converges by the $L^1$ uniform bound of $R[g_{\eps_i}]$ using the estimate \eqref{eq:Unif_bound_Rgeps} in $\{|x|\geq \lambda_i\}$ as in \Cref{lemma:conv_eps_sc_curv_L1}. The third term converges to zero by \Cref{prop:ConvofRs_weighted} and noticing that $\tau>(n-2)/2$ implies $L^{p/2}_{-2\tau-2} \subset L^1$. Consequently, we have 
\begin{align}\label{eq:weakconvergence}
\int_M \left| R[\widetilde{g}_i]-R[g_{\eps_i}]\right|  \,d\mu_h \longrightarrow 0 \quad \text{ as } i \to \infty,
\end{align}
and, as in the proof of \Cref{lemma:conv_eps_sc_curv_L1}, this is sufficient to conclude that
\begin{equation}\label{eq:measure_convergence}
R[\widetilde{g}_i]\,d\mu_h \rightharpoonup dR[g]
\end{equation}
weakly as finite nonnegative Radon measures.

We begin the proof of item (3) by verifying the $W^{1,p}_{-\tau^\prime}$ closeness of $\widetilde{g}_i$ and $g$. For this, we have the following estimate
\begin{equation}\label{eq:W1p_close}
    \|\widetilde{g}_i-g\|_{W^{1,p}_{-\tau^\prime}} \leq \|\widetilde{g}_i-g_i\|_{W^{1,p}_{-\tau^\prime}} + \|g_{\eps_i}-g_i\|_{W^{1,p}_{-\tau^\prime}} + \|g_{\eps_i}-g\|_{W^{1,p}_{-\tau^\prime}}.
\end{equation}

By \Cref{lemma:FamilyOfMetrics_weighted}, it remains to estimate the first two terms on the right-hand side of \eqref{eq:W1p_close}. First, using that $|\chi_i-1| \leq 1$ and $|\partial\chi_i|\leq C |x|^{-1}$, we have
\begin{align}
    \|g_i-g_{\eps_i}\|_{W^{1,p}_{-\tau}} &\leq  C  \|g_{\eps_i}-\delta\|_{W^{1,p}_{-\tau}}\leq C,
\end{align}
uniformly in $\eps$ and $i$ by the usual mollifier properties.
 
Now, since $\operatorname{supp}(g_i - g_{\eps_i}) \subset \{|x|\geq \lambda_i\}$, we see
\begin{equation}\label{eq:gi_gepsi_W1p_estimate}
\|g_i - g_{\eps_i}\|_{L^p_{-\tau^\prime}}  = \||x|^{-\tau+\tau^\prime}(g_i - g_{\eps_i})\|_{L^p_{-\tau}}\leq 2^{-(\tau-\tau^\prime)i}  \|g_i - g_{\eps_i}\|_{L^p_{-\tau}}\leq C2^{-(\tau-\tau^\prime)i},
\end{equation}
where $C>0$ is independent of $\eps$ and $i$. A similar argument works for the first derivatives, and this concludes the closeness of $g_i$ and $g_{\eps_i}$ in $W^{1,p}_{-\tau^\prime}$.

Since $\widetilde{g}_i-g_i = \left(u_i^{\frac{4}{n-2}}-1\right) g_i$, $g_i$ is uniformly bounded, and $g_i \to g$ in $W^{1,p}_{-\tau^\prime}$, we have the following estimate, with $C>0$ independent of $i$,
\begin{equation}
    \|\widetilde{g}_i-g_i\|_{W^{1,p}_{-\tau^ \prime}}\leq C \|u_i-1\|_{W^{1,p}_{-\tau^ \prime}}\leq C \|V_i\|_{L^{p/2}_{-\tau^ \prime-2}}.
\end{equation}
The right-hand side converges to zero as $i \to \infty$ by the estimate \eqref{eq:V_bound}

The convergence $\widetilde{g}_i \to g$ in $W^{1,p}_{-\tau^\prime}$, with $p>n$ and $\tau^\prime>0$, is sufficient to show that $\widetilde{g}_i \to g$ in $L^\infty$ by the Sobolev embedding. In particular, if $\tau^\prime>(n-2)/2$ (and then $n\geq 4$), we obtain that $\partial \widetilde{g}_i \to \partial g$ in $L^p_{-1-\tau^\prime} \subset L^2$. This is no longer immediate if $\tau^\prime<(n-2)/2$, and we instead use the gradient estimate \eqref{eq:gradient_estimate_L2} from \Cref{prop:Isom}, together with the structure of the potential $V_i$. 

Note that the difference $R[g_i] - \chi_i R[g_i] $ can be written in divergence form by using the coordinate formula for the scalar curvature \eqref{eq:CoordScalarC} in asymptotically flat coordinates and then applying the product rule. Thus, we can write
\begin{equation}
R[g_i] - \chi_i R[g_{\eps_i}] = \operatorname{div}_{g_{i}} X_{i}  + \mathcal{Q}_{i}, 
\end{equation}
where $X_{i}$ is a compactly supported vector field containing the relevant terms involving $\partial g_{\eps_i}$ and $\partial g_{i}$, and $\mathcal{Q}_i$ is a compactly supported function consisting of quadratic terms, the remaining terms from the product rule, and the corrections needed to obtain the divergence. Moreover, $X_i$ and $\mathcal{Q}_i$ are supported in the annular region $A_i$. In particular, $\|X_i \|_{L^p_{-\tau-1}}$ and $\|\mathcal{Q}_i \|_{L^{p/2}_{-\tau-2}}$ are uniformly bounded in $i$ because they depend on $g_{\eps_i}$ and $g_i$, with $g_{\eps_i} \to g$ in $W^{1,p}_{-\tau}$ and $\|g_i-g\|_{W^{1,p}_{-\tau}}\leq C$ uniformly in $i$. 

By \eqref{eq:gradient_estimate_L2} and the smallness of $\|V_i\|_{L^{n/2}}$ \eqref{eq:V_bound}, we have
        \begin{equation}
            \|\nabla^{g_i} w_i\|_{L^{2}}  \leq  C \left(\|X_i\|_{L^2} +  \|\mathcal{Q}_i\|_{L^{\frac{2n}{n+2}}} + \|R[g_{\eps_i}]_-\|_{L^{\frac{2n}{n+2}}}\right),
        \end{equation}
for some constant $C>0$ independent of $i$. Applying \Cref{lemma:dyadic}, we obtain
        \begin{equation}
        \left(\|X_i\|_{L^2} +  \|\mathcal{Q}_i\|_{L^{\frac{2n}{n+2}}}\right) \leq C2 ^{-(\tau - \frac{n-2}{2})i}\left( \|X_i\|_{L^p_{-\tau-1}} +   \|\mathcal{Q}_i\|_{L^{p/2}_{-\tau-2}}\right).
        \end{equation}
Furthermore, $ \|R[g_{\eps_i}]_-\|_{L^{\frac{2n}{n+2}}} \leq C  \|R[g_{\eps_i}]_-\|_{L^{p/2}_{-2\tau-2}} \to 0$ as $i \to \infty$ by \Cref{prop:ConvofRs_weighted}. Because $\tau>(n-2)/2$, both powers of $2$ above converge to zero and the uniform weighted bounds imply
\begin{equation}
\|\nabla^{g_i} w_i\|_{L^{2}} \longrightarrow 0 \text{ as } i \to \infty,
\end{equation}
 and using this fact, $\partial g_i \to \partial g$ in $L^2$, and $w_i \to 0 $ in $L^\infty_{-\tau^\prime}$, we can show that $\partial \widetilde{g}_i \to \partial g$ in $L^2$.

Therefore, by \Cref{lemma:mass_cont}, $\misoc(\widetilde{g}_i) \to \misoc(g)$ as $i \to \infty$. Also note that, since $\widetilde{g}_i$ is a $W^{2,p}_{-\tau^\prime}$, $p>n$, asymptotically flat metric, harmonically flat outside a compact set, and $R[\widetilde{g}_i] \in L^1$, the ADM mass $\madm(\widetilde{g}_i)$ is well-defined and finite, and, by \Cref{cor:proof_ADMmass}, $\madm(\widetilde{g}_i) = \misoc(\widetilde{g}_i)$. Thus $\madm(\widetilde{g}_i) \to \misoc(g)$ as $i \to \infty$. It is a standard fact that $\madm(\widetilde{g}_i ) = 2 A_i$ for metrics that are conformal to the Euclidean metric, where $A_i$ is the parameter in the expansion \eqref{eq:harmonic_expansion}. The $\delta_i$-closeness as bilinear forms on $M$ in item (4) follows from the $L^\infty$ convergence of $\widetilde{g}_i $ to $g$.
\end{proof}

\begin{theorem}\label{theo:PMT_W1p}
    Let $(M^n, g)$ be a $W^{1,p}_{-\tau}$, $p>n$, asymptotically flat manifold with $\tau>(n-2)/2$. Assume $R[g]\geq0$ in $\D^\prime$ and $R[g]$ is a finite measure outside some compact set. Then $\misoc(g)\geq0$.
\end{theorem}
\begin{proof}
    Suppose that $\misoc(g)<0$. Then, we approximate $g$ by a smooth metric $\widetilde{g}$ as in \Cref{theo:density}, chosen such that $\madm(\widetilde{g})<0$. By construction, $\widetilde{g}$ has nonnegative scalar curvature and  is harmonically flat outside a compact. In particular, it satisfies the asymptotic assumptions required to apply the classical smooth positive mass theorem. However, this gives a contradiction. 
\end{proof}

\subsection{The rigidity case}\label{sec:Rigidity_W1p}

This section is devoted to proving the rigidity case for the isoperimetric mass, i.e., $\misoc(g) = 0$ implies that $(M^n, g)$ is isometric to Euclidean space $(\RR^n, \delta)$. The proof proceeds by first establishing scalar-flatness and then Ricci-flatness, before concluding flatness. At each step, we use the conclusion obtained in the preceding one. The overall strategy and methods are adapted from \cite{Hafemann2026}, with the individual steps modified slightly to accommodate $\misoc$ and the weaker asymptotic assumptions.

The first two steps rely on the improved scalar curvature estimates from \Cref{prop:ConvofRs_weighted}. Both conclusions are established by contradiction, if the curvature does not vanish, we perturb the metric and apply a conformal deformation to construct metrics satisfying the hypothesis of the low-regularity positive mass theorem (\Cref{theo:PMT_W1p}), but having negative isoperimetric mass. The Ricci-flatness argument follows the approach of Jiang-Sheng-Zhang \cite{Jiang2022}, and uses, as an essential input, the consequence of \Cref{prop:ConvofRs_weighted}, that $R[g_\eps] \to 0$ in $L^1$ whenever $R[g]=0$. This strategy is similar to the standard rigidity argument, but uses a positive mass theorem requiring weaker decay. 

Once distributional Ricci-flatness has been established, we use the equivalence between distributional Ricci curvature lower bounds and nonnegative Ricci curvature in the sense of $\RCD$ spaces established in \cite{mondino2025}. The volume comparison argument for $\RCD$ spaces employed in \cite{Jiang2022} then shows that $(M^n, d_g)$ is isometric to $(\RR^n, d_\delta)$. This metric space isometry can finally be upgraded to a Riemannian isometry by the low-regularity Myers-Steenrod theorem \cite{Matveev2017}. 

For completeness, we state and prove below the versions of the rigidity results in \cite{Hafemann2026} for metrics that need not be smooth outside a compact. Although many details closely follow that reference, we include the proofs to indicate the modifications required to prove the result for $\misoc$ and weaker decay assumption.

\begin{lemma}\label{lemma:ScalarLowerBound_W1p}
     Let $(M, g)$ be a $W^{1,p}_{-\tau}$ asymptotically flat manifold with $p>n$ and $\tau>(n-2)/2$. Let $g_{\eps,t} = g_\eps + t q$, where $q$ is a smooth compactly supported $(0,2)$ symmetric tensor and $g_\eps$ is the mollified metric. Consider functions $w,v \in C^{0}_{-\tau}$, with $dw, dv \in L^2(M)$ satisfying $w, v \in L^{\frac{2n}{n-2}}(M)$. Then, there exist $\eps_0>0$ and $t_0>0$, and constants $C_0, C_1,C_2>0$ independent of $\eps$ and $t$ such that for all $\eps \in (0, \eps_0)$ and $t \in [0, t_0)$, we have 
    \begin{equation}\label{eq:linear_control_W1p}
        \left| \int  R[g_{\eps,t}] w \, d\mu_{\eps,t} \right| \leq \left( C_0  + C_1 t +C_2 t^2 \right) \|\nabla^{g_{\eps,t}} w\|_{L^2(M, g_{\eps,t})} ,
    \end{equation}
    and
     \begin{equation} \label{eq:quadratic_control_W1p}
        \left| \int  R[g_{\eps,t}] w v \, d\mu_{\eps,t} \right| \leq \left( C_0  + C_1 t +C_2 t^2 \right) \|\nabla^{g_{\eps,t}} w\|_{L^2(M, g_{\eps,t})} \|\nabla^{g_{\eps,t}} v\|_{L^2(M, g_{\eps,t})},
    \end{equation}
   
    Moreover, if $R[g] = 0$, then, after possibly decreasing $\eps_0$ and $t_0$, there exists $A(\eps)\geq0$ satisfying $A(\eps) \to 0$ as $\eps \to 0$, such that $C_0$ may be replaced by $A(\eps)$ in both estimates above.
\end{lemma}
    \begin{proof}
        The proof of this theorem is a generalization of \cite[Lemma 4.5]{Hafemann2026} without assuming smoothness outside a compact set, but with a Sobolev decay rate instead. The proof can be followed almost line by line, except by replacing local arguments or estimates by global ones. We sketch the main steps and constant dependence, while the intermediate steps are detailed in \cite[Lemma 4.5]{Hafemann2026}.

        We begin with recalling the following variational formula for the scalar curvature \cite[Exercise 1.18]{L:19} (or \cite[Proposition 4]{BrendleMarques2011}),
        \begin{equation} \label{eq:VarScalarCurv_complete2_W1p}
            R[g_{\eps,t}] = R[g_\eps] + t\restr{DR}{g_\eps}(q) + t^2 B_{\eps,t}(q),
        \end{equation}
        where
        \begin{equation}\label{eq:VarScalarCurv2_W1p}
            \restr{DR}{g_\eps}(q) = \operatorname{div}_{g_\eps}(\operatorname{div}_{g_\eps}(q)) - \Delta_{g_\eps} \tr_{g_\eps}(q)  -  \langle \Ric[g_\eps] , q \rangle_{g_\eps},
        \end{equation}
        and $B_{\eps,t}(q)$ is the sum of a divergence term $\operatorname{div}_{g_\eps} U_{\eps,t}$, where $U_{\eps,t}$ is a compactly supported vector field obtained by contracting $q*\nabla^{g_\eps} q$ with $g^{-1}_\eps$ and $g^{-1}_{\eps,t}$, a term quadratic in $\nabla^{g_\eps} q$ and a term $\Ric[g_\eps]*q^2$ with analogous contractions.
        
         By a partition of unity argument, we may work on $M = \RR^n$. In these coordinates, we use the decomposition $R[g_\eps] = \partial_k V^k[g_\eps] + Q[g_\eps]$ from \eqref{eq:dist_scalar}. Although $w$ and $v$ need not to be compactly supported, the following identity holds
        \begin{equation}\label{eq:sc_wv}
            \int R[g_\eps] w v \,d\mu_{g_\eps} = \int -V^k[g_\eps] \partial_k(wv \sqrt{|g_\eps|}) + Q[g_\eps]  wv \sqrt{|g_\eps|} \, dx.
        \end{equation}
        To obtain this identity, we apply integration by parts to the compactly supported function $\chi_R wv$, where $\chi_R$ is the cutoff function introduced in \Cref{sec:ProofDensityTheo}. The resulting error term containing $\partial \chi_R$ is supported in an annular region and tends to zero as $R \to \infty$ by \Cref{lemma:dyadic}, since $V[g_\eps] \in L^p_{-\tau-1}$, $w,v \in C^{0}_{-\tau}$, and $d\chi_R=O(R^{-1})$. The argument applies with $\chi_R w$ in place of $\chi_R w v$.

        From a standard argument involving H\"{o}lder's inequality and the Sobolev inequality (\Cref{SobolevIneq}), we can show that 
       \begin{align}\label{eq:Scalar_UpperBound1_W1p}
           \left| \int  R[g_{\eps}] w v \, d\mu_{\eps}\right| \leq  C(\eps)  \|\nabla^{g_\eps} w\|_{L^2(M, g_\eps)} \|\nabla^{g_\eps} v\|_{L^2(M,g_\eps)},
    \end{align}
    and
    \begin{align}\label{eq:Scalar_UpperBound2_W1p}
        \left| \int  R[g_{\eps}] w \, d\mu_{\eps}\right| \leq  C(\eps)  \|\nabla^{g_\eps} w\|_{L^2(M, g_\eps)},
    \end{align}
    where
    \begin{equation}\label{eq:C_eps_def}
         C(\eps) \colonequals C \left(\sum_{k=1}^n\|V^k[g_\eps]\|_{L^n\cap L^2(M, g_\eps)} +  \|Q[g_\eps]\|_{L^\frac{n}{2} \cap L^\frac{2n}{n+2}(M, g_\eps)}\right),
    \end{equation}
    for some constant $C>0$ depending on the Sobolev constant $C(g_\eps)$ and $\|\partial {\sqrt{|g_\eps|}}\|_{L^n}$. For sufficiently small $\eps$, these quantities can be bounded uniformly in $\eps$, since $g_\eps \to g$ in $W^{1,p}_{-\tau}$. Moreover, the same convergence provides uniform bounds for the norms of $V^k[g_\eps]$ and $Q[g_\eps]$. Consequently, there exists a constant $C_0$ independent of $\eps$ such that $C(\eps)\leq C_0$.

    The remaining terms in \eqref{eq:VarScalarCurv_complete2_W1p} contain $q$ or its derivatives and are therefore compactly supported. They can be treated as in \cite[Lemma 4.5]{Hafemann2026}, using integration by parts, H\"{o}lder's inequality, and the Sobolev inequality. The Ricci term in \eqref{eq:VarScalarCurv2_W1p} can be estimated by the same argument used for the scalar curvature. Then, we obtain that
    \begin{equation}\label{eq:DistTermsQuad}
         t\left|\int_M\restr{DR}{g_\eps}(q)wv  d\mu_{g_\eps} \right|+ t^2 \left|\int_M B_{\eps,t}(q) wv  d\mu_{g_\eps} \right| \leq (C_1 t  + C_2 t^2)\|\nabla^{g_\eps} w\|_{L^2(M, g_\eps)} \|\nabla^{g_\eps} v\|_{L^2(M, g_\eps)},
    \end{equation}
    and
    \begin{equation}\label{eq:DistTermsLin}
         t\left|\int_M\restr{DR}{g_\eps}(q)w  d\mu_{g_\eps} \right|+ t^2 \left|\int_M B_{\eps,t}(q) w  d\mu_{g_\eps} \right| \leq (C_1 t  + C_2 t^2)\|\nabla^{g_\eps} w\|_{L^2(M, g_\eps)},
    \end{equation}
    where the constants $C_1$ and $C_2$ depend on the relevant $L^q$ norms, with $q$ as in \eqref{eq:C_eps_def}, of products involving $\partial q$ contracted with $g_\eps^{-1}$, and of $\Gamma[g_\eps]$ contracted with $q$. Since $q$ is compactly supported and $g_\eps \to g$ in $W^{1,p}_{-\tau}$, these norms are uniformly bounded for sufficiently small $\eps$. Then, $C_1$ and $C_2$ can be chosen independently of $\eps$.

    It is not hard to see that replacing $wv$ by $w v d\mu_{g_{\eps,t}}/d\mu_{g_{\eps}} = \left(wd\mu_{g_{\eps,t}}/d\mu_{g_{\eps}}\right) v $ is sufficient to conclude that the above estimates also hold with respect to $d\mu_{g_{\eps,t}}$ using the Sobolev inequality, since $g_{\eps,t} \to g$ in $W^{1,p}_{-\tau}$ as $t \to 0$ and $\eps \to 0$. Thus, taking $\eps$ and $t$ sufficiently small, we obtain
    \begin{equation}\label{eq:Scalar_CaseR0_W1p}
        \left| \int  R[g_{\eps,t}] w v \, d\mu_{g_{\eps,t}} \right| \leq \left( C_0  + C_1 t +C_2 t^2 \right) \|\nabla^{g_{\eps,t}} w\|_{L^2(M, g_{\eps,t})}\|\nabla^{g_{\eps,t}} v\|_{L^2(M, g_{\eps,t})},
    \end{equation}
    and
    \begin{equation}\label{eq:Scalar_CaseR1_W1p}
        \left| \int  R[g_{\eps,t}] w \, d\mu_{g_{\eps,t}} \right| \leq \left( C_0  + C_1 t +C_2 t^2 \right) \|\nabla^{g_{\eps,t}} w\|_{L^2(M, g_{\eps,t})}.
    \end{equation}
    where $C_0, C_1, C_2>0$ are constants independent of $t$ and $\eps$.

    Assume that $g$ is scalar flat, i.e., $R[g]=0$, then we have that $R[g] \star_M \rho_\eps = 0$. Applying H\"{o}lder's inequality and the Sobolev inequality, we obtain
    \begin{equation}
        \left|\int R[g_\eps] w v \,d\mu_{g_\eps} \right|\leq C(g_\eps)^2 \|R[g_\eps]\|_{L^{n/2}(M, g_\eps)}  \|\nabla^{g_\eps} w\|_{L^2(M, g_\eps)} \|\nabla^{g_\eps} v\|_{L^2(M,g_\eps)},
    \end{equation}
    and,
    \begin{equation}
        \left|\int R[g_\eps] w \,d\mu_{g_\eps} \right|\leq C(g_\eps) \|R[g_\eps]\|_{L^{\frac{2n}{n+2}}(M, g_\eps)}  \|\nabla^{g_\eps} w\|_{L^2(M, g_\eps)},
    \end{equation}
    where $C(g_\eps)$ is the Sobolev constant of $g_\eps$.
    
    Since $R[g]\star_M \rho_\eps = 0$ and $\tau>(n-2)/2$, we have that $\|R[g_\eps]\|_{L^q(M, g_\eps)} \to 0$ as $\eps \to 0$, for every $q \in [1,p/2]$, by \Cref{prop:ConvofRs_weighted}. Then, we have
    \begin{align} \label{eq:Scalar_UpperBound1_small_W1p}
           \left| \int  R[g_{\eps}] w v \, d\mu_{\eps}\right| \leq  A(\eps)  \|\nabla^{g_\eps} w\|_{L^2(M, g_\eps)} \|\nabla^{g_\eps} v\|_{L^2(M,g_\eps)},
    \end{align}
    and
        \begin{align} \label{eq:Scalar_UpperBound2_small_W1p}
           \left| \int  R[g_{\eps}] w \, d\mu_{\eps}\right| \leq  A(\eps)  \|\nabla^{g_\eps} w\|_{L^2(M, g_\eps)},
    \end{align}
    where
       \begin{equation}
         A(\eps)\colonequals C  \|R[g_\eps]\|_{L^{\frac{n}{2}} \cap L^{\frac{2n}{n+2}} (M, g_\eps)} ,
    \end{equation}
    for some constant $C>0$ uniformly in $\eps$, for $\eps$ sufficiently small, and $A(\eps) \to 0$ as $\eps \to 0$.

    Again, combining \eqref{eq:Scalar_UpperBound1_small_W1p} and \eqref{eq:DistTermsQuad} (or \eqref{eq:Scalar_UpperBound2_small_W1p} and \eqref{eq:DistTermsLin}) and choosing $\eps$ and $t$ smaller if necessary, we have
    \begin{equation}
        \left|\int  R[g_{\eps,t}] w v \, d\mu_{g_{\eps,t}} \right|\leq \left( A(\eps)  + C_1 t +C_2 t^2 \right) \|\nabla^{g_{\eps,t}} w\|_{L^2(M, g_{\eps,t})}\|\nabla^{g_{\eps,t}} v\|_{L^2(M, g_{\eps,t})},
    \end{equation}
    and
    \begin{equation}
        \left|\int  R[g_{\eps,t}] w \, d\mu_{g_{\eps,t}} \right|\leq \left( A(\eps)  + C_1 t +C_2 t^2 \right) \|\nabla^{g_{\eps,t}} w\|_{L^2(M, g_{\eps,t})}.
    \end{equation}
    where $C_1,C_2>0$ are constants independent of $t$ and $\eps$ and $A(\eps) \to 0$ as $\eps \to 0$.
    \end{proof}

    \begin{proposition} \label{prop:scalar_flat_W1p}
        Let $(M, g)$ be a $W^{1,p}_{-\tau}$ asymptotically flat manifold with $p>n$, $\tau>(n-2)/2$. Assume that $R[g] \geq 0 $ in $\D^\prime$. If $\misoc(g) = 0$, then $g$ is scalar-flat.
    \end{proposition}
    \begin{proof}
    A more detailed exposition of this proof can be found in \cite[Proposition 4.6]{Hafemann2026}, the proof idea and steps are essentially the same but using global results for $W^{1,p}_{-\tau}$ asymptotically flat obtained in this work.
    
    Let $g_\eps$ be the regularisation of $g$ via mollifiers. Using that $g_\eps \to g$ in $W^{1,p}_\loc\cap C^0_\loc$, we can prove that if $R[g]\geq 0$ in $\D^\prime$ and $R[g]\neq0$, then there exists a test function $\varphi \in C^\infty_c(M)$, $\varphi\geq0$, and constants $\eps_0, C_\varphi>0$, such that
        \begin{equation}\label{eq:ScalarCurv_lowerbound_W1p}
    \int R[g_\eps] \varphi d\mu_{g_\eps} \geq C_\varphi>0,
    \end{equation}
    for every $\eps \in (0, \eps_0)$.
    
    Consider the following potential,
    \begin{equation}
    V_{\eps,t}\colonequals  R[g_\eps]_+ t\varphi - R[g_\eps]_- =  R[g_\eps]  t\varphi - R[g_\eps]_-(1-t \varphi).
    \end{equation}
    Since $\varphi \geq 0$ and $L^{p/2}_{-2\tau-2} \subset L^{n/2}$, $p>n$ and $\tau>(n-2)/2$, it is clear that 
    \begin{equation}
    \|(V_{\eps,t})_-\|_{L^{n/2}}=\|R[g_\eps]_-\|_{L^{n/2}}\leq C \|R[g_\eps]_-\|_{L^{p/2}_{-2-2\tau}} \longrightarrow 0 \quad \text{ as } \eps \to 0,
    \end{equation}
     for all $0 \leq t\leq 1/\|\varphi\|_{L^\infty}$ by \Cref{prop:ConvofRs_weighted}. Then, the isomorphism condition in \Cref{prop:ConfReescProp} is satisfied for $V_{\eps,t}$ when $\eps$ is sufficiently small for any $0 \leq t\leq 1/\|\varphi\|_{L^\infty}$. Therefore, there exists $\eps_0>0$ such that for all $\eps \in (0, \eps_0)$ and $t \in [0, 1/\|\varphi\|_{L^\infty}]$, there exists a positive solution $u_{\eps,t}$ on $M$ of
    \begin{equation}
        -c_n\Delta_{g_\eps}u_{\eps,t} + V_{\eps,t} u_{\eps,t}  = 0,    
    \end{equation}
    such that $u_{\eps,t} -1 \in W^{2,p/2}_{-\tau}$.
    
    The conformally rescaled metric $\widetilde{g}_{\eps,t} = u_{\eps,t}^{\frac{4}{n-2}}g_\eps$ is a $W^{1,p}_{-\tau}$ asymptotically flat metric and its scalar curvature $R[\widetilde{g}_{\eps,t}] $ satisfies
    \begin{equation}
    R[\widetilde{g}_{\eps,t}] = u_{\eps,t}^{-\frac{4}{n-2}} (R[g_\eps]- R[g_\eps]_+ t\varphi + R[g_\eps]_-) = u_{\eps,t}^{-\frac{4}{n-2}} R[g_\eps]_+ (1 - t\varphi )\geq0.
    \end{equation}
    Since $V_{\eps,t} \in L^1(M)$, \Cref{prop:ConfReescProp} gives
    \begin{align}\label{eq:estimate_mass_W1p}
    \misoc(\widetilde{g}_{\eps,t})&  =  \misoc(g_\eps) - \frac{1}{2(n-1)\omega_{n-1}} \int_M V_{\eps,t} u_{\eps,t} d\mu_{g_\eps}.
    \end{align}
    From the weighted $L^p$ embeddings and \Cref{prop:ConvofRs_weighted}, we have that $\|R[g_\eps]_-\|_{L^q} \leq C \|R[g_\eps]_-\|_{L^{p/2}_{-2\tau-2}}$ which converges to zero as $\eps \to 0$, for every $q \in [1,p/2]$. Using this estimate, \Cref{lemma:ScalarLowerBound_W1p}, the Sobolev inequality, and the uniform boundedness of $g_\eps$, we can show that there exist $\eps_0>0$ and $t_0>0$ such that for all $\eps \in (0, \eps_0)$ and $t \in [0, t_0)$, we have
    \begin{equation}\label{eq:V_eps,t_lowerbound_W1p}
    \int_M V_{\eps,t} u_{\eps,t} d\mu_{g_\eps} \geq t \frac{C_\varphi}{2} - C \|R[g_\eps]_-\|_{L^{p/2}_{-2\tau-2}},
  \end{equation}
  where the constants $C_\varphi$ and $C$ are independent of $\eps$ and $t$.  Note that $\|\nabla u\|_{L^2}$ can be estimated by integrating by parts, similar to the proof of \Cref{prop:Isom}, and applying \Cref{lemma:ScalarLowerBound_W1p}.
  
  Therefore, combining the previous estimate for $V_{\eps,t}$ \eqref{eq:V_eps,t_lowerbound_W1p} and the estimate for $\misoc(\widetilde{g}_{\eps,t})$ \eqref{eq:estimate_mass_W1p}, we obtain 
    
    \begin{align}
   \misoc(\widetilde{g}_{\eps,t})&  \leq  \misoc(g_\eps) -\frac{1}{ 2(n-1)\omega_{n-1}}\left(  t \frac{C_\varphi}{2} - C\|R[g_\eps]_-\|_{L^{p/2}_{-2\tau-2}} \right),
    \end{align}
    for all $\eps \in (0, \eps_0)$ and $t \in [0, t_0)$. Since $\misoc(g_\eps) \to \misoc(g)=0$ as $\eps \to 0$ by \Cref{cor:miso_close_eps} and $\|R[g_\eps]_-\|_{L^{p/2}_{-2\tau-2}} \to 0$, we have that for any fixed positive $t \in (0, t_0)$, there exists $\eps_0>0$ such that $\misoc(\widetilde{g}_{\eps,t}) <0$ for all $\eps \in (0, \eps_0)$,  with $R[\widetilde{g}_{\eps,t}] \geq 0$, contradicting the positive mass theorem (\Cref{theo:PMT_W1p}). Hence, we have $R[g]=0$.
    \end{proof}

    \begin{proposition}\label{prop:Ricciflat_W1p}
        Let $(M, g)$ be a $W^{1,p}_{-\tau}$ asymptotically flat manifold with $p>n$, $\tau>(n-2)/2$. Assume that $R[g] \geq 0 $ in $\D^\prime$. If $\misoc(g) = 0$, then $g$ is Ricci-flat.
    \end{proposition}
    \begin{proof}
    Again, a more detailed exposition of this proof can be found in \cite[Proposition 4.7]{Hafemann2026}.  By \Cref{prop:scalar_flat_W1p}, we have $R[g]=0$ in $\D^\prime$. Assume that $\Ric[g]\neq 0$ in $\D^\prime$. Without loss of generality, there exists a $C^0 \cap W^{1,p}$ regular compactly supported $(0,2)$-symmetric tensor $q$ and a sequence of smooth compactly supported $(0,2)$-symmetric tensor $q_\eps$ such that $q_\eps \to q$ in $C^0_\loc \cap W^{1,p}_\loc$, such that 
    \begin{equation} \label{eq:Riccibound_W1p}
        \int_M \langle \Ric[g_\eps], q_\eps\rangle_{g_\eps} d\mu_{g_\eps} \leq -C_q<0,
    \end{equation} 
    for some constant $C_q>0$ independent of $\eps$ for all $\eps$ sufficiently small. We may replace $q$ by $-q$ if necessary.

     Define $g_{\eps,t} = g_\eps + t q_\eps$, and note that $g_{\eps,t}$ is a compactly supported perturbation of $g_\eps$, since $q_\eps$ is compactly supported, and $g_{\eps,t}$ remains a Riemannian metric uniformly in $\eps$ and $t$, for $\eps,t>0$ sufficiently small. 

    Although we do not have a uniform $L^{n/2}$ bound for $R[g_{\eps,t}]$ needed to apply \Cref{prop:ConfReescProp} uniformly in $\eps$ and $t$, we can use \Cref{lemma:ScalarLowerBound_W1p} (case $R[g]=0$) to conformally rescale $g_{\eps,t}$ to a scalar-flat metric through the following claim.
    
    Claim: For any $\delta>0$, there exist $\eps_0, t_0>0$ such that for all $t \in (0,t_0)$ and $\eps \in (0, \eps_0$), there is a positive solution $u_{\eps,t}>0$ of 
    \begin{equation}
    -c_n\Delta_{g_{\eps,t}}u_{\eps,t} + R[g_{\eps,t}] u_{\eps,t}  = 0,    
    \end{equation}
    satisfying $(u_{\eps,t}-1) \in W^{2,p/2}_{-\tau}$ and $\|\nabla^{g_{\eps,t}}u_{\eps,t}\|_{L^2(M, g_{\eps,t})} < \delta$.
     
    The proof of this claim essentially repeats the proof of \Cref{prop:Isom} and \Cref{prop:ConfReescProp}, but using \Cref{lemma:ScalarLowerBound_W1p} instead of $L^p$ estimates. An exposition for the proof can be found in \cite[Proposition 4.7]{Hafemann2026}, where in our setting, we use \Cref{lemma:ScalarLowerBound_W1p} instead of \cite[Lemma 4.5]{Hafemann2026}.

    From this claim, as in \Cref{prop:ConfReescProp}, we have that the conformally rescaled metric $\widetilde{g}_{\eps,t} = u^{\frac{4}{n-2}}_{\eps, t} g_{\eps,t}$ is $W^{1,p}_{-\tau}$ asymptotically flat, and its scalar curvature is given by
    \begin{equation}
    R[\widetilde{g}_{\eps,t}] = u_{\eps,t}^{-\frac{4}{n-2}} ( R[g_{\eps,t}] - R[g_{\eps,t}])= 0, 
    \end{equation}
    i.e., scalar flat, and the isoperimetric mass is given by
    \begin{align}
    \misoc(\widetilde{g}_{\eps,t})&  =\misoc(g_{\eps,t}) - \frac{1}{ 2(n-1)\omega_{n-1}} \int_M R[g_{\eps,t}]  u_{\eps,t}   \,d\mu_{g_{\eps,t}}\\
    &  = \misoc(g_{\eps}) - \frac{1}{ 2(n-1)\omega_{n-1}}\int_M R[g_{\eps,t}]  u_{\eps,t}   \,d\mu_{g_{\eps,t}},
    \end{align}
    where we have used that $\misoc(g_{\eps,t}) = \misoc(g_\eps)$ since both metrics agree outside a compact set.

    Applying the variational formula for the scalar curvature \eqref{eq:VarScalarCurv_complete2_W1p} again to estimate the integral $R[g_{\eps,t}]$ together with \Cref{lemma:ScalarLowerBound_W1p} and the curvature bound \eqref{eq:Riccibound_W1p}, we have
    \begin{align*}
        \int_{M} R[g_{\eps,t}] u_{\eps,t} \,d\mu_{g_{\eps,t}} & =  \int_{M} R[g_{\eps,t}] \,d\mu_{g_{\eps}}  +  \int_{M} R[g_{\eps,t}] (u_{\eps,t}\frac{d\mu_{g_{\eps,t}}}{d\mu_{g_{\eps}}}-1) \, d\mu_{g_{\eps}} \\
        &\geq \left(\int_{M}  R[g_{\eps}] \,d\mu_{g_{\eps}}+  \frac{C_q}{2} t - C_3 t^2\right) - \left(\left( A(\eps)  + C_1 t +C_2 t^2 \right) \|\nabla^{g_{\eps,t}} u_{\eps,t}\|_{L^2}\right),
    \end{align*}
    where $A(\eps) \to 0$ as $\eps \to 0$, and $C_1,C_2,C_3$ are positive constants independent of $\eps$ and $t$, for $\eps$ and $t$ sufficiently small. Note that \Cref{lemma:ScalarLowerBound_W1p} was applied to $q_\eps$ instead of a fixed compactly supported tensor, however, the relevant estimates can be taken uniformly in $\eps$ since $q_\eps \to q$ in $C^0_\loc \cap W^{1,n}_\loc$. A more detailed derivation can be found in \cite{Hafemann2026}.
  
   Therefore, by the claim, for any $\delta>0$, there exist $\eps_0, t_0>0$, such that for any $\eps \in (0, \eps_0)$, $t \in (0,t_0)$, we have
    \begin{equation}
     \misoc(\widetilde{g}_{\eps,t}) \leq  \misoc(g_{\eps}) -\frac{1}{ 2(n-1)\omega_{n-1}}\left(\int_{M}  R[g_{\eps}] \,d\mu_{g_{\eps}}+  \frac{C_q}{4}t - A(\eps)\delta  - (C_1+C_2)t  \delta\right).
    \end{equation}
    Taking $\delta =\frac{C_q }{8 (C_1+C_2)}$, we have
    \begin{equation}
         \misoc(\widetilde{g}_{\eps,t}) \leq  \misoc(g_{\eps}) -\frac{1}{ 2(n-1)\omega_{n-1}}\left( \int_{M}  R[g_{\eps}] \,d\mu_{g_{\eps}} +   \frac{C_q}{8} t -  A(\eps)  \frac{C_q }{8 (C_1+C_2)}\right).
    \end{equation}
    Since $R[g] = 0$ implies that $R[g] \star_M \rho_\eps=0$, and $\tau>(n-2)/2$, we have that $R[g_\eps] \to 0$ in $L^1$ by \Cref{prop:ConvofRs_weighted}. \Cref{cor:miso_close_eps} implies that $\misoc(g_\eps) \to \misoc(g)=0$. Therefore, for any fixed $t \in (0,t_0)$, there exists $\eps_0>0$ such that for all $\eps \in (0, \eps_0)$ we obtain $\misoc(\widetilde{g}_{\eps,t}) < 0$, with $R[\widetilde{g}_{\eps,t}] = 0$, contradicting the low-regularity positive mass theorem (\Cref{theo:PMT_W1p}). Then, $\Ric[g]=0$.
     \end{proof}

    \begin{theorem}\label{theo:flat}
         Let $(M, g)$ be a $W^{1,p}_{-\tau}$ asymptotically flat manifold with $p>n$, $\tau>(n-2)/2$. Assume that $R[g] \geq 0 $ in $\D^\prime$. If $\misoc(g) = 0$, then there exists a $C^{1,1-\frac{n}{p}}$ isometry between $(M,g)$ and $(\RR^n, \delta)$.
\end{theorem}
\begin{proof}
    The results follow directly from the argument used in the proof of \cite[Theorem 4.11]{Hafemann2026} (similar to \cite{Jiang2022}). The argument requires only $C^0_{-\tau}$ asymptotic flatness and distributional Ricci flatness to conclude that $(M^n, d_g)$ is isometric to $(\RR^n, d_\delta)$ as a metric space. The proof relies crucially on equivalence between distributional Ricci curvature lower bounds and $\RCD$ spaces proved in \cite[Corollary 7.7]{mondino2025}, together with volume comparison arguments for $\RCD$ spaces (see \cite[Corollary 1.7]{Gigli2018}). Finally, the low-regularity version of Myers-Steenrod, \cite[Corollary C]{Matveev2017}, implies that there exists a $C^{1,\alpha}$ isometry, $\alpha=1-n/p$, from $(M, g)$ and $(\RR^n, \delta)$.    
\end{proof}

\appendix
\section{Appendix}\label{append:A}
In this section, we work on $\RR^n$, which is sufficient for extending the results to manifolds. In particular, the weighted norms are defined as in \Cref{sec:preliminaries}, i.e., the radial function is the identity on $B_R$. We also assume that $R=1$. We denote by $\partial_j$ the partial derivative $\frac{\partial}{\partial x^j}$, where $j=1,\ldots,n$, and by $\rho_\eps$ the standard mollifier, i.e., $\rho_\eps(x) =\rho(\frac{x}{\eps})/{\eps^n}$, where $\eps \in (0,1]$ and $\rho: \RR^n \to \RR$ is a smooth positive mollifier supported in the unit ball with unit integral. For any $p \in [1, +\infty)$, we define the \textit{conjugate exponent} $p^\prime$ of $p$ by $\frac{1}{p}+\frac{1}{p^\prime}=1$, with the usual conventions for $p=1$ and $p=+\infty$.

Let $h \in \RR^n$. The translation operator along $h$ is defined for any $f \in L^p_s(\RR^n)$ by
\begin{equation}
    \tau_h(f)(x)\colonequals \left(x \mapsto f(x-h)\right).
\end{equation}
\begin{lemma}\label{lemma:translation_weighted}
    Let $f \in L^{p}_s(\RR^n)$, with $p \in [1,\infty)$, $s \in \RR$. Then for every  $h \in \RR^n$ with $|h|\leq 1$, the translation operator satisfies
    \begin{equation}
        \|\tau_h(f)\|_{L^{p}_s(\RR^n)} \leq C \|f\|_{L^{p}_s(\RR^n)},
    \end{equation}
    for some $C>0$ depending only on $p$ and $s$. In particular, $\tau_h: L^{p}_s(\RR^n) \to L^{p}_s(\RR^n)$ is a bounded linear continuous operator, uniformly for $|h|\leq1$.
\end{lemma}
\begin{proof}
    From the definition, we have
        \begin{align*}
        \| \tau_h(f)\|_{L^p_s(\RR^n)}^p & = \left(\| \tau_h(f)\|_{L^p(B_1)} + \|\tau_h(f)\|_{L^p_s(\mathcal{E})}\right)^{p} \\
         & \leq 2^{p-1} \int_{\RR^n} \left|f(x-h) \right|^{p} \sigma(x)^{-sp-n}dx,
    \end{align*}
where $\mathcal{E}= \RR^n \setminus B_1$ and $\sigma(x) \colonequals \max\{1, |x|\}$. A standard computation shows that there exists $C_1 > 1$, only depending on $p$ and $s$,  such that 
\begin{equation}\label{eq:compare_weights}
   \frac{1}{C_1} (1+|x|)^{-sp-n} \leq \sigma(x)^{-sp-n} \leq C_1 (1+|x|)^{-sp-n}.
\end{equation}
Then, we can estimate the norm by
        \begin{align}
        \| \tau_h(f)\|_{L^p_s(\RR^n)}^p \leq  2^{p-1}C_1 \int_{\RR^n} \left|f(x-h) \right|^{p}  (1+|x|)^{-sp-n}dx.
    \end{align}
Consider the change of coordinates $z = x-h$, thus
\begin{align}
    \| \tau_h(f)\|_{L^p_s(\RR^n)}^p \leq 2^{p-1} C_1 \int_{\RR^n} \left|f(z) \right|^{p}  (1+|z+h|)^{-sp-n} dz.
\end{align}
Since $|h|\leq1$, we have that 
\begin{equation}
    (1+|z+h|)^{-sp-n} \leq C (1+|z|)^{-sp-n},
\end{equation}
with $C=2^{|sp+n|}$. The above estimate together with \eqref{eq:compare_weights} yields
    \begin{align*}
        \| \tau_h (f)\|_{L^p_s(\RR^n)}^p & \leq  2^{p-1}C C_1 \int_{\RR^n} \left|f(z) \right|^{p}  (1+|z|)^{-sp-n} dz\\
         & \leq  2^{p-1}C C_1^2 \int_{\RR^n} \left|f(z) \right|^{p}  \sigma(z) ^{-sp-n} dz\\
        & \leq  2^{p-1}C C_1^2\| f\|_{L^p_s(\RR^n)}^p.
    \end{align*}
\end{proof}

\begin{lemma}\label{eq:Weighted_mollifier}
    Let $f \in L^{p}_s(\RR^n)$, with $p \in [1,\infty)$, $s \in \RR$, and let $\rho_\eps$ be a standard mollifier. For any $\eps<1$, we have
    \begin{equation}\label{eq:Weighted_Jensen}
        \|f * \rho_\eps\|_{L^p_s(\RR^n)}\leq C \|f\|_{L^p_s(\RR^n)},
    \end{equation}
and
    \begin{equation}\label{eq:Weighted_reg_dev}
        \|\partial_j (f*\rho_\eps)\|_{L^p_s(\RR^n)} \leq \frac{C}{\eps} \|f\|_{L^p_s(\RR^n)},  
    \end{equation}
    for some $C>0$ depending only on $\rho$, $p$ and $s$.
\end{lemma}
\begin{proof}
    Let $1_{B_1}$ be the indicator function on $B_1(0)$, then
    \begin{align}
        \|f * \rho_\eps\|_{L^p_s(\RR^n)}^p & = \int_{\RR^n}  \left|(f * \rho_\eps) (x)\right|^{p}  (1_{B_1}+|x|^{-sp-n}) dx.
    \end{align}
    Applying H\"{o}lder's inequality with $p^\prime=1-1/p$, the conjugate exponent of $p$, i.e., Jensen's inequality, we have 
        \begin{align*}
        \left|(f * \rho_\eps) (x)\right|^{p}  & \leq   \| \rho_\eps\|^{p-1}_{L^1(B_\eps)} \int_{B_\eps} \left| f(x-y)\right|^{p} \rho_\eps(y) dy = \int_{B_\eps} \left| f(x-y)\right|^{p} \rho_\eps(y) dy,
    \end{align*}
since $ \|\rho_\eps\|_{L^1(B_\eps)}=1$ and $\rho_\eps$ is nonnegative. This estimate together with Fubini's theorem yields
    \begin{align*}
        \|f * \rho_\eps\|_{L^p_s(\RR^n)}^p & \leq \int_{B_\eps} \int_{\RR^n} \left|f(x-y) \right|^{p}  |x|^{-sp-n} \rho_\eps(y) dx dy .
    \end{align*}
Observe that $\tau_y(f)(x) = f(x-y)$ is a translation by $y \in B_\eps(0)$ and $\eps<1$. Applying \Cref{lemma:translation_weighted} and $\|\rho_\eps\|_{L^1}=1$, we obtain 
\begin{align}
    \|f * \rho_\eps\|_{L^p_s(\RR^n)}^p & \leq \int_{B_\eps} \|\tau_y(f)\|_{L^p_s(\RR^n)}^p \rho_\eps(y) dy \leq C \|f\|_{L^p_s(\RR^n)},
\end{align}
for some constant $C$ depending only on $p$ and $s$.

For the second inequality \eqref{eq:Weighted_reg_dev}, recall that 
\begin{equation}
    \partial_j(f* \rho_\eps)(x) = \int_{\RR^n} f(y) \partial_j^x\rho_\eps(x-y)dy=\int_{B_\eps} f(x-y) \partial_j\rho_\eps(y)dy.
\end{equation}
Arguing again with the Jensen inequality, we obtain
        \begin{align*}
        \left|\partial_j (f * \rho_\eps) (x)\right|^{p}  & \leq \| \partial_j \rho_\eps\|^{p-1}_{L^1(B_\eps)} \int_{B_\eps} \left| f(x-y)\right|^{p} |\partial_j \rho_\eps(y)| dy.
    \end{align*}
Then, repeating the argument with the translation operator, we have the estimate
    \begin{align*}
        \|\partial_j(f * \rho_\eps)\|_{L^p_s(\RR^n)}^p & \leq C \| \partial_j \rho_\eps\|^{p}_{L^1(B_\eps)} \|f\|_{L^p_s(\RR^n)}^p \\
        & \leq \frac{C}{\eps^p} \| \partial_j \rho\|^{p}_{L^1(B_1)} \|f\|_{L^p_s(\RR^n)}^p, 
    \end{align*}
where we have applied $\partial_j (\rho_\eps) = \frac{1}{\eps^{n+1}} \partial_j\rho (x/\eps)$ and $\int_{B_{\eps}}\left|\eps \partial_j \rho_{\eps}(y)\right| dy=\int_{B_1}\left|\partial_j \rho(y)\right| d y$ in the last inequality. 
\end{proof}

\begin{lemma}{\label{lemma:AuxFriedrichsWp_weighted}}
Let $f \in L^{p}_s(\RR^n)$, with $p \in [1,\infty)$, $s \in \RR$, and let $\rho_\eps$ be a standard mollifier. Then, 
$$\eps \|\partial_j(f * \rho_\eps)\|_{L^p_s(\RR^n)} \longrightarrow 0 \text{ as }\eps \to 0.$$
\end{lemma}
\begin{proof}
    First, notice that for any compact set $K \in \RR^n$ a similar version of \Cref{lemma:translation_weighted} shows that
    \begin{equation} \label{eq:translation_openset}
    \|\partial_j (f*\rho_\eps)\|_{L^p_s(\RR^n \setminus K_\eps)} \leq \frac{C}{\eps} \|f\|_{L^p_s(\RR^n\setminus K)},   \end{equation}
    for some constant $C>0$ depending on $p$, $s$ and $\rho$, where $K_\eps$ is an $\eps$-neighbourhood of $K$. Since $f \in L^p_s(\RR^n)$, for any $\delta>0$ there exists a compact set $K \subset \RR^n$ such that $\|f\|_{L^p_s(\RR^n \setminus K)}\leq \delta$. Hence,
\begin{align*}
     \eps \|\partial_j (f*\rho_\eps)\|_{L^p_s(\RR^n)}  &\leq  \eps C_{K_1} \|\partial_j (f*\rho_\eps)\|_{L^p(K_1)} + C \delta\|f\|_{L^p_s(\RR^n \setminus K)},
\end{align*}
where $K_1$ is a $1$-neighbourhood of $K$ and $C_{K_1}$ is a constant depending on $K_1$. The claim follows by observing that $\delta>0$ is arbitrary and that $\eps C_{K_1} \|\partial_j (f*\rho_\eps)\|_{L^p_s(K_1)} \to 0$ as $\eps \to 0$ by the corresponding result for $f \in L^p_\loc(\RR^n)$ \cite[Lemma 3.1]{Hafemann2026}.
\end{proof}

\begin{lemma}\label{lemma:RateConvSobolev_weighted}
Let $a \in W^{1,p}_{s_1}\left(\RR^n\right)$ and $f \in L^q_{s_2}\left(\RR^n\right)$, with $s_1$, $s_2 \in \RR$, and $p, q \in [1, +\infty)$, with $q\geq p^\prime$. Then
\begin{equation}
\|(a * \rho_\eps) (f* \rho_\eps) - (af)* \rho_\eps \|_{L^r_{s_1+s_2-1}} \leq C \eps \|Da\|_{L^p_{s_1-1}} \|f\|_{L^q_{s_2}},    
\end{equation}
where $r$ is defined by $1/r=1/p+1/q$.
\end{lemma}
\begin{proof}
    We follow the argument in \cite[Lemma 3.2]{Hafemann2026}, keeping track of the decay rate. Adding and subtracting the term $a (f*\rho_\eps)$ gives
\begin{align*}
   (a * \rho_\eps) (f* \rho_\eps) - (af)* \rho_\eps & = \big[(a * \rho_\eps - a) (f* \rho_\eps)\big] \\
   & +  \big[a(f* \rho_\eps)-(af)* \rho_\eps\big].
\end{align*}
Applying the H\"{o}lder's inequality and $\|f* \rho_\eps\|_{L^q_{s_2}} \leq C \|f\|_{L^q_{s_2}}$ \eqref{eq:Weighted_Jensen}, we obtain
\begin{align*}
\| (a * \rho_\eps) (f* \rho_\eps) - (af)* \rho_\eps\|_{L^r_{s_1+s_2-1}} & \leq C \|a * \rho_\eps - a\|_{L^p_{s_1-1}} \|f\|_{L^q_{s_2}} \\
& +  \|a(f* \rho_\eps) - (af)* \rho_\eps\|_{L^r_{s_1+s_2-1}}.
\end{align*}
We first estimate the term $a - a*\rho_\eps$. By the definition of convolution,
\begin{align*}
 \big(a-a * \rho_\eps  \big) (x) & = \int_{B_\eps} \big(a(x) - a(x-y)\big) \rho_\eps(y) dy \\
 &  = \eps \int_{B_1} \frac{\big(a(x) - a(x-\eps z)\big)}{\eps}\rho (z) dz.
\end{align*}
If $a$ is a test function, we have
\begin{equation}
\big(a - a * \rho_\eps\big) (x) = \eps \int_{B_1} \int_0^1 \big[Da{(x-t\eps z)} \cdot z\big] \rho (z) dtdz. 
\end{equation}
Integrating over $\mathcal{E}:=\RR^n \setminus B_r(0)$, $r\geq1$ and applying Jensen's inequality gives
\begin{align*}
\int_{\mathcal{E}} \left| (a - a * \rho_\eps) (x) \right|^p|x|^{-(s_1-1)p-n} dx
& \leq \eps^p \int_{B_1} \int_0^1 \|Da(\cdot-t\eps z)\|^p_{L^p_{s_1-1}(\mathcal{E})} \rho(z)  dt dz.
\end{align*}
Adding the contribution from the compact region and applying \Cref{lemma:translation_weighted}, we have
\begin{equation}\label{eq:speed_decay2}
\| a - a * \rho_\eps\|_{L^p_{s_1-1}} \leq C \eps  \|Da\|_{L^p_{s_1-1}}.
\end{equation}
By density of test functions in $W^{1,p}_{s_1}(\RR^n)$, for $p<\infty$, the estimate holds for all $a \in W_{s_1}^{1,p}(\RR^n)$.

It remains to estimate the commutator term $ \big(a (f* \rho_\eps) - (af) * \rho_\eps\big)$. Using the identity
\begin{equation}
\big(a (f* \rho_\eps) - (af) * \rho_\eps\big) (x) = \eps \int_{B_1} \int_0^1 \big[Da{(x-t\eps z)} \cdot z\big] f(x-\eps z)\rho (z) dtdz,     
\end{equation}
and applying the H\"older's inequality together with \Cref{lemma:translation_weighted}, we obtain
\begin{equation}
      \|a(f* \rho_\eps) - (af)* \rho_\eps\|_{L^r_{s_1+s_2-1}} \leq C \eps \|Da\|_{L^p_{s_1-1}} \|f\|_{L^q_{s_2}}.
\end{equation}
Combining these two estimates concludes the proof.
\end{proof}

\begin{proposition}\label{prop:FriedrichsWp_weighted}
For $a \in W^{1,p}_{s_1}\left(\RR^n\right)$ and $f \in L^q_{s_2}\left(\RR^n\right)$, with $s_1$, $s_2 \in \RR$, $p, q \in [1, +\infty)$, with $q\geq p^\prime$, we have that 
\begin{equation}
    \|\left(a * \rho_{\eps}\right)\left(f * \rho_{\eps}\right)-(a f) * \rho_{\eps}\| _{W^{1,r}_{s_1+s_2}} \longrightarrow 0 \text{ as } \eps \to 0,
\end{equation}
where $r$ is defined by $1/r=1/p+1/q$.
\end{proposition}
\begin{proof}
 We mostly follow and use the arguments in \cite[Proposition 3.3]{Hafemann2026}. In this proof, given a function $f$, we denote $f_\eps := f * \rho_\eps$ and $\partial_j f_\eps := \partial_j(f *\rho_\eps)$. The convergence in $L^r_{s_1+s_2}(\RR^n)$ follows directly from \Cref{lemma:RateConvSobolev_weighted}.

Now, we proceed to demonstrate the convergence of first-order derivatives. For any function $f$, it holds that $\partial_j f_\eps := \partial_j(f* \rho_\eps) =  f * \partial_j \rho_\eps$, then
\begin{align*}
\partial_j (a_\eps f_\eps - (af)_\eps)  & = \partial_j ((a_\eps -a ) f_\eps + af_\eps - (af)_\eps)\\
& = \partial_j ((a_\eps -a ) f_\eps) + \partial_j(af_\eps) - \partial_j(af)_{\eps}.
\end{align*}
The first term on the right-hand side can be controlled by the following inequalities
\begin{align*}
\|\partial_j \left((a_\eps -a ) f_\eps\right)\|_{L^r_{s_1+s_2-1}} 
&  \leq \|\partial_j (a_\eps -a ) f_\eps\|_{L^r_{s_1+s_2-1}} + \| (a_\eps -a) \partial_jf_\eps\|_{L^r_{s_1+s_2-1}} \\
& \leq \|\partial_j (a_\eps -a )\|_{L^p_{s_1-1}}  \|f_\eps\|_{L^q_{s_2}} + \|a_\eps -a\|_{L^p_{s_1-1}} \| \partial_j(f_\eps)\|_{L^q_{s_2}}\\
& \leq C\left(\|a_\eps -a\|_{W^{1,p}_{s_1}}  \|f\|_{L^q_{s_2}} +  \eps \| \partial_j(f_\eps)\|_{L^q_{s_2}}\right),
\end{align*}
for some constant $C>0$. Using that $a_\eps \to a$ in $W^{1,p}_{s_1}(\RR^n)$ as $\eps \to 0$ and applying \Cref{lemma:AuxFriedrichsWp_weighted}, we obtain 
\begin{equation}\label{eq:Partial_j_Fried_weighted}
    \|\partial_j ((a_\eps -a ) f_\eps)\|_{L^r_{s_1+s_2-1}} \longrightarrow 0 \text{ as } \eps \to 0. 
\end{equation}
Then, it remains to deal with the last two terms on the right-hand side. Define $A_\eps \colonequals \partial_j(af_\eps) - \partial_j(af)_\eps$. Arguing as in \cite[Proposition 3.3]{Hafemann2026}, but using the weighted H\"{o}lder's inequality instead, we have
\begin{equation}
    \|A_\eps\|_{L^r_{s_1+s_2-1}} \leq C \|Da\|_{L^p_{s_1-1}} \|f\|_{L^q_{s_2}},
\end{equation}
for some constant $C$ independent of $\eps$. 

Since smooth compactly supported functions are also dense in $W^{1,p}_{s_1}$ and $L^q_{s_2}$, for $p,q<\infty$, the same argument in \cite[Proposition 3.3]{Hafemann2026} using the uniform boundedness principle applies here to conclude that $A_\eps \to 0$ in $L^r_{s_1+s_2-1}(\RR^n)$ as $\eps \to 0$ for every $(a, f) \in W^{1,p}_{s_1}(\RR^n) \times L^{q}_{s_2}(\RR^n)$.

Finally, we have
\begin{align*}
\|\partial_j (a_\eps f_\eps - (af)_\eps)\|_{L^r_{s_1+s_2-1}} &\leq     \|\partial_j ((a_\eps -a ) f_\eps)\|_{L^r_{s_1+s_2-1}} + \|\partial_j(af_\eps) - \partial_j(af)_\eps\|_{L^r_{s_1+s_2-1}}\\
& \leq C\left(\|a_\eps -a\|_{W^{1,p}_{s_1}}  \|f\|_{L^q_{s_2}} + \eps \| \partial_j(f_\eps)\|_{L^q_{s_2}}\right) + \| A_\eps\|_{L^r_{s_1+s_2-1}},
\end{align*}
which converges to zero as $\eps \to 0$.
\end{proof}

\begin{corollary}
\label{cor:ConvFriedPlusApproximations_weighted}
    Let $p,q \in [1,+\infty)$, $q \geq p^\prime$ and $s_1,s_2 \in \RR$. Let $a, a_\eps \in  W_{s_1}^{1,p}\left(\RR^n\right)$ and assume that $a_\eps \to a$ in $W_{s_1}^{1,p}(\RR^n)$ as $\eps \to 0$. Additionally, suppose that there exists $C$ such that $\| a - a_\eps\|_{L^p_{s_1-1}} \leq C \eps$. Let $f \in L_{s_2}^q\left(\RR^n\right)$, then 
    \begin{equation}
        a_\eps  (f * \rho_\eps) - (a f)* \rho_\eps \longrightarrow 0 \text{ in } W^{1,r}_{s_1+s_2}(\RR^n) \text{ as } \eps \to 0,
    \end{equation}
    where $r$ is defined by $1/r=1/p+1/q$.
\end{corollary}
\begin{proof}
    The claim follows by considering the following identity
\begin{align}
      a_\eps (f*\rho_\eps) - (af)* \rho_\eps& = \big[a_\eps - (a*\rho_\eps) \big] (f*\rho_\eps) 
  + \big[ (a*\rho_\eps) (f *\rho_\eps) - (af)* \rho_\eps \big].
\end{align}
The second term on the right-hand side converges in $W^{1,r}_{s_1+s_2}(\RR^n)$ by \Cref{prop:FriedrichsWp_weighted}. For the first term, we control its derivatives. Applying H\"older's inequality yields
\begin{align*}
\|\partial_j  \left(\big[a_\eps - (a*\rho_\eps) \big] (f*\rho_\eps)\right)\|_{L^r_{s_1+s_2-1}(\RR^n)} 
& \leq C \|a_\eps - (a*\rho_\eps)\|_{W^{1,p}_{s_1}(\RR^n)} \|f\|_{L^q_{s_2}
(\RR^n)} \\ &\quad + \|a_\eps - (a*\rho_\eps)\|_{L^p_{s_1-1}(\RR^n)}  \|\partial_j(f*\rho_\eps)\|_{L^q_{s_2}(\RR^n)}.    
\end{align*}
A standard computation together with \eqref{eq:speed_decay2} and the hypothesis $\|a-a_\eps\|_{L^p_{s_1-1}}\leq C \eps$ shows that $\|a*\rho_\eps- a_\eps\|_{L^p_{s_1-1}} \leq C \eps$, for some constant $C>0$, and $\|a*\rho_\eps - a_\eps\|_{W^{1,p}_{s_1}(\RR^n)}\to 0$ as $\eps \to 0$. Finally, the result follows by applying \Cref{lemma:AuxFriedrichsWp_weighted}.
\end{proof}

\begin{remark}\label{remark:appendixFried}
    For asymptotically flat metrics, one typically only has  $a-c \in W^{1,p}_{s_1}(\RR^n)$, where $a$ denotes (some of) the components of the metric in an asymptotically flat chart and $c \in \RR$. Using that $c * \rho_\eps = c$, we obtain the identity
    \begin{align*}
    ((a-c)*\rho_\eps)(f*\rho_\eps) - ((a-c)f)*\rho_\eps &= (a*\rho_\eps)(f*\rho_\eps) - c f*\rho_\eps - (af)*\rho_\eps + cf*\rho_\eps\\ 
    &= (a*\rho_\eps)(f*\rho_\eps) - (af)*\rho_\eps.
    \end{align*}
    This shows that the same arguments apply in this setting and immediately yields \Cref{lemma:RateConvSobolev_weighted}, \Cref{prop:FriedrichsWp_weighted}, and \Cref{cor:ConvFriedPlusApproximations_weighted} in the case $a-c \in W^{1,p}_{s_1}$.
\end{remark}

\printbibliography[heading=bibintoc]

\end{document}